\documentclass[11pt]{amsart}
\usepackage{geometry}                
\usepackage{graphicx}
\usepackage{ytableau}
\usepackage{amssymb,upgreek}
\usepackage{amsmath,amsthm,amscd}
\numberwithin{equation}{section}
\usepackage{esvect}

\usepackage{epstopdf}
\usepackage{extarrows}
\usepackage[draft]{pdfcomment}
\usepackage[colorinlistoftodos]{todonotes}
\usepackage{hyperref}
\hypersetup{
pdftex,colorlinks,citecolor=red,linkcolor=blue
}

\newtheorem{thm}{Theorem}[subsection]
\newtheorem{definition}[thm]{Definition}
\newtheorem{prop}[thm]{Proposition}
\newtheorem{lemma}[thm]{Lemma}
\newtheorem{cor}[thm]{Corollary}
\newtheorem{rem}[thm]{Remark}
\newtheorem{conj}[thm]{Conjecture}
\theoremstyle{definition}
\newtheorem{ex}[thm]{Example}

\newcommand{\Sc}{\mathcal{S}}

\newcommand{\Rc}{\mathcal{R}}
\newcommand{\Bc}{\mathcal{B}}
\newcommand{\Hc}{\mathcal{H}}
\newcommand{\Dc}{\mathcal{D}}

\newcommand{\vt}{{\tt{v}}}
\newcommand{\qt}{{\tt{q}}}
\newcommand{\ct}{{\tt{c}}}
\newcommand{\dt}{{\tt{d}}}

\newcommand{\zt}{{\tt{z}}}

\newcommand{\Qt}{{\tt{Q}}}
\newcommand{\Ut}{{\tt{U}}}

\newcommand{\Pbb}{\mathbb{P}}
\newcommand{\Dbb}{\mathbb{D}}
\newcommand{\Ibb}{\mathbb{I}}

\newcommand{\Fr}{\textup{Fr}}

\newcommand{\mbf}{\mathbf}

\begin{document}

\title[Affine Hecke algebras and the conjectures of Hiraga, Ichino and Ikeda ]
{Affine Hecke algebras and the conjectures of Hiraga, Ichino and Ikeda on the Plancherel density}
\author{Eric Opdam}
\address[E.~Opdam]
{Korteweg-de Vries Institute for Mathematics\\Universiteit van Amsterdam\\Science Park 105-107\\ 
1098 XG Amsterdam, The Netherlands}
\email{e.m.opdam@uva.nl}

\date{\today}
\thanks{It is a pleasure to thank Maarten Solleveld for many useful comments. It is also a 
pleasure to acknowledge the excellent comments and suggestions of the 
participants of the workshop ``Representation Theory of Reductive Groups Over
Local Fields and Applications to Automorphic forms'' at the Weizmann institute, and 
of the workshop ``Representation theory of p-adic groups'' at IISER, Pune.}

\begin{abstract} 
Hiraga, Ichino and Ikeda have conjectured an explicit expression for the 
Plancherel density of the group of points of a reductive group defined 
over a local field $F$, in terms of local Langlands parameters. 
In these lectures we shall present a proof of these conjectures for 
Lusztig's class of representations of unipotent reduction if $F$ is $p$-adic 
and $G$ is of adjoint type and splits over an unramified extension of $F$.   
This is based on the author's paper [Spectral transfer morphisms for unipotent affine Hecke algebras, 
Selecta Math. (N.S.) {\bf 22} (2016), no. 4, 2143--2207]. 

More generally for $G$ connected reductive (still assumed to be split over an unramified extension 
of $F$), we shall show that 
the requirement of compatibility with the conjectures of Hiraga, Ichino and Ikeda essentially determines 
the Langlands parameterisation for tempered representations of unipotent reduction. 
We shall show that there exist parameterisations for which the conjectures of 
Hiraga, Ichino and Ikeda hold up to rational constant factors.   
The main technical tool is that of spectral transfer maps between normalised affine
Hecke algebras used in \it{op. cit.}
\end{abstract}
\dedicatory{Dedicated to Joseph Bernstein}
\maketitle
\tableofcontents


\section{Introduction}
Let $F$ be local field of characteristic $0$, let $\Gamma_F := \textup{Gal}(\overline{F} /F)$ 
be the absolute Galois group of $F$, and let $G$ be a the group of points 
of a connected reductive algebraic group defined over $F$. Let $G^\vee$ 
denote the Langlands dual group of $G$ (a complex Lie group with root system dual to that of $G$), 
and let ${}^LG := G^\vee \rtimes \Gamma_F$  
be the Galois form of the L-group of $G$. 
The Langlands group $L_F$ of $F$ is defined to be $W_F$ (the Weil group of $F$) if $F$ is archimedean, 
and $W_F \times \textup{SL}_2(\mathbb{C})$ 
otherwise. Let $\psi$ be a fixed additive character of $F$. 
To a finite dimensional complex representation $V$ of ${}^LG$ one attaches  
epsilon factors $\epsilon(s,V,\psi)$ and $L$-functions $L(s,V)$, 
where $s\in\mathbb{C}$ is a complex variable (see \cite{Tate}). 
A Langlands parameter for $G$ is a homomorphism $\varphi: L_F\to {}^LG$ some 
natural conditions (cf. Section \ref{Sect:LLP}). 
With all this in place, the adjoint $\gamma$-factor of a Langlands parameter $\varphi$ of $G$ 
is defined as
\[
\gamma(s,\textup{Ad} \circ \varphi,\psi):= \frac{\epsilon(s,\textup{Ad} \circ \varphi,\psi) L(1-s,\textup{Ad} \circ \varphi)}
{ L(s,\textup{Ad} \circ \varphi)},
\]
where $\textup{Ad}$ is the adjoint representation of ${}^LG$ on 
$\textup{Lie}(G^\vee)/\textup{Lie}(Z^\vee)$, with $Z^\vee$ the center of $G^\vee$. 
Let $\widehat{G}$ be the space of equivalence classes of irreducible unitary representations
$\pi$ of $G$, equipped with the Fell topology. We will denote by $\Theta_\pi$ 
the distribution character of $\pi$. Then Harish-Chandra's Plancherel formula states 
the existence of a unique positive measure $\nu_{Pl}$ on $\widehat{G}$,  
such that $f(e) = \int_{\widehat{G}}\Theta_\pi(f)d\nu_{Pl}(\pi)$, for all $f \in C_c^\infty(G)$.
The measure $d\nu_{Pl}$ is called the Plancherel measure of $G$.  
If $\pi$ is a discrete series representation of $G$, its formal degree is equal to 
$\textup{fdeg}(\pi) := \nu_{Pl}({\pi})>0$. For more general tempered representations 
$d\nu_{Pl}$ is described in terms of a density function.
Hiraga, Ichino, and Ikeda formulated two conjectures \cite{HII, HIIcor} expressing the 
Plancherel density at a tempered representation $\pi$ in terms of the conjectural enhanced  
Langlands parameter attached to $\pi$.  
For an essentially discrete series representation $\pi_\rho$ in an $L$-packet $\Pi_\varphi(G)$ 
attached to a (discrete) Langlands parameter $\varphi: L_F \to {}^LG$, enhanced with 
a local system $\rho \in \textup{Irr}(\mathcal{S}_\varphi,\chi_G)$ on the $G^\vee$-orbit 
of $\varphi$ (see Section \ref{Sect:LLC} for unexplained notation and more details), 
the conjecture reduces to the
equality
\[
\textup{fdeg}(\pi_\rho) = \frac{\textup{dim}(\rho)}{|\mathcal{S}^\natural_{\varphi}|} |\gamma(0,\textup{Ad} \circ \varphi,\psi)|. 
\]
Hiraga, Ichino, and Ikeda proved the conjecture for $F = \mathbb{R}$ \cite{HII, HIIcor}.  
For $F$ non-archimedean, it has been proved in several cases but not in general
(see Section \ref{Sect:known} for an overview of known results). 
From now on we assume that $F$ is a non-archimedean field and that $G$ 
splits over an unramified extension of $F$. When $G$ is absolutely almost 
simple of adjoint type, the conjecture above and its extension to general tempered 
representations are known to hold for representations of $G$ with unipotent reduction 
from previous works \cite{Re0}, \cite{Re}, \cite{HOH}, \cite{FO}, \cite{Opdl}, \cite{Fe2}, 
\cite{FOS}. The proof in \cite{Opdl} uses the 
Lusztig classification (composed with the Iwahori-Matsumoto involution in order 
to map tempered representations to bounded parameters) as a Langlands parameterisation.
The main goal of the present manuscript is to extend this result to a general connected reductive 
group $G$ (still assumed to be split over an unramified extension of $F$). This is achieved in 
the main result Theorem \ref{thm:main}. 

The irreducible representations with unipotent reduction (a terminology introduced by Moeglin and Waldspurger), 
called ÒunipotentÓ by Lusztig, are the representations of $G$ which admit non-zero invariant vectors by the pro-$p$ 
unipotent radical of a parahoric subgroup of $G$. 
In particular, they are depth-zero representations. They are expected to correspond to unramified Langlands parameters.
From their definition it follows that the category of representations with unipotent reduction of $G$ is Morita equivalent 
to the module category of a direct sum of affine Hecke algebras $\Hc_t(G)$, where $t$ runs over the set of 
equivalence classes of unipotent types of $G$. Let $G^*$ be the quasi-split group in the inner class of $G$, 
and let $\Ibb$ denote an Iwahori subgroup 
of $G^*$. Let $\Hc_\Ibb(G^*)$ be the Iwahori Hecke algebra of $G^*$.

One of the main ingredients of the proof of Theorem \ref{thm:main} is 
the notion of spectral transfer morphism between normalised affine Hecke
algebras \cite{Opds}, \cite{Opdl}, which allows us to construct a bijection between the set
$\widehat{G}^\textup{temp}_{uni}$ of equivalence classes of tempered irreducible representations of $G$ with unipotent 
reduction and the set $\Phi^{\textup{temp}}_{nr}(G)$ of $G^\vee$-conjugacy classes of unramified 
bounded enhanced Langlands parameters for $G$. 
The construction of such a bijection has some 
interest in its own right. A key point is the fact that a spectral transfer morphism $\Psi:\Hc_t(G)\leadsto\Hc_\Ibb(G^*)$ 
from the Hecke algebra $\Hc_t(G)$ of a unipotent type $t$ of $G$ to the Iwahori Hecke algebra $\Hc_\Ibb(G^*)$ 
defines Langlands parameters $\pi\to\varphi_\pi\in \Phi^{\textup{temp}}_{nr}(G)$ for the tempered 
representations covered by $t$ such that the conjectures of 
Hiraga, Ichino and Ikeda hold (up to rational constant factors independent of the cardinality $q$ 
of the residue field of $F$). This is explained in Corollary \ref{cor:STMpar}, which is itself based on the 
Iwahori-spherical case Theorem \ref{thm:cc} and Theorem \ref{thm:Iw}.  Conversely, a Langlands 
parameterisation such that the conjectures of 
Hiraga, Ichino and Ikeda hold (up to rational constant factors independent of the cardinality $q$ 
of the residue field of $F$) and satisfies a certain algebraic condition (see Theorem \ref{thm:main}(a)) 
defines such STMs uniquely.

It is remarkable that for tempered representations of unipotent reduction the 
conditions imposed on a Langlands parameterisation by the conjectures of Hiraga, Ichino and Ikeda
determine it up to twisting by certain diagram automorphism. This can be viewed as a generalisation 
and strengthening of the principle expressed by Mark Reeder \cite{Re} for discrete series $L$-packets, saying that 
``alleged $L$-packets can only be convicted upon circumstantial evidence,  
of which the formal degrees are one piece''. 

\section{The conjecture of Hiraga, Ichino and Ikeda}
Let $F$ be a local field of characteristic $0$, and let $\mbf{G}$ be a connected reductive group defined 
over $F$. The group $G=\mbf{G}(F)$ of $F$-points of $\mbf{G}$ is a separable locally compact topological 
group which is unimodular. Let $\mu_G$ denote a Haar measure on $G$. 
Let $C^*(G)$ be the group $C^*$-algebra of $G$, i.e. the $C^*$-envelope of the 
Banach algebra $L^1(G,\mu_G)$ with respect to convolution. By famous results of Harish-Chandra
\cite{HC1} (if $F$ is archimedean) and by Bernstein \cite{Be} (if $F$ is non-archimedean) we know that 
$C^*(G)$ is liminal, hence of Type I. 
Let $\hat{G}$ denote the space of equivalence classes of 
irreducible unitary representations of $G$, equipped with 
the Fell topology. For each $\pi\in \hat{G}$ 
we choose a representative denoted by $(V_\pi,\pi)$. 
The abstract Plancherel formula for separable locally compact 
unimodular topological groups of Type I asserts that: 
\begin{thm} There exists a unique positive measure $\nu_{Pl}$ (called the Plancherel measure of $G$)   
on $\hat{G}$ such that:  
\begin{equation}
L^2(G,\mu_G)\simeq \int^\oplus_{\pi\in\hat{G}}H(\pi)d\nu_{Pl}(\pi)
\end{equation}  
where $H(\pi):=V_\pi^*\hat\otimes V_\pi$ denotes the algebra of 
Hilbert-Schmidt operators on $V_\pi$.
\end{thm}  
Much of study of harmonic analysis on reductive groups is devoted to making 
the abstract Plancherel formula in this context explicit. 
This is a problem with many different facets, some of which are poorly understood 
or even unsolved even after more than 70 years into the subject.
One part of this is conceptual. The tremendous success of the approach of Langlands
towards harmonic analysis on reductive groups points out that number theory and 
algebraic geometry are inherent parts of this endeavour. 
An explicit Plancherel formula has to reflect the deep number theoretical  
problems which are conspiring in the background. There are formidable technical obstacles 
as well, stemming from the fact that one is forced to deal with representation theory on 
rather general topological vector spaces even if one's goal is the study unitary representations.

Harish-Chandra made deep contributions to our understanding of the 
structure of the explicit Plancherel formula
(\cite{HC2}, \cite{Wal}). 
He discovered that support of the Plancherel measure 
is not all of $\hat{G}$, except if $G$ happens to be built from anisotropic and commutative almost factors.  
In general, the support of the Plancherel measure is the set of so-called irreducible \emph{tempered} 
representations
of $G$. A connected component (Harish-Chandra series) of this set consists of the irreducible summands 
of the representations obtained by unitary parabolic induction of a discrete series representation of 
a Levi subgroup of $G$. There are countably many Harish-Chandra series. 

Assuming the local Langlands correspondence for tempered representations, Hiraga, Ichino and Ikeda 
conjectured an explicit formula 
for the Plancherel measure $\nu_{Pl}$ of $G$. The appeal of these conjectures is that they formulate the 
answers in terms of a natural number theoretical invariant which is associated with an irreducible 
representation $\pi$ in the conjectural local Langlands correspondence, the so-called adjoint gamma 
factor.   
\subsection{The decomposition of the trace}
The regular representation $L^2(G,\mu_G)$ corresponds to the semi-finite positive 
trace $\tau_G(f):=f(e)$ on $C^*(G)$, and in particular $C^*(G)$ is a Hilbert algebra. 
The dense subalgebra $C^\infty_c(G)\subset C^*(G)$ has the special property that 
for all $\pi\in \hat{G}$ and $f\in  C^\infty_c(G)$ the operator $\pi(f)\in\mathcal{B}(V_\pi)$
is of trace class. This defines a distribution $\Theta_\pi$ 
defined by:
\begin{equation}
\Theta_\pi(f):=\textup{Tr}_{V_\pi}(\pi(f))
\end{equation}
This is called Harish-Chandra's distributional character of $\pi$.

The positive trace $\tau_G$ is defined on the dense subalgebra $C^\infty_c(G)\subset C^*(G)$, 
and the Plancherel measure $\nu_{Pl}$ is completely determined by the decomposition 
of $\tau_G$ as a superposition of the distributional characters $\Theta_\pi$ with $\pi\in\hat{G}$:
\begin{cor} The Plancherel measure $\nu_{Pl}$ is the unique positive measure 
on $\hat{G}$ such that for all $f\in C^\infty_c(G)$:
\begin{equation}
f(e)= \int_{\pi\in\hat{G}}\Theta_\pi(f)d\nu_{Pl}(\pi)
\end{equation}  
\end{cor}
Some remarks are in order:
\begin{enumerate}
\item[(1)] The measure $\nu_{Pl}$ depends on the normalization of the 
Haar measure. If we replace $\mu_G$ by $a\mu_G$ (for $a>0$) then 
$\nu_{Pl}$ is replaced by $a^{-1}\nu_{Pl}$.
\item[(2)] The definition of the distributional character $\Theta_\pi$ can be extended 
naturally to the class of admissible representations of $G$. An irreducible admissible 
representation $(V_\pi,\pi)$ is tempered iff the distribution $\Theta_\pi$ is tempered, by which 
we mean that $\Theta_\pi$ extends continuously to the Harish-Chandra Schwartz algebra 
$\mathcal{C}(G)\supset C^\infty_c(G)$ of $G$. In turn this is equivalent to the requirement 
that for every standard parabolic subgroup $P\subset G$, the exponents $\chi\in \textup{Exp}(\pi_P)$ 
of the Jacquet module $(V_{\pi,P},\pi_P)$ satisfy the Casselman conditions 
$\textup{Re}(\chi)\in\overline{{}^+\mathfrak{a}^{G,*}_P}$ (see \cite{Wal}). Here $\overline{{}^+\mathfrak{a}^{G,*}_P}$
denotes the closed convex cone inside the real span $\mathfrak{a}^{G,*}_M$ of the set of $G$-roots 
$\Sigma(A_M,G)$ of the connected split center $A_M$ of 
the standard Levi-factor $M$ of $P$ spanned by the set of roots associated to the unipotent radical of $P$.    
\item[(3)] By a deep result of Harish-Chandra, the support of $d\nu_{Pl}$ is contained in the set $\hat{G}^{\textup{temp}}$ 
of equivalence classes of tempered representations of $G$. 
This was explained more conceptually by Joseph Bernstein \cite{Be1}.
\item[(4)] We call an admissible irreducible representation $(V_\pi,\pi)$ of $G$ a discrete series representation 
if the matrix coefficients of $\pi$ are in $L^2(G,\mu_G)$. Let $\hat{G}^{\textup{disc}}\subset \hat{G}$ denote the subset of 
equivalence classes of discrete series representations. 
It is well known by Casselman's results (see \cite{Wal})
that $\pi\in \hat{G}^{\textup{disc}}$ iff $\textup{Re}(\chi)\in{}^+\mathfrak{a}^{G,*}_P$ for all standard parabolic subgroups 
$P\subset G$ and all $\chi\in \textup{Exp}(\pi_P)$.
The set $\hat{G}^{\textup{disc}}$ is not empty iff the center $Z(G)$ is anisotropic. By a well known characterisation of 
Dixmier we have $\pi\in \hat{G}^{\textup{disc}}$ iff $\nu_{Pl}(\{\pi\})>0$.
\begin{definition}
If $\pi\in \hat{G}^{\textup{disc}}$ then we define the formal degree of $\pi$ as $\textup{fdeg}(\pi):=\nu_{Pl}(\{\pi\})>0$.
\end{definition}
\item[(5)] Let $A\subset G$ be the maximal split component of $Z(G)$. 
Then $G$ has discrete series only if $A$ is trivial.
More generally, we call an 
irreducible representation $\pi$ a discrete series modulo the center if 
$\pi$ is tempered and its matrix coefficients are $L^2(G/A)$.  
If $\pi$ is a discrete series representation modulo the center 
then one can show that there exists a constant 
$\textup{fdeg}(\pi)>0$ such that for all $v,w\in V_\pi$: 
\begin{equation}
\int_{G/A}|(\pi(g)v,w)|^2 d\mu_G(g)=\textup{fdeg}(\pi)^{-1}\Vert v\Vert^2\Vert w\Vert^2
\end{equation}
This generalizes the notion of formal degree for a discrete series representation as in (3).
\end{enumerate}
\subsection{Normalization of Haar measure}
As we have seen above, we need to fix the Haar measure $\mu_G$ in order 
to fix $\nu_{Pl}$ uniquely. The Haar measure depends on the choice of an additive 
character $\psi$ of $F$. The construction is explained in \cite{GG} (also see 
\cite[Section 4]{G}, and the discussion in \cite{HIIcor} of the differences between 
these two constructions). 
The corresponding Haar measure will be denoted 
by $\mu_G^\psi$ if we need to stress the dependence on the choice of $\psi$.
\begin{lemma}
\begin{enumerate}
\item[(a)] Suppose that $F$ is non-archimedean and that 
$\mbf{G}$ is split over an unramified extension. Let $\psi_0$ be an 
additive character with conductor $\mathfrak{p}\subset\mathfrak{o}$. Let $q$ be the 
cardinality of the residue field $\mathfrak{o}/\mathfrak{p}$.
Then for 
any parahoric subgroup $\mathbb{P}\subset G$ with reductive quotient $\overline{\mathbb{P}}$
we have 
\begin{equation}
\textup{Vol}(\mathbb{P},\mu_G^{\psi_0})=q^{-\textup{dim}(\overline{\mathbb{P}})/2}|\overline{\mathbb{P}}|
\end{equation}
(This is the normalization of Haar measure used in \cite{DeRe}.)
\item[(b)] Suppose that $F=\mathbb{R}$, and that $\psi_0(x)=\exp(2\pi\sqrt{-1}x)$.
Assume that $G$ has discrete series representations. By Harish-Chandra's 
criterion this is equivalent to the existence of an anisotropic maximal torus $T\subset G$ of $G$. 
Let $\mu^0_G$ be the Haar measure on $G$ defined by the volume form 
on $\mathfrak{g}=\textup{Lie}(G)$ corresponding to invariant norm $\Vert x\Vert^2=-B(x,\theta(x))$, 
where $x\in\mathfrak{g}$, $\theta$ denotes the Cartan involution, and $B$ a nondegenerate bilinear form 
on $\mathfrak{g}$ as in \cite{HC1}. 
Let $\mu_T^0$ be the Haar measure on $T$ defined similarly. 
We denote by $\Sigma^\vee$ the root system of $\mathfrak{g}^\vee:=Lie(G^\vee)$, 
with $|\Sigma^\vee|=2N$ and $\textup{dim}(T)=l$. 
Then:
\begin{equation}
\mu_G^{\psi_0}=
2^{N}(2\pi)^l\textup{Vol}(T,\mu_T^0)^{-1}\prod_{\alpha^\vee\in\Sigma^\vee_+}(\alpha^\vee,\alpha^\vee)\mu_G^0
\end{equation}
\end{enumerate}
\end{lemma}
\begin{proof}
Part (a) follows from comparing \cite[(1.1) and further]{HII} with \cite{DeRe}.
Part (b) is an easy consequence of the computation in \cite[\S 2]{HII}.
\end{proof}
In the sequel we will use $\mu_G:=\mu^{\psi_0}_G$ as the standard normalization of the Haar measure on $G$.
\subsection{Local Langlands parameters}\label{Sect:LLP}
Let $\Gamma:=\textup{Gal}(\overline{F}/F)$ denote the absolute Galois group of $F$.
Choose a Borel subgroup $B$ and maximal torus $T\subset B$ of $G':=\mbf{G}_{\overline{F}}$, 
and let $\beta(G)=(X^*(T),\Delta,X_*(T),\Delta^\vee)$ be the corresponding based root datum.
Choose a pinning $(G',B,T,\{x_\alpha\}_{\alpha\in\Delta})$, 
which induces a splitting of the exact sequence: 
\begin{equation}
1\to \textup{Int}(G')\to \textup{Aut}(G')\to \textup{Aut}(\beta(G'))\to 1
\end{equation} 
The action of $\Gamma$ on $G(\overline{F})$ 
induces an action of $\Gamma$ on 
$\beta(G')$, and via the above splitting this gives rise to an (algebraic) action 
$a$ of $\Gamma$ on $G'$. There is a unique split $F$-structure on $G'$ which 
fixes the chosen splitting, and clearly this commutes with $a$. Therefore the 
composition of these two actions defines a quasi-split $F$-structure $\mbf{G}^*$ 
on $G'$.  We denote by $G^*=\mbf{G}^*(F)$ the corresponding group of points.

Let $G^\vee$ be the connected complex reductive group with 
$\beta(G^\vee):=\beta(G')^\vee$. Choose a pinning 
$(G^\vee,B^\vee,T^\vee,\{y_{\alpha^\vee}\}_{\alpha^\vee\in\Delta^\vee})$ of $G^\vee$. 
The action of $\Gamma$ on $\beta(G^\vee)$ induced by the action on $\beta(G')$ 
gives rise to an (algebraic) action of $\Gamma$ on $G^\vee$. 
We define ${}^LG:=G^\vee\rtimes\Gamma$, the Langlands dual group of $\mbf{G}$.
\\

The Langlands group $L_F$ of $F$ is defined as follows:
\begin{equation}
L_F=
\begin{cases}
W_F \text{\ if\ $F$\ is\ archimedean,}\\
W_F\times \textup{SL}_2(\mathbb{C})\text{\ else.}
\end{cases}
\end{equation}
Here $W_F$ denotes the Weil group of $F$ (see \cite{Tate}). 
A Langlands parameter is a homomorphism $\varphi:L_F\to{}^LG$ such that 
\begin{enumerate}
\item[(1)] $\varphi$ is continuous.
\item[(2)] In the non-archimedean case, $\varphi|_{\textup{SL}_2(\mathbb{C})}$ is algebraic.
\item[(3)] $\textup{pr}_2\circ\varphi|_{W_F}\to\Gamma$ is the canonical homomorphism $W_F\to \Gamma$. 
\item[(4)] $\varphi(W_F)$ is semisimple.
\item[(5)] If $\textup{Im}(\varphi)$ is contained in the Levi-subgroup of a parabolic 
subgroup $P$ of ${}^LG$ then $P$ is $G$-relevant (in the sense of \cite[\S 8]{Bo}). 
\end{enumerate} 
\begin{definition}
We call $\varphi$ \emph{discrete} if $C_{G^\vee}(\varphi)$ is finite.
We call $\varphi$ \emph{essentially discrete} if $\textup{Im}(\varphi)$ is not contained 
in the Levi subgroup of a proper relevant parabolic subgroup of ${}^LG$.
\end{definition}
\begin{lemma}\label{lem:dlp} 
Discrete parameters exist iff     
the connected center of $G$ is anisotropic and $F$ is non-archimedean, or  
else if $G$ admits an anisotropic maximal $F$-torus.  
In this situation $\varphi$ 
is discrete iff $\varphi$ is essentially discrete.
\end{lemma}
\begin{proof}
It is easy to see that the connected center of $G$ is $F$-anisotropic if and only if the center 
${}^LZ:=Z(G^\vee)^\Gamma$ is finite.  
The group $C:=C_{G^\vee}(\varphi)\subset G^\vee$ is reductive since 
$\varphi(W_F)$ is semisimple and $\varphi({\textup{SL}_2(\mathbb{C})})$
is reductive. Hence $C$ is finite iff $C$ does not contain a nontrivial 
torus. By the above remark this can happen only if the connected center 
of $G$ is anisotropic. In this case \cite[Proposition 3.5, 3.6]{Bo} implies that 
$C$ does not contain a nontrivial torus 
iff $\textup{Im}(\varphi)$ is not contained in any proper Levi subgroup of ${}^LG$.
By definition this is equivalent to saying that  
$\textup{Im}(\varphi)$ is not contained in a Levi subgroup of ${}^LG$
of any proper relevant parabolic subgroup of ${}^LG$. 
In the non-archimedean case we can define a discrete character $\varphi_0$ 
which is trivial on $W_F$ and corresponds to the regular unipotent orbit on 
$\textup{SL}_2(\mathbb{C})$ (the principal parameter). In the case $F=\mathbb{C}$ 
there are no discrete parameters. 
If $F=\mathbb{R}$ and $\varphi$ is discrete then 
$\varphi(\mathbb{C}^\times)$ (with $\mathbb{C}^\times=W_\mathbb{C}\subset W_\mathbb{R}$) 
must contain regular semsimple elements. Thus
$\varphi(\mathbb{C}^\times)$ is contained in a unique maximal torus 
$T^\vee$ of $G^\vee$ which must be $\theta$-stable 
(with $\theta$ the automorphism corresponding to the nontrivial element of 
$\textup{Gal}(\mathbb{C}/\mathbb{R})$). Since $\theta\varphi(z)\theta^{-1}=\varphi(\overline{z})$, 
the discreteness of $\varphi$ implies that $\varphi(\overline{z})=\varphi(z)^{-1}$, and 
thus that $\theta$ restricted to $T^\vee$ is sending $t$ to $t^{-1}$.
This implies that $G$ has an anisotropic $F$-torus (this is clear if $G=G^*$ is 
quasi-split, and it is well known that this condition is independent of the inner form
\cite[Corollary 2.9]{S}). Conversely, when this condition holds it is easy to write 
down discrete parameters. 
\end{proof}
\begin{definition}
We call a Langlands parameter $\varphi$ tempered if $\varphi(W_F)$ is bounded.
\end{definition}
It is not difficult to see that $\varphi$ is tempered if $\varphi$ is discrete.
\begin{definition}
Two Langlands parameters $\varphi,\varphi'$ are called equivalent iff they 
are in the same orbit of $\textup{Int}(G^\vee)$ (acting on ${}^LG$). The set of 
equivalence classes of Langlands parameters of $G$ is denoted by 
$\Phi(G)$ and the subset of equivalence classes of tempered 
Langlands parameters of $G$ is denoted by $\Phi^{\textup{temp}}(G)$.  
\end{definition}
\subsection{$L$-functions and $\epsilon$ factors}
We associate an $L$-function and an $\epsilon$-factor to a representation $V$ 
of $L_F$ in the usual way (see \cite{Tate}). The $L$-function $L(s,V)$ 
is a meromorphic function of a parameter $s\in\mathbb{C}$ which only depends 
on the semisimplification of $V$, and satisfies  
by inductivity (i.e. if $F'\subset F$ then $L(s,V)=L(s,\textup{Ind}_{L_F}^{L_{F'}}(V))$)
and additivity. It is known these properties determine $L(s,V)$ completely if 
$L(s,\chi)$ is known for all characters $\chi$ of $L_F$.

In the archimedean case 
the $L$-functions assigned to characters are as follows:
\begin{enumerate}
\item[(a)] If $F=\mathbb{R}$, $L_F^{ab}\simeq \mathbb{R}^\times$. A character has a unique 
representation of form $\chi(x)=x^{-n}|x|^{s_0}$ with $n\in\{0,1\}$; then 
$L(s,\chi)= \pi^{-(s+s_0)/2}\Gamma((s+s_0)/2)$.
\item[(b)]  If $F=\mathbb{C}$, $L_F^{ab}\simeq \mathbb{C}^\times$. A character has a unique 
representation of form $\chi(z)=\sigma(z)^{-n}\Vert z\Vert^{s_0}$ with $n\in\mathbb{Z}_{\geq 0}$ 
and $\sigma\in\Gamma(\mathbb{C}/\mathbb{R})$; then 
$L(s,\chi)= 2(2\pi)^{-(s+s_0)}\Gamma(s+s_0)$.
\end{enumerate}
For the non-archimedean case, let $\textup{Fr}$ denote the Frobenius automorphism of the   
maximal unramified extension $F_{ur}$ of $F$. We choose once and for all an extension 
of $\textup{Fr}$ to $\overline{F}$, defining an element of $W_F$ which we will also denote by $\textup{Fr}\in W_F$.
Define 
\begin{equation}
\tilde{\textup{Fr}}=(\textup{Fr},\begin{pmatrix} v^{-1}&0\\0&v\end{pmatrix})\in L_F,\ \text{with\ } v=q^{1/2}
\end{equation}
Now we define for a representation $(V,\varphi)$ of $L_F$:
\begin{align*}
L(s,V)&=\textup{det}(1-q^{-s}\varphi(\tilde{\textup{Fr}})|_{V_N^{I_F}})^{-1}\\
&=\prod_{n\geq 0}\textup{det}(1-q^{-s-n/2}\varphi(\textup{Fr})|_{V_n^{I_F}})^{-1}
\end{align*}
where $I_F\subset W_F$ denotes the inertia subgroup $I_F:=\Gamma(\overline{F}/F_{ur})$, 
and $V_N^{I_F}$ the space of highest weight vectors in the 
$(W_F/I_F)\times \textup{SL}_2(\mathbb{C})$-module $V^{I_F}$. In the second line 
we decomposed $V$ as $V\simeq\oplus_{n\geq 0} V_n\otimes \textup{Sym}^n(\mathbb{C}^2)$
for certain representations $V_n$ of $W_F$.
\\

The $\epsilon$-factors depend on the choice of the additive character $\psi$ of $F$. It is known 
that $\epsilon$ is also additive, and inductive for virtual representations of degree $0$. 
In the non-archimedean case we have $\epsilon(s,V,\psi):=\omega(V,\psi)q^{a(V)(1/2-s)}$
where $a(V)$ is the Artin conductor of $V$ \cite[Section 2]{GR}.  
Here $\omega(V,\psi)\in \mathbb{C}^\times$ is independent of $s$.
In the archimeadean case we have $\epsilon(s,\chi,\psi)=c_\psi.(\sqrt{-1})^n$ where
$\chi$ is a character of $W_F$ expressed as above (see the discussion of the $L$-functions
in the archimedean case).  
\\

Given a Langlands parameter $\varphi:L_F\to {}^LG$, we define the 
adjoint $\gamma$ factor of $\varphi$ as follows. Let $\textup{Ad}$ denote the adjoint representation 
of ${}^LG$ on $\textup{Lie}(G^\vee)/\textup{Lie}({}^LZ)$.
\begin{equation}
\gamma(s,\textup{Ad}\circ\varphi,\psi):=\frac{\epsilon(s,\textup{Ad}\circ\varphi,\psi)L(1-s,\textup{Ad}\circ\varphi)}
{L(s,\textup{Ad}\circ\varphi)}
\end{equation}
It is not difficult to show that \cite[Lemma 1.2]{HII}: 
\begin{prop}
If $\varphi$ is tempered then $\gamma(s,\textup{Ad}\circ\varphi,\psi)$ is regular at $s=0$. 
Moreover $\varphi$ is tempered and essentially discrete iff $\gamma(s,\textup{Ad}\circ\varphi,\psi)$ is 
nonzero at $s=0$.
\end{prop}
\subsection{A conjectural tempered local Langlands correspondence}\label{Sect:LLC}
The conjecture on Plancherel densities of Hiraga, Ichino and Ikeda presupposes 
the existence of a local Langlands correspondence for tempered representations. 
A satisfactory formulation of a refined local Langlands conjecture in full detail 
(including certain desired properties of transfer factors) at this level of generality seems not to 
be known. We refer the reader to \cite{Ar1}, \cite{Vogan}, \cite{ABV}, \cite{Ar2}, \cite{HII}, \cite{HS} 
and \cite{Kal} for more background, discussion and overview of known results 
supporting various forms of the conjecture.  

In this section we would like to formulate a more crude version of the local Langlands 
conjecture for tempered representations covering the aspects which are relevant to our goals. 
We mainly follow \cite[Section 3]{Ar2}, \cite[Section 1]{HII}, \cite[Section 9]{HS}.

Put $\Pi(G)$ for the set of admissible irreducible representations of $G$.
The local Langlands conjecture predicts that there exists a partition 
\begin{equation}
\Pi(G)=\sqcup_{[\varphi]}\Pi_\varphi(G)
\end{equation}
where the disjoint union is over the set of equivalence classes $[\varphi]$ of local Langlands 
parameters $\varphi:L_F\to{}^LG$. The sets $\Pi_\varphi(G)$ are called $L$-packets.
Some of the fundamental expected properties of this conjectural partitioning are:
\begin{enumerate}
\item[(i)] $\Pi_\varphi$ is a non-empty finite set. 
\item[(ii)] $\Pi_\varphi$ contains tempered characters iff $\varphi$ is tempered. In this case 
all members of $\Pi_\varphi$ are tempered.
\item[(iii)] $\Pi_\varphi$ contains characters which are discrete modulo center iff $\varphi$ is 
essentially discrete. In this case all members of $\Pi_\varphi$ are discrete modulo center. 
\item[(iv)] Suppose that $F$ is $p$-adic. Then $\Pi_\varphi$ 
contains a character which is generic and supercuspidal iff $\varphi|_{W_F}$ is discrete.
In this case all members of $\Pi_\varphi$ are supercuspidal.
\end{enumerate}

Let us now look in more detail into the conjectural parameterisation of the $L$-packets $\Pi_\varphi$
for $\varphi$ tempered, following \cite{HII}. 

Let $A$ denote the (group of points of) the maximal split torus of $Z(G)$. 
Let $(G_{ad})^\vee=G_{sc}^\vee$ be the simply connected cover of the derived group 
of $G^\vee$.
Let $G^{\vee,\natural}$ denote the dual group of $G/A$, and let 
$G^\vee_{ad}=G^\vee_{sc}/Z(G^\vee_{sc})$ be the dual group of the simply connected 
cover of the derived group of $G$.  

We have homomorphisms $G^\vee_{sc}\xrightarrow{\beta} G^\vee_{ad}
\xleftarrow{\alpha} G^{\vee,\natural}$.
Given a tempered Langlands parameter $\varphi:L_F\to {}^LG$ for $G$, we define
$S_\varphi^\natural:=\{s\in G^{\vee,\natural}\mid \textup{Ad}(s)\circ \varphi=\varphi\}$
and $S_\varphi:=\beta^{-1}\alpha(S_\varphi^\natural)$. Next we 
define  $\mathcal{S}_\varphi^\natural=\pi_0(S_\varphi^\natural)$ and 
$\mathcal{S}_\varphi=\pi_0(S_\varphi)$. 

Recall that $G$ is (the group of $F$-points of) an inner form the quasi-split 
$F$-group $G^*$, which defines 
a class in $H^1(F,\mbf{G}^*_{ad})$. Kottwitz constructed a canonical map 
$H^1(F,\mbf{G}^*_{ad})\to({}^LZ_{sc})^*$
(with ${}^LZ_{sc}=Z(G_{sc}^\vee)^\Gamma$) which is bijective in the $p$-adic case. 
Let $\chi_G\in ({}^LZ_{sc})^*$ be the character 
that corresponds to $G$. We have $\chi_{G^*}=\textup{triv}$.
By \cite[\S 3]{Ar2}, \cite[Lemma 9.1]{HS} we have:
\begin{lemma}
The kernel $\textup{Ker}(\chi_G)$ contains ${}^LZ_{sc}\cap S_\varphi^0$, hence $\chi_G$ 
descends to $\textup{Im}({}^LZ_{sc}\to\mathcal{S}_\varphi)\subset \mathcal{S}_\varphi$
(also called $\chi_G$). 
\end{lemma}
Following \cite{Ar2}, we choose an extension of $\chi_G$ to $\textup{Im}(Z_{sc}\to\mathcal{S}_\varphi)$
(also denoted by $\chi_G$) such that $\chi_{G^*}=\textup{triv}$ and define: 
$\Pi(\mathcal{S}_\varphi,\chi_G)=\{\rho\in\textup{Irr}(\mathcal{S}_\varphi)
\mid \rho|_{\textup{Im}(Z_{sc}\to\mathcal{S}_\varphi)}=n\chi_G\}$. 

From \cite[Section 3]{Ar2}, \cite{HII} and \cite[Section 9]{HS} we 
distill the following crude form of the local Langlands conjecture: 
\begin{conj}\label{conj:LC}
There exists a bijection $\rho\to\pi_\rho$ between  $\Pi(\mathcal{S}_\varphi,\chi_G)$ and $\Pi_\varphi(G)$ 
such that for all tempered local Langlands parameters $\varphi$ for $G$, 
\begin{equation}
\Theta_\varphi:=\sum_{\rho\in\Pi(\mathcal{S}_\varphi,\chi_G)}\textup{dim}(\rho)\Theta_{\pi_\rho}
\end{equation}
is a stable character of $G$. Here $\Theta_{\pi_\rho}$ denotes the  
distributional character 
of $G$ corresponding to the tempered irreducible representation $\pi_\rho$.  
Any stable linear combination of characters from 
$\Pi(\mathcal{S}_\varphi,\chi_G)$ is a multiple of $\Theta_\varphi$. 
\end{conj}
\begin{definition}
We define $\tilde{\Phi}^{\textup{temp}}(G)=\{(\varphi,\rho)\mid \varphi\in \Phi^{\textup{temp}}(G) \text{\ and\ } 
\rho\in \Pi(\mathcal{S}_\varphi,\chi_G)\}$. 
Suppose that for all $\varphi\in \Phi^{\textup{temp}}(G)$ 
a parameterisation $\rho\to\pi_\rho$ of $\Pi_\varphi(G)$ as in Conjecture \ref{conj:LC} exists. 
The corresponding bijection $\hat{G}^{\textup{temp}}\to \tilde{\Phi}^{\textup{temp}}(G)$, 
$\pi\to (\varphi_\pi,\rho_\pi)$ such that for each $\varphi\in \Phi^{\textup{temp}}(G)$, 
$\Pi_\varphi(G)=\{\pi\in \tilde{\Phi}^{\textup{temp}}(G)\mid \varphi_\pi=\varphi\}$ and such that 
the bijection $\Pi_\varphi(G)\to  \Pi(\mathcal{S}_\varphi,\chi_G)$, $\pi\to \rho_\pi$ is the inverse 
of the bijection $\rho\to \pi_\rho$ in Conjecture \ref{conj:LC}, is called 
an enhancement of the Langlands parameterisation $\pi\to \varphi_\pi$.
\end{definition}
\subsection{The conjectures of Hiraga, Ichino and Ikeda}
We now have everything in place in order to formulate the conjectures
of Hiraga, Ichino and Ikeda \cite{HII}. Suppose that we have given an enhanced 
Langlands parameterisation $\hat{G}^{\textup{temp}}\to \tilde{\Phi}^{\textup{temp}}(G)$. 
\begin{conj}[Conjecture 1.4 of \cite{HII}]\label{conj:1}
Let $\varphi: L_F\to {}^LG$ be a discrete Langlands parameter for $G$, let 
$\rho\in\Pi(\mathcal{S}_\varphi,\chi_G)$ and let $\pi_\rho\in\Pi_\varphi$ be the tempered 
essentially discrete series representation corresponding to $(\varphi,\rho)$.
Then
\begin{equation}
\textup{fdeg}(\pi_\rho)=
\frac{\textup{dim}(\rho)}{|\mathcal{S}^\natural_{\varphi}|}|\gamma(0,\textup{Ad}\circ\varphi,\psi)|
\end{equation}
\end{conj}
For general tempered representations \cite{HII} formulate a conjecture expressing the 
Plancherel density. This amplification is based on Harish-Chandra's Plancherel Theorem 
(\cite{HC2}, \cite{Wal}) and  
Langlands' conjecture on the Plancherel 
measure \cite[Appendix II]{L}, \cite{Sha}. 
\begin{conj}[Conjecture 1.5 of \cite{HII}]\label{conj:2}
Let $P=MN\subset G$ be a semi-standard $F$ parabolic subgroup. Let $\mathcal{O}$ be an 
orbit of tempered essentially discrete series characters of $M$. Let $d\pi$ denote the Haar measure 
on $\mathcal{O}$, normalised as in \cite[pages 239 and 302]{Wal}. For $\pi\in\mathcal{O}$ we put 
\begin{equation}
d\nu(\pi)=
\frac{\textup{dim}(\rho)}{|\mathcal{S}^\natural_{\varphi_M}|}|\gamma(0,r_M\circ\varphi,\psi)|d\pi
\end{equation} 
where $r_M$ denotes the adjoint representation of ${}^LM$ on 
$\textup{Lie}(G^\vee)/\textup{Lie}({}^LZ_M)$. Then the Plancherel density 
at $\textup{Ind}_P^G(\pi)$ is $c_Md\nu(\pi)$ for some explicit 
constants $c_M\in\mathbb{R}_+$ 
independent of $F$ and $\mathcal{O}$.
\end{conj}
\subsection{Known results and further comments}\label{Sect:known}
Conjecture \ref{conj:1} is reduced to the case of generic tempered representations 
by Shahidi's paper \cite{Sha}, if one knows the stability of $\Theta_\varphi$ in Conjecture 
\ref{conj:LC}.

The Conjecture \ref{conj:1} is known for $F=\mathbb{R}$ \cite[Section 3]{HII}.
For $F$ non-archimedean Conjecture \ref{conj:1} is known in the following cases:
\begin{enumerate}
\item[(a)] $G$ an inner form of $\textup{GL}_n$ (\cite{SZ}, \cite{Z}, \cite{HII}). 
\item[(b)] $G$ an inner form of $\textup{SL}_n$ (\cite{HS}, \cite{HII}).
\item[(c)] $G$ arbitrary, $\pi$ the Steinberg representation, $\varphi$ the 
principal parameter (due essentially to Borel, \cite{Bo1}, \cite{HII}).
\item[(d)] $G$ split exceptional of adjoint type, $\pi$ discrete series of unipotent reduction 
(due to Reeder, \cite{Re}).
\item[(e)] $G$ arbitrary,  $\pi$ depth $0$ supercuspidal representation (tame regular semisimple case) 
(DeBacker and Reeder \cite{DeRe}, \cite{HII}).
\item[(f)] Ichino, Lapid and Mao proved Conjecture \ref{conj:1} for odd orthogonal groups.
\item[(g)] Beuzart-Plessis proved Conjecture \ref{conj:1} for unitary groups. 
\item[(h)] For supercuspidal representations of unipotent reduction of connected semisimple 
$p$-adic groups which split over an unramified extension \cite[Theorem 1.3]{FOS}.
\end{enumerate}

The main result we will discuss in these lectures is and extension to general connected reductive $G$
of the following result: \footnote{The Lusztig parameterisation should be twisted by 
the Iwahori-Matsumoto involution in order to map tempered representations 
to bounded parameters (cf. \cite[Text below Theorem 2]{AMS}). Here and elsewhere 
we will tacitly assume this modification.}
\begin{thm}[\cite{Re}, \cite{Opds}, \cite{Opdl}, \cite{Fe2}]\label{thm:main0}
Let $G$ be absolutely almost simple of adjoint type over 
a non-archimedean field $F$ such that $G$ splits 
over an unramified extension of $F$. Then Conjectures \ref{conj:1}, 
\ref{conj:2} hold for representations of unipotent reduction, when we 
use Lusztig's classification \cite{Lusztig-unirep}, \cite{Lusztig-unirep2} 
as a Langlands parameterisation.
\end{thm}
The proof of Theorem \ref{thm:main0} and its extension is based on two techniques 
for affine Hecke algebras:
\begin{enumerate}
\item[(1)] Spectral transfer maps between Hecke algebras \cite{Opds}, \cite{Opdl}, \cite{FO}, \cite{FOS}.
We use these tools to deal with the $q$-rational factors of the formal degree.
\item[(2)] Dirac induction for affine Hecke algebras \cite{COT}, \cite{CO}.  
This tool is useful to determine the precise rational constant factors of the formal degree.
\end{enumerate}
In fact, by the theory of types and Theorem \ref{thm:main} I expect that these techniques may reduce 
the general case of Conjectures \ref{conj:1}  and \ref{conj:2} 
to the case of generic supercuspidal representations. 

The  ``converse results'' Theorem \ref{thm:cuspidalcase}, Theorem \ref{thm:main} 
are interesting in their own right, and are closely related to the theory of spectral transfer maps 
between normalised affine Hecke algebras \cite{Opdl}. 
\section{The Plancherel formula for affine Hecke algebras}
Let $F$ be a nonarchimedean local field with residue field of cardinality $q$ from here onwards.
\subsection{The Bernstein center}
Let $\mathcal{C}(G)$ be the abelian category of smooth representations of $G$, and let 
$\Pi(G)$ denote the space of classes of irreducible objects of $\mathcal{C}(G)$. Let 
$\mathcal{B}(G)$ be the set of inertial equivalence classes of cuspidal pairs $(M,\tau)$ 
(with $M$ a Levi subgroup, and $\tau$ a supercuspidal representation of $M$). For 
$s\in\mathcal{B}(G)$, let $\Omega_s$ be the corresponding set of cuspidal pairs
in the class $s$, modulo $G$-conjugacy (an affine variety). 
We have a central character map 
\begin{equation}\label{eq:bercc}
\textup{cc}:\Pi(G)\to\sqcup_{s\in\mathcal{B}(G)}\Omega_s
\end{equation} 
such that 
$\textup{cc}(\pi)=(M,\tau)$ if $\pi$ is a subquotient of $i_P^G(\tau)$, where $P=MU$ 
is a parabolic subgroup with Levi factor $M$.
Let $\mathcal{O}(\Omega_s)$ be the ring of regular 
functions on $\Omega_s$, and put $\mathfrak{z}(G)=\prod_{s\in\mathcal{B}(G)}\mathcal{O}(\Omega_s)$.
We put $Z_\mathcal{B}(G)=\textup{End}(\textup{Id}_{\mathcal{C}(G)})$, the Bernstein center.
As is well known, we can interpret $Z_\mathcal{B}(G)$ as the set of $G$-invariant distributions 
$z$ on $G$ such that $z\mathcal{H}(G)\subset \mathcal{H}(G)$ and 
$\mathcal{H}(G)z\subset \mathcal{H}(G)$, where $\mathcal{H}(G)$ denotes the Hecke algebra 
of $G$. The famous theorem of Bernstein and Deligne states:
\begin{thm}[\cite{BeDe}, Theorem 2.13]\label{thm:BeDe} 
There is a unique algebra isomorphism (the Fourier transform)
\begin{align}
Z_\mathcal{B}(G)&\to \mathfrak{z}(G)\\
z&\to \hat{z}
\end{align} 
characterised by the property that for all $\pi\in \Pi(G)$, one has 
$z_\pi=\hat{z}({\textup{cc}(\pi)})\textup{Id}_{V_\pi}$. 
\end{thm}
As an immediate consequence one obtains the Bernstein decomposition of $\mathcal{C}(G)$:
\begin{cor}
We have a family of orthogonal idempotents $e_s\in Z_\mathcal{B}(G)$ 
(with $s\in \mathcal{B}(G)$) such that $\widehat{e_s}$ is the characteristic function of $\Omega_s$.  
We have corresponding decompositions
\begin{equation}
\mathcal{C}(G)=\prod_{s\in\mathcal{B}(G)}\mathcal{C}(G)_s,\ \ \Pi(G)=\sqcup_{s\in\mathcal{B}(G)}\Pi(G)_s
\end{equation} 
where $\Pi(G)_s$ is the set of irreducible objects of $\mathcal{C}(G)_s$. 
\end{cor}
\subsection{Types, Hecke algebras and Plancherel measure}
A \emph{type} is a pair $t=(J,\rho)$ such that:
\begin{enumerate}
\item[(1)] $J\subset G$ is a compact open subgroup of $G$.
\item[(2)] $\rho$ is a finite dimensional irreducible representation of $J$.
\item[(3)] Let $e_t\in\mathcal{H}(G)$ be the idempotent given by 
\begin{equation}
e_t(x)=\begin{cases}\frac{\textup{Tr}(\rho(x^{-1}))}{\textup{Vol}(J)}&\text{\ if\ }x\in J \\
0&\text{\ else.\ }\end{cases}
\end{equation} 
Let $\mathcal{C}^t(G)$ be the full subcategory of $\mathcal{C}(G)$ consisting of representations 
$(\pi,V_\pi)$ such that $\mathcal{H}(G)(e_tV_\pi)=V_\pi$, and let $\mathcal{H}_t=e_t\mathcal{H}(G)e_t$. 
Then the functor $m_t:\mathcal{C}^t(G)\to \mathcal{H}_t-\textup{mod}$ given by $V_\pi\to e_t(V_\pi)$
is an equivalence of categories.
\end{enumerate} 
\begin{thm}[Bushnell and Kutzko \cite{BK}] Let $t$ be a type. Then 
$\mathcal{C}^t(G)=\prod_{s\in\mathcal{B}^t(G)}\mathcal{C}(G)_s$ where 
$\mathcal{B}^t(G)$ is a finite set.
\end{thm}
As a consequence of Theorem \ref{thm:BeDe} we see:
\begin{cor}\label{cor:beta} 
Let $t$ be a type, and let $Z(\mathcal{H}_t)$ denote the center of 
$\mathcal{H}_t$. There exists a unique isomorphism: 
\begin{equation}
\beta^t:\Omega^t:=\sqcup_{s\in \mathcal{B}^t(G)}\Omega_s\to \textup{Spec}(Z(\mathcal{H}_t))
\end{equation}
such that for all $\pi\in \Pi(G)^t:=\sqcup_{s\in \mathcal{B}^t(G)}\Pi(G)_s$ we have:  
$\textup{cc}^t(m_t(\pi))=\beta^t(cc(\pi))$. Here $\textup{cc}^t$ denotes the central character 
map of the algebra $\mathcal{H}_t$.
\end{cor}
\begin{thm}[Yu, J-L Kim] If $p$ is sufficiently large then 
for all $s\in \mathcal{B}(G)$ we can find a type $t$ 
such that $\mathcal{C}^t(G)=\mathcal{C}(G)_s$.
\end{thm}
\begin{ex}\label{ex:type}
\begin{enumerate}
\item[(1)] The archetypical example is that of the Borel component: Let 
$\mathbb{B}\subset G$ be the Iwahori subgroup, then $(\mathbb{B},\textup{triv})$ 
is a type. More generally, if $\mathbb{B}_m$ is the $m$-th filtration subgroup of 
$\mathbb{B}$ in the Moy-Prasad filtration, then $(\mathbb{B}_m,\textup{triv})$ is 
a type.
\item[(2)] (Moy-Prasad, Morris, Lusztig) Let $t=(\mathbb{P},\sigma)$ with 
$\mathbb{P}\subset G$ a parahoric subgroup, and $\sigma$ a cuspidal 
unipotent representation of the reductive quotient $\overline{\mathbb{P}}(\mathbb{F}_q)$.
We refer to such $t$ as a ``unipotent type'', and to the objects in 
the associated categories $\mathcal{C}^t(G)$ as ``representations of unipotent reduction''.
\item[(3)] Let $x\in B(G)$ be a point in the building of $G$, and let $r\geq 0$.
Let $G_{x,r,+}$ denote the corresponding Moy-Prasad subgroups. 
Then $(G_{x,r,+},\textup{triv})$ is a type (Bestvina-Savin). 
\end{enumerate}
\end{ex}
The Hecke algebra $\mathcal{H}_t=e_t\mathcal{H}(G)e_t$  of a type $t$ inherits a $*$ 
(an anti-linear anti-involution) and trace $\tau$ from $\mathcal{H}(G)$, defined by 
\begin{enumerate}
\item[(1)] $f^*(g):=\overline{f(g^{-1})}$ for $f\in \mathcal{H}_t$.
\item[(2)] $\tau(f):=f(1)$ (observe that for the unit $e_t\in\mathcal{H}_t$ we have 
$\tau(e_t)=\frac{\textup{dim}(\rho)}{\textup{Vol}(J)}$).
\end{enumerate}
The Hermitian form $(x,y):=\tau(x^*y)$ is positive definite, and defines a 
Hilbert space completion $L^2(\mathcal{H}_t)$ of $\mathcal{H}_t$. This turns 
$\mathcal{H}_t$ into a Hilbert algebra with trace $\tau$. It is well known that the 
irreducible representations of $\mathcal{H}_t$ are finite dimensional, hence 
this Hilbert algebra has a type I $C^*$-algebra envelop. As a consequence of 
Dixmier's central decomposition theorem for Type I $C^*$-algebras we conclude:
\begin{cor} Let $\hat{\mathcal{H}_t}=
\{[\pi]\mid \pi \text{\ is\ an\ irrducible\ $*$-unitary\ }\mathcal{H}_t\text{-mod}\}$.
There exists a unique positive measure $\nu_{\mathcal{H}_t}$ on $\hat{\mathcal{H}_t}$ 
such that 
\begin{equation}
\tau=\int_{\pi\in\hat{\mathcal{H}_t}}\chi_\pi d\nu_{\mathcal{H}_t}(\pi)
\end{equation}
The support $\hat{\mathcal{H}_t}^{\textup{temp}}:=\textup{Supp}(\nu_{\mathcal{H}_t})\subset \hat{\mathcal{H}_t}$ 
is called the tempered dual of $\mathcal{H}_t$. 
\end{cor}  
\begin{thm}[\cite{BHK}]\label{thm:BHK}
The functor $m_t$ defines a homeomorphism 
$\hat{m}_t^{\textup{temp}}:\hat{G}^{t,{\textup{temp}}}:=\Pi(G)^t\cap \hat{G}^{\textup{temp}}
\to \hat{\mathcal{H}_t}^{\textup{temp}}$ such that 
$(\hat{m}_t^{\textup{temp}})_*(\nu_{Pl}|_{\hat{G}^{t,{\textup{temp}}}})=\nu_{\mathcal{H}_t}$. 
\end{thm}
Therefore we can compute $\nu_{Pl}$ by computing $\nu_{\mathcal{H}_t}$ for the Hecke algebras 
$\mathcal{H}_t$ of a collection of types $t$ such that the open closed sets 
$\hat{G}^{t,{\textup{temp}}}$ cover $\hat{G}^{\textup{temp}}$. In this sense the measures 
$\nu_{\mathcal{H}_t}$ are the building blocks of the Plancherel measure of $G$.
\subsection{Affine Hecke algebras as Hilbert algebras}
The algebras $\mathcal{H}_t$ are slight generalisations of affine Hecke algebras.
In Lusztig's case of unipotent types (\cite{Lusztig-unirep}, \cite{Lusztig-unirep2}; 
also see \cite{Mo}, \cite{MP1}, \cite{MP2}) the associated ``unipotent Hecke algebras'' are precisely 
(extended) affine Hecke algebras. 

Let us therefore review the theory of affine Hecke algebras and the spectral decomposition
of the corresponding Hilbert algebra.

Let $W=W^a\rtimes \Omega$ be an extended affine Weyl group. By this we mean that 
we have an affine Coxeter group $W^a$ with set of simple reflections $S$ say, and 
a group of \emph{special automorphisms} $\Omega\subset\textup{Out}_S(W^a,S)$. 
(A diagram automorphism $\omega\in\textup{Out}_S(W^a,S)$ is called special 
if its restriction to the canonical normal subgroup $Q\subset W^a$ consisting of the 
elements with a finite conjugacy class, equals the restriction to $Q$ of an inner 
automorphism of $W^a$.) 

Let $\Lambda=\mathbb{Z}[v_s^{\pm 1}\mid s\in S;\ v_s=v_{s'}\text{\ if\ }s\sim _W s']$ (the $v_s$
are commuting indeterminates), 
and put $\Qt=\{\vt\in\textup{Hom}_{alg}(\Lambda,\mathbb{C}^\times)\mid 
\vt_s=\vt(v_s)\in\mathbb{R}_+\forall s\in S\}\simeq\mathbb{R}_+^N$ where $N$ denotes the number 
of conjugacy classes of affine reflections in $W$. We have a length function $l$ on $W^a$ relative to 
the set $S$ of simple reflections, which we extend to $W$ by giving elements of $\Omega$ 
length $0$. 
\begin{definition}[Coxeter presentation of the Hecke algebra]
Let $\Hc_\Lambda$ be the free $\Lambda$-algebra with basis $\{N_w\}_{w\in W}$ 
such that
\begin{enumerate}
\item[(i)] $N_uN_v=N_{uv}$ for $u,v\in W$ such that $l(uv)=l(u)+l(v)$.
\item[(ii)] For all $s\in S$ we have $(N_s-v_s)(N_s+v_s^{-1})=0$.
\end{enumerate}
\end{definition}
Given $\vt\in\Qt$ we define $\Hc_\vt=\Hc_\Lambda\otimes\mathbb{C}_\vt$.
Let $d\in \Lambda$ be positive on $\Qt$ (or on some subset of $\Qt$ which 
contains $\vt$).  Then we define a Hilbert algebra structure on $\Hc_\vt$ by 
defining a $*$-operator and a \emph{positive trace} $\tau$ as we did before with $\mathcal{H}(G)$:
\begin{enumerate}
\item[(i)] $\tau(N_w)=\delta_{w,e}d(\vt)$. 
\item[(ii)] $N_w^*=N_{w^{-1}}$.
\end{enumerate}
We recall that the positivity of $\tau$ means that the Hermitian form 
$(x,y)=\tau(x^*y)$ on $\Hc_\vt$ is positive definite. This elementary fact is crucial in all that follows.

Next we would like to express the Hilbert algebra stucture in terms of the Bernstein presentation 
of $\Hc_\vt$.
Let $X\subset W$ be the canonical normal subgroup of $W$ consisting of the elements 
which have finitely many conjugates. This is the translation subgroup of $W$, and we 
can choose a splitting of $W/X\simeq W_0$ by choosing a special point $0\in C$, where 
$C\subset V=\mathbb{R}\otimes X$ denotes the alcove. 

The length $l(w)$ of $w\in W$ can be interpreted more geometrically as the number of affine 
reflection hyperplanes of $W^a$ separating the fundamental alcove $C\subset V=\mathbb{R}\otimes Q$
and $w(C)$. In particular we have $l(x)=2\rho(x)$ for $x\in X^+$, the dominant cone in $X$.
By the defining relations of $\Hc_\Lambda$ this implies that the elements $N_x$ with  
$x\in X^+$ form a commutative monoid of invertible elements of $\Hc_\Lambda$. Bernstein 
and Zelevinski turned this into a very important alternative presentation of $\Hc_\Lambda$, 
(see \cite{Lus3} for further background).
\begin{lemma}
There exists a unique homomorphism $X\ni x\to\theta_x\in\Hc_\Lambda$ such that 
for $x\in X^+$ one has $\theta_x=N_x$. Let $A\subset \Hc_\Lambda$ be the commutative 
subalgebra $\Lambda[\theta_x\mid x\in X]\subset \Hc_\Lambda$ generated by the $\theta_x$.
Let $\Hc_0=\Hc_\Lambda(W_0,S_0)\subset \Hc_\Lambda$ be the finite type Hecke subalgebra 
associated to the isotropy group $(W_0,S_0)$ of the chosen special point $0\in C$. 
Then the multiplication map $A\otimes\Hc_0\to \Hc_\Lambda$ is an isomorphism of 
$(A,\Hc_0)$-bimodules, and $\Hc_0\otimes A\to \Hc_\Lambda$ is an isomorphism of 
$(\Hc_0,A)$-bimodules, and the algebra structure of $\Hc_\Lambda$ is determined by the 
Bernstein relation:
\begin{equation}
\theta_x N_s-N_s \theta_{s(x)}=\left((v_s-v_s^{-1})+
(v_{s^\prime}-v_{s^\prime}^{-1})\theta_{-\alpha}\right)
\frac{\theta_x-\theta_{s(x)}}{1-\theta_{-2\alpha}}
\end{equation}
where $s=s_{\alpha^\vee}\in S_0$ for some simple root $\alpha_0$, 
and $s'\in S$ is such that $s'\sim_W s_{\alpha^\vee+1}$. 
\end{lemma}  
\begin{rem}
\begin{enumerate}
\item[(a)] Let $\Rc=(\Delta_0,X,\Delta^\vee_0,Y)$ the based root datum associated 
with $W=W_0\rtimes X$, where $S_0$ corresponds to $\Delta_0$. For $\alpha\in \Delta_0$ 
one can put $q_\alpha^+=v_sv_{s'}$ and $q_\alpha^-=v_s/v_{s'}$. The parameters 
$q_\alpha^{\pm}$ will be more convenient than the $v_s$ in the spectral theory of $\Hc_v$.  
\item[(b)] If $v_{s'}\not=v_s$ then $\alpha^\vee\in 2Y$. In irreducible cases this 
happens only if $\Rc$ is of type $C_n^{(1)}$. That means that $\Delta_0$ is the 
basis $\{e_1-e_2,\dots, e_{n-1}-e_n, e_n\}$ of the irreducible root system of type $B_n$, 
and $X=\mathbb{Z}^n$ is the root lattice of this root system. The affine Dynkin diagram 
of $W$ is the untwisted affine extension of the Dynkin diagram of $\Delta_0^\vee$  
of type $C_n$.
Note that $q_\alpha^-=1$ unless we are in the $C_n^{(1)}$ case, and $\alpha=e_n$. 
\end{enumerate}
\end{rem}
\begin{cor}
The center $Z:=Z(\Hc_\Lambda)$ is equal to $A^{W_0}$, which is 
naturally isomorphic to the ring 
$\Lambda[\theta_x\mid x\in X]^{W_0}$. Let $T$ denote the algebraic torus with 
character lattice $\mathbb{Z}^{S/\sim}\times X$, viewed as a split torus over $\Lambda$ 
via $v_s^n\to (ne_s,0)$. Then $Z(\Hc_\Lambda)=\mathbb{C}[W_0\backslash\backslash T]$.
\end{cor}
\begin{definition}
We have $T=T_0\times\textup{Spec}(\Lambda)$ where $T_0$ is algebraic torus 
over $\mathbb{C}$ with character lattice $X$. If $\tt{v}\in\Qt$ then we will write 
$T^\vt$ for the fibre of $T$ above $\tt{v}$, i.e. the  spectrum of $\mathbb{C}_{\tt{v}}[X]$ where 
$\mathbb{C}_{\tt{v}}$ denotes the residue field of $\Lambda$ at $\tt{v}$.
\end{definition}
\subsection{A formula for the trace of an affine Hecke algebra}
We will now write the trace $\tau$ of $\Hc_\vt$ (for some $\vt\in \Qt$) in terms of the Bernstein 
presentation of $\Hc_\vt$. First we introduce ``intertwining elements'' 
$R_s\in\Hc_\Lambda$ for every $s=s_\alpha\in S_0$
by:
\begin{equation}
R_s=v_s\left((1-\theta_{-2\alpha})N_s-((v_s-v_s^{-1})+(v_{s'}-v_{s'}^{-1})\theta_{-\alpha}) \right)
\end{equation}
These elements satisfy:
\begin{enumerate}
\item[(i)] For all $x\in X$ and $s=s_\alpha\in S_0$ we have: $R_s\theta_x=\theta_{s(x)}R_s$.
\item[(ii)] $R_s^2=v_s^2\Delta_{2\alpha}\Delta_{-2\alpha}c_\alpha c_{-\alpha}$, where 
$\Delta_{\pm 2\alpha}=(1-\theta_{\pm 2\alpha})$ and where 
the rational functions $c_\alpha$ are the famous \emph{Harish Chandra $c$-functions} in the present
context: $c_\alpha=({\Delta_{-2\alpha}})^{-1}
{(1+\theta_{-\alpha}/q_\alpha^-)(1-\theta_{-\alpha}/q_\alpha^+)}$.
\end{enumerate}
\begin{cor}
We have $\Hc_\Lambda\otimes_\mathcal{Z} Z'\simeq A'\# W_0$, where 
$Z'$ is the localization of $Z$ on the open subset of $W_0\backslash T$ which is the intersection 
of the open subsets $U_\alpha\subset T$ ($\alpha\in \Sigma_0$) where $c_\alpha$ is an 
invertible regular function (the complement of the union of the hyperplanes of the form 
$\alpha(t)=\pm q_\alpha^{\pm}$ and $\alpha(t)=\pm 1$). 
\end{cor}
\begin{definition}
Define 
\begin{equation}
\mu(t)=\frac{d}{q(w_0)}\frac{1}{\prod_{\alpha\in\Sigma_0^+}c_\alpha(t)c_{-\alpha}(t)}
=\frac{1}{q(w_0)}\frac{d}{c(t)c(t^{-1})}, 
\end{equation}
the $\mu$-function of $\Hc_\Lambda$ with normalising factor $d\in\Lambda$.
\end{definition}
\begin{thm}[\cite{Opd0}]\label{thm:conv}
Let $t\in T^\vt$ where $\vt$ is such that $q_\alpha^\pm>0$ for all $\alpha$.
Via the nondegenerate symmetric bilinear form $\langle x,y\rangle:=\tau(xy)$ we view $\Hc_\vt$ as 
a subspace of  $\Hc^*_\vt$, and equip $\Hc_\vt^*$ with the weak topology. 
Then $\mathcal{E}_t:=\sum_{x\in X}t(-x)\theta_x\in\Hc^*_\vt$ is convergent 
in $\Hc_\vt^*$ if for all $\alpha\in\Delta_0$, we have $|\alpha(t)|<\textup{min}\{(q_\alpha^-)^{-1},(q_\alpha^+)^{-1}\}$.
Moreover, 
\begin{equation}
\mathcal{E}_t=\frac{E_t}{q(w_0)\Delta(t)}\mu(t)
\end{equation}
where $T^\vt\ni t\to E_t\in \Hc^*_\vt$ is a certain regular family of matrix coefficients of 
minimal principal series at $t$ such that $E_t(1)=q(w_0)\Delta(t)$ and such that 
for all $a,b\in A$, $E_t(ahb)=a(t)b(t)E_t(h)$.  
\end{thm}
\begin{cor}[\cite{Opd0}]\label{cor:dis}
We have the following disintegration of $\tau$ on $\Hc_\vt$:
\begin{equation}\label{eq:dis}
\tau=\int_{t_0T^\vt_u}\frac{E_t}{q(w_0)\Delta(t)}\mu(t)dt
\end{equation}
where $T^\vt_u$ denotes the compact form of $T^\vt$, and $t_0\in T^\vt_v$ is a real base point such that 
the inequality $|\alpha(t_0)|<\textup{min}\{(q_\alpha^-)^{-1},(q_\alpha^+)^{-1}\}$ holds. 
\end{cor}
\begin{proof}
Immediate from Theorem \ref{thm:conv} by the Fourier inversion formula 
on $T^\vt_u$. 
\end{proof}
\subsection{Spectral decomposition of $\tau$}
The disintegration of $\tau$ given in Corollary \ref{cor:dis} is not yet a spectral decomposition 
because the matrix coefficients $E_t$ are neither tempered on $t_0T^\vt_u$, nor tracial (i.e. they do not 
vanish on commutators). To arrive at the spectral decomposition of $\tau$ several steps of refinement 
are necessary. The first step uses ``residue distributions" for integrals in the form (\ref{eq:dis}), 
and symmetrizing the resulting distributions on $T$ over $W_0$. This leads to a decomposition of the 
form \cite{Opd3}:
\begin{equation}\label{eq:dectau}
\tau=\int_{W_0t\in W_0\backslash T^{\vt}}\chi_{W_0t}d\nu(W_0t)
\end{equation}    
where 
\begin{enumerate}
\item[(i)] $\nu$ denotes the spectral measure of the decomposition of $\tau|_{Z_\vt}$, 
where $Z_\vt\subset \Hc_\vt$ denotes the center. We remark that $Z_\vt$ is invariant 
for $*$, and the restriction of $*$ and $\tau$ to $Z_\vt$ equips it with the structure of a 
commutative Hilbert algebra. 
\item[(ii)] The support of $\nu$ is denoted by 
$W_0\backslash T^{\vt,\textup{temp}}$. For each $W_0t$ in $W_0\backslash T^{\vt,\textup{temp}}$, 
$\chi_{W_0t}$ is a tempered positive trace of $\Hc_\vt$, with central character $W_0t$.
\item[(iii)] We have $T^{\vt,\textup{temp}}=\cup_{L\text{\ residual coset}}L^{\textup{temp}}$, where a 
coset $L\subset T^\vt$ of a subtorus is called \emph{residual} if 
\begin{equation}
\#\{\alpha\in\Sigma_0\mid \alpha|_{L}=\pm q_\alpha^\pm\}-
\#\{\alpha\in\Sigma_0\mid \alpha|_{L}=\pm 1\}=\textup{codim}(L)
\end{equation}
Furthermore if $L\subset T$ a residual coset, we define its tempered part $L^{\textup{temp}}\subset L$ 
as follows. Let $\Sigma_L\subset \Sigma_0$ be the parabolic subsystem of the roots which are 
constant on $L$. Let $T^L\subset T^\vt$ be the identity component of the simultaneous kernel 
of the $\alpha\in \Sigma_L$. Let $T_L\subset T^\vt$ be the subtorus associated with the 
subspace $\mathbb{C}\Sigma_L^\vee\subset \mathfrak{t}=\textup{Lie}(T^\vt)$. If 
$r_L\in T_L\cap L$  then $L=r_LT^L$. Now put $L^{\textup{temp}}=r_LT^L_u$ (this does not depend on 
the choice $r_L\in T_L\cap L$).
\end{enumerate}
The computation of the spectral decomposition of $\tau$ as trace on $\Hc^\vt$ now reduces to 
the problem of computing the measure $\nu$ explicitly, and for each $W_0t\in W_0\backslash T^{\vt,\textup{temp}}$, 
decomposing $\chi_{W_0t}$ as a (positive) superposition of irreducible tempered characters 
(with central character $W_0t$).   

In the special case of the discrete series of $\Hc^\vt$ we see that these correspond to the 
$W_0$-orbits of \emph{residual points}. This case is the basic building block for the spectral decomposition:
\begin{thm}[\cite{Opd3}] \label{thm:ds}
\begin{enumerate}
\item[(i)] An orbit $W_0r\subset W_0\backslash T^\vt$ is the central character 
of a discrete series representation $\pi$ of $\Hc^\vt$ if and only if $r$ is a residual point (a residual 
coset of dimension $0$). 
\item[(ii)] If $W_0r\subset T^\vt$ is an orbit of residual points then $\nu(\{W_0r\})=c\mu^{\{r\}}(r):=
c\frac{d(\vt)}{q(w_0)}m_r(\vt)$ (the residue of $\mu$ at $W_0r$), where 
$c\in\mathbb{Q}^\times$ and where the regularisation $\mu^{\{r\}}$ of $\mu$ at $r$ is defined by: 
\begin{equation}
\mu^{\{r\}}=\frac{d(\vt)}{q(w_0)}
{\prod_{\alpha\in \Sigma_0}}'\frac{(1-\alpha^{-2})}{(1+\alpha^{-1}/q_\alpha^-)(1-\alpha^{-1}/q_\alpha^+)}
=\frac{d(\vt)}{q(w_0)}m_r(\vt)
\end{equation}
where the symbol $\prod'$ means that all irreducible factors of the numerator and the 
denominator which become identically $0$ upon evaluation at $r\in T^\vt$ are omitted. 
\item[(iii)] We have $\chi_{W_0r}=\sum_{\delta\text{\ ds}, \textup{cc}(\delta)=W_0r}d_{\Hc,\delta}(\vt)\chi_\delta$ 
where $d_{\Hc,\delta}(\vt)>0$.
\item[(iv)] (Scaling invariance.) 
Define $\vt(\epsilon)$ by $\vt_s(\epsilon)=\vt_s^\epsilon$ for $\epsilon\in\mathbb{R}_+$. Every orbit 
of residual points $W_0r\in T^\vt$ has a unique extension to a real analytic $\epsilon$-family of orbits of residual 
points $W_0\tilde{r}$ such that $W_0\tilde{r}(\vt(1))=W_0r$. A discrete series character $\delta$ with 
$\textup{cc}(\delta)=W_0r$ has a unique extension to a continuous $\epsilon$-family of discrete series 
characters $\tilde{\delta}$ of $\Hc_{\vt(\epsilon)}$, and we have 
$\textup{cc}(\tilde{\delta}(\epsilon))=W_0\tilde{r}(\vt(\epsilon))$ for all $\epsilon>0$. This yields for all $\epsilon>0$ 
a canonical bijection between $\{\delta\text{\ ds\ of\ }\Hc_{\vt}\mid\textup{cc}(\delta)=W_0r\}$ and 
$\{\delta'\text{\ ds\ of\ }\Hc_{\vt(\epsilon)}\mid\textup{cc}(\delta')=W_0r(\epsilon)\}$. 
Then $d_{\Hc,\tilde{\delta}(\epsilon)}(\vt(\epsilon))$ is independent of $\epsilon>0$.
\item[(v)] If $\delta$ is a discrete series representation of $\Hc_\vt$ then 
$\textup{fdeg}(\delta)=d_{\Hc,\delta}\frac{d(\vt)}{q(w_0)} |m_r(\vt)|$
where $d_{\Hc,\delta}\in\mathbb{R}_+$ as defined in (iii) and $m_r$ as defined in (ii)
 (we will see below that in fact $d_{\Hc,\delta}\in\mathbb{Q}_+$). 
\end{enumerate}
\end{thm} 
The non-discrete contributions to the spectral decomposition of $\tau$ can be obtained from the discrete 
summend of the spectral decomposition of the corresponding traces of ``parabolic subalgebras'' by a process 
of unitary parabolic induction, analogous to Harish-Chandra's theory of the Plancherel decomposition for 
reductive groups. More precisely we have \cite{Opd3}:
\begin{thm}\label{thm:parb}
Let $L=r_LT^L\subset T^\vt$ be a residual coset, such that $\Sigma_L\subset \Sigma_0$ is a standard 
parabolic subsystem. Let $\Hc^{P(L)}\subset \Hc$ be the subalgebra corresponding 
to the based root datum $P(L):=(\Delta_L,X,\Delta^\vee_L,Y)$. Let $X_L$ be the character lattice of $T_L$, and $Y_L\subset Y$ 
its dual. Let $\Hc_{P(L)}$ be the Hecke algebra with the semisimple based root datum 
$P_{ss}(L)=(\Delta_L,X_L,\Delta^\vee_L,Y_L)$
and Hecke parameters $q_\alpha^\pm$ obtained by restriction from $\Delta_0$ to $\Delta_L$. Given $t^L\in T^L$ there 
exists a homomorphism $\phi_{t^L}:\Hc^{P(L)}\to \Hc_{P(L)}$ defined by (in the Bernstein presentation) $N_w\to N_w$ for all 
$w\in W_L$, 
and $\theta_x\to x(t^L)\theta_{\textup{pr}(x)}$ where $\textup{pr}:X\to X_L$ is the canonical projection.
\begin{enumerate}
\item[(i)] Let $t=r_Lt^L\in L^{\textup{temp}}$ be a generic point, i.e. $c_\alpha$ defines a regular and invertible 
germ at $t$ for all $\alpha\in\Sigma_0\backslash\Sigma_L$. Then in (\ref{eq:dectau}) we have:  
\begin{equation}
|W_0/W_L|\chi_{W_0t}=
\sum_{\delta\in\hat{\Hc}_{P(L),ds}, \textup{cc}(\delta)=W_0r_L}\textup{Ind}_{\Hc^L}^\Hc(\chi_{L,W_Lr_L}\circ\phi_{t^L})
\end{equation}
 \item[(ii)] We have $\nu=\sum_{L\text{\ residual coset} }\nu_L$ where $\nu_L$ is the push forward of a measure on 
 $L^{\textup{temp}}$ given by $d\nu_L(t)=\mu^L(t)dt=c_L\mu_{\Hc_L}^{\{r_L\}}(r_L)m^L(t)dt^L$ with $c_L\in\mathbb{Q}_+$  and 
 \begin{equation}
m^L(t)=\frac{1}{q(w^L)}\prod_{\alpha\in \Sigma^L_+:=
\Sigma_{0,+}\backslash\Sigma_{L,+}}\frac{1}{c_\alpha(t)c_\alpha(t^{-1})}
\end{equation}
\end{enumerate}
\end{thm}
\begin{cor}
The explicit spectral decomposition of the trace $\tau$ on $\Hc_\vt$  
reduces, by Theorem \ref{thm:parb}, to the classification of the discrete series of the standard parabolic semisimple 
subquotient Hecke algebras $\Hc_{P}$ of $\Hc$, and the computation of their formal degree.
This reduces further to the classification of the set of $W_0$-orbits of residual points $\{W_0r\}$, and 
of the finite set of discrete series characters $\delta$ with $\textup{cc}(\delta)=W_0r$
(which has been carried out in \cite{OpdSol}),  
and the computation of the constants $d_{\Hc_P,\delta}\in\mathbb{R}_+$ (carried out in \cite{CO}).
\end{cor}
\begin{rem}
In the context of Hecke algebras of a type of a reductive 
group over a non-archimedean local field $F$, changing the  
base field to an unramified extension of $F$ of degree $n$ 
corresponds to the scaling $\vt\to \vt(n)$. This explains 
the importance of the scaling invariance properties.   
\end{rem}
\subsection{Residual cosets and their properties}
Given the importance of residual subspaces for the spectral decomposition of $\tau$ we 
discuss some of their properties \cite{Opd3,Opd4,Opds} and \cite{OpdSol}.
\begin{thm} 
For every coset of a subtorus $L\subset T^\vt$ we have 
\begin{equation}
\#\{\alpha\in\Sigma_0\mid \alpha|_{L}=\pm q_\alpha^\pm\}-
\#\{\alpha\in\Sigma_0\mid \alpha|_{L}=\pm 1\}\leq \textup{codim}(L)
\end{equation}
\end{thm}
\begin{prop}\label{prop:par}
Let $L\subset T$ be a residual coset of the subtorus $T^L$, with 
$\frak{t}^L:=\textup{Lie}(T^L)=\Sigma_L^\perp$. Let $T_L\subset T$
be the subtorus such that $\frak{t}_L:=\textup{Lie}(T_L)=\mathbb{R}\Sigma_L^\vee$. 
Then $T=T_LT^L$ and $T^L\cap T_L=K_L$ is a finite abelian group. Moreover  
$L\cap T_L=K_Lr_L$ for a residual point $r_L\in T_L$ of $\Hc_L$.
\end{prop}
 \begin{cor}
 There exists only finitely many residual cosets $L\subset T^\vt$.
 \end{cor}
 \begin{cor}
 If $L,M\subset T$ are residual cosets then $L^{\textup{temp}}\subset M^{\textup{temp}}$ if 
 and only if $L=M$. 
 \end{cor}
 \begin{cor}
 The measure $\nu_L$ defined in \ref{thm:parb} is smooth on $L^{\textup{temp}}$.
 \end{cor}
 \begin{thm}
 For $L\subset T$ residual, put 
 $S_L:=W_0\backslash W_0L^{\textup{temp}}\subset \textup{supp}(\nu)=S$ with 
 $S=S(\mathcal{H}):=\textup{cc}(\hat{\Hc}^{\textup{temp}}_\vt)\subset W_0\backslash T^\vt$. The sets 
 $S_L\subset S$ are the connected components of $S$.  
 \end{thm}
 \begin{cor}
 For every connected component $C\subset \hat{\Hc}_\vt^{\textup{temp}}$ there 
 exists a residual coset $L$ such that $\textup{cc}(C)=S_L$. Then 
 $(\nu_\Hc)|_C=c_Ci_*(\textup{cc}|_{C^{reg}})^*(\nu_L)$ for some constant $c_C>0$.
 Here $C^{reg}\subset C$ is open and dense, and has a unique structure of a 
 smooth manifold such that $ \textup{cc}|_{C^{reg}}$ is a smooth finite covering map
 to $S_L^{reg}$, where $S_L^{reg}\subset S_L$ is the image of the largest stratum 
 with respect to the action of $W_0$ on $W_0L^{\textup{temp}}$, and $i:C^{reg}\to C$ denotes 
 the embedding.
 \end{cor}
 \begin{thm}[\cite{OpdSol}]
 Let $L^\vt\subset T^\vt$ be a residual coset. Then there exists a residual coset 
 $L_\Lambda\subset T=T_\Lambda$ defined over $\Lambda$ such that 
 $L^\vt=\{\vt\}\times_{\textup{Spec}(\Lambda)} L_\Lambda$. 
 \end{thm}
 \begin{thm}[\cite{OpdSol}, \cite{CO}]
 Let $Q(\Lambda)$ denote the quotient field of $\Lambda$, and let 
 $r_\Lambda\in T_\Lambda$ be a residual point defined over $\Lambda$. 
 Then $m_{r_\Lambda}\in Q(\Lambda)$ is regular on $\Qt$, and if $\vt\in\Qt$ 
 then $r_\Lambda^\vt\in T^\vt$ is residual 
 if and only if $m_{r_\Lambda}(\vt)\not=0$. The set 
$\{\vt\in\Qt\mid m_{r_\Lambda}(\vt)=0\}$ is a union of finitely many hyperplanes 
 in the real vector group $\Qt$.
 \end{thm}
 \subsection{Deformation of discrete series and the computation of $d_{\Hc,\delta}$}
 In this section we review a 
 deformation principle in the parameters 
 $\vt\in\Qt$ for the discrete series characters  
 $\delta$ of an affine Hecke algebra $\Hc_\vt$. 
 As we will see, this leads to an important tool to compute the rational constants 
 $d_{\Hc_\vt,\delta}$ for unequal parameter Hecke algebras $\Hc_\vt$ 
(which are abundant among unipotent Hecke algebras).
 \begin{thm}[\cite{OpdSol}]\label{thm:def}
 Assume that $\Hc$ is a semisimple affine Hecke algebra.
Let $\Sc=\{\sum_{w\in W}c_wN_w\mid \forall N\in\mathbb{N}: W\ni w\to l(w)^N|c_w|\text{\ is\ a\ bounded\ function}\}$
denote the Schwartz completion of $\Hc$.
 Note that this nuclear Frechet space is independent of the Hecke parameter $v$.  
 Suppose that $\delta$ is a discrete series character of $\Hc_\vt$ with $cc(\delta)=W_0r$.
 There exists an (analytic) open neighbourhood $\Ut\subset \Qt$ of $\vt$, 
 a \emph{unique} $W_0$-orbit $W_0r_\Lambda$ 
 of residual points defined over $\Lambda$, and a \emph{unique} continuous family 
 $\Ut\ni\vt'\to\tilde{\delta}^{\vt'}\in\Sc$ of discrete series characters $\tilde{\delta}^{\vt'}$ of $\Hc_{\vt'}$ 
 such that $\tilde{\delta}^\vt=\delta$.  
 \end{thm}
We now review a remarkable rationality property of the formal degree of 
$\tilde{\delta}$ as in Theorem \ref{thm:def}. According to \cite{CO} there exists an orthonormal 
set $\Bc_{gm}$ of elliptic virtual characters of the affine Weyl group $W=X\rtimes W_0$ 
(cf. \cite[2A1]{CO}) 
(with respect to the Euler-Poincar\'e pairing) which naturally parameterises the generic 
families of discrete series characters. More precisely, to each $b\in\Bc_{gm}$ we assign
(based in part on the previous subsection): 
\begin{enumerate}
\item[(i)] An orbit of generic residual points $W_0r_b\in W_0\backslash T(\Lambda)$.
\item[(ii)] The open set $\Qt^{reg}_b=\{\vt\in\Qt\mid m_b(\vt):=m_{W_0r_b}(\vt)\not=0\}$
(the complement of finitely many hyperplanes of $\Qt$).
\item[(iii)] A continuous family $\Qt^{reg}_b\ni\vt\to \textup{Ind}_D(b,\vt)$ of virtual 
characters of $\Hc_\vt$ (the ``Dirac induction'' of $b$, cf. \cite[2B4]{CO}) with the properties that:
\begin{enumerate}
\item[(a)] 
For each $b\in\Bc_{gm}$ there exists a locally constant function 
$\epsilon(b,\vt)\in\{\pm 1\}$ on $\Qt^{reg}_b$ such that 
for all $\vt\in \Qt^{reg}_b$, $\epsilon(b,\vt)\textup{Ind}_D(b,\vt)$ is an irreducible discrete series 
character of $\Hc_\vt$. 
\item[(b)] We have $cc(\textup{Ind}_D(b,\vt))=W_0r_b$, for all $\vt\in \Qt^{reg}_b$. 
\item[(c)] For all $\vt\in \Qt$, the set of irreducible discrete series character of $\Hc_\vt$ is 
equal to 
$\{\epsilon(b,\vt)\textup{Ind}_D(b,\vt)\mid b\in\Bc_{gm}^\vt\}$, where 
$\Bc_{gm}^\vt:=\{b\in Bc_{gm}\mid \vt\in \Qt^{reg}_b\}$. 
\item[(d)] Let $[\pi]$ denote the elliptic class of a virtual character $\pi$ of 
$W$. Then we have $[\lim_{\epsilon\to 0}\textup{Ind}_D(b,\vt^\epsilon)]=b$.   
\end{enumerate}
\end{enumerate} 
\begin{rem}
By the Langlands classification, virtual elliptic characters of $\Hc_\vt$ can be written as 
linear combinations of \emph{tempered} characters of $\Hc$. Hence it makes sense to 
view $\textup{fdeg}$ as a linear function on the space of virtual elliptic characters of $\Hc_\vt$. 
\end{rem}
\begin{thm}[\cite{CO}]\label{thm:rat} Assume that $\tau$ is normalised by $\tau(1)=1$.
For all $b\in \Bc_{gm}$ there exist a constant $d_b\in\mathbb{Q}_+$ such that 
for all $\vt\in \Qt^{reg}_b$:  $\textup{fdeg}(\textup{Ind}_D(b,\vt))=d_bm_b(\vt)$. 
\end{thm}
Hence for each $b\in \Bc_{gm}$ the function $\textup{fdeg}(\textup{Ind}_D(b,\vt))$ 
is a rational function of $\vt$ which is regular on $\Qt$. This rationality is remarkable, because 
the family of characters $\Qt_b^{reg}\ni \vt\to \textup{Ind}_D(b,\vt)$ does not extend 
continuously to $\Qt$. The rationality is very powerful to compute the rational constants 
$d_{\Hc_\vt,\delta}$ for the discrete series characters $\delta$ of $\Hc_\vt$ in the unequal parameters 
cases, because we see that it is enough to compute the single constant $d_b$ in each generic family.
These ``generic constants'' $d_b$ are known for all irreducible root data and are quite simple (for 
example, for the classical Hecke algebras of the type $\textup{C}_n[m_-,m_+](q^\beta)$ 
we have $d_b=1$ for all $b$, (cf. \cite{CO}, \cite{CK}, \cite{Opdl})).
\subsection{Central characters and Langlands parameters} The orbits $S_L=W_0\backslash W_0L^{\textup{temp}}$ of 
tempered residual cosets for affine Hecke algebras can be viewed as ``parameter deformations
of unramified Langlands parameters''. This is a crucial point in order to be able to cast the results on 
spectral decompositions of traces of affine Hecke algebras as discussed above in terms of adjoint gamma 
factors. 

The basic result is the following (\cite{KLDL}, \cite{HOH}, \cite{Opd3}, \cite{Opdl}):
\begin{thm}\label{thm:cc}
Consider the special case of the Iwahori Hecke algebra $\Hc_{\mathbb{I},\vt}$ of (the group 
of points of) an 
unramified connected reductive group $G$ over $F$. By \cite[Lemma 6.5, Proposition 6.7]{Bo}
we have $\textup{Spec}(Z(\Hc_{\mathbb{I},\vt}))\simeq (G^\vee\theta)_{ss}=W_{\mathbb{I},0}\backslash T_{\mathbb{I},\vt}$ 
where 
$W_{\mathbb{I},0}:=W^\theta$ is the $F$-Weyl group of $G$, and where the torus $T_{\mathbb{I},\vt}$ 
can be identified with the quotient $T_{\mathbb{I},\vt}=T^\vee/(1-\theta)T^\vee$ of the maximal 
torus $T^\vee$ of $G^\vee$.  Let 
\begin{equation*}
S_{\mathbb{I},\vt}:=S(\mathcal{H}_{\mathbb{I},\vt})=\bigsqcup_{L\subset T_{\mathbb{I},\vt}\text{\ residual}}S_L\subset 
W_{\mathbb{I},0}\backslash T_{\mathbb{I},\vt}=(G^\vee\theta)_{ss}
\end{equation*}
denote the central support of the tempered spectrum of 
$\Hc_{\mathbb{I},\vt}$, and let $\Phi^{\textup{temp}}_{nr}(G)$ denote the set of 
equivalence classes of unramified tempered Langlands parameters for $G$.
The map 
\begin{align}
\gamma^\mathbb{I}:\Phi^{\textup{temp}}_{nr}(G)&\to S_{\mathbb{I},\vt}\subset 
W_{\mathbb{I},0}\backslash T_{\mathbb{I},\vt}\\
[\varphi]&\to W_{\mathbb{I},0}\overline{\varphi(\tilde{\textup{Fr}})}
\end{align}
is a bijection. 
\end{thm}
Using this fact it is not difficult to translate the results on the spectral decomposition of $\tau$
in this special case using adjoint $\gamma$-factors, a remark that essentially goes back to 
\cite{HII}. In fact, using the work of Reeder \cite{Ree4} and results from \cite{CO} one can 
deduce: 
\begin{thm}[\cite{HII}, \cite{HOH}, \cite{Opd3}, \cite{KLDL}, \cite{Re}, \cite{Ree4}, \cite{CO}]\label{thm:Iw} 
Suppose $\mbf{G}$ is unramified over $F$. There exists an enhanced Langlands parameterisation 
of the tempered Iwahori-spherical representations of the packets $\Pi_\varphi(G)$ such that 
the conjectures \ref{conj:1} and \ref{conj:2} hold true for Iwahori-spherical representations.
\end{thm}
\begin{proof} Based on the results of \cite{Re} this was shown in \cite[\S 3.4]{HII} for Iwahori 
spherical discrete series of a group $G$ which is split of adjoint type. 
The results of \cite{Re} have been extended to Iwahori-spherical representations of general 
semisimple unramified groups $G$ in \cite[Proposition 4.9]{CO} using \cite{Ree4}, 
\cite{Opd3}, \cite{HOH}, and Theorem \ref{thm:BHK}. 
Applying the same proof as in \cite[\S 3.4]{HII} shows the result for the discrete series in 
the general unramified semisimple case. 
By Theorems \ref{thm:ds}, \ref{thm:parb} this proves the required results 
for Iwahori spherical tempered representations of an arbitrary unramified group $G$.
\end{proof}
\section{Lusztig's representations  of unipotent reduction and spectral transfer maps. Main result.}
Let $\mbf{G}$ be a connected reductive group over $F$ which is split over an unramified 
extension of $F$. Our main theorem is a slight sharpening and extension of the main 
result of \cite{Opdl}.  

In the formulation of the main result, the action of the group $X_{wur}(G)$ of weakly 
unramified characters of $G$ plays an important role. 
Let $F_{ur}/F$ be  
an unramified extension of $F$ over which $G$ splits, and let 
$\textup{Fr}\in\textup{Gal}(F_{ur}/F)$ denote the geometric Frobenius map.
We also denote by $\textup{Fr}$ the corresponding automorphism on 
$G(F_{ur})$, and by $\textup{Fr}^*$ an inner twist of $\textup{Fr}$ which 
defines an $F$-quasi-split structure (denoted by $\mbf{G}^*$, with group of 
points $G^*$) on $\mbf{G}$. 
We denote by $\theta$ the 
action of $\textup{Fr}$ on $G^\vee$, to that ${}^LG=G^\vee\rtimes\langle\theta\rangle$.
Let $\Omega=\textup{Hom}(Z(G^\vee),\mathbb{C}^\times)$. 
A complex character $\chi$ of $G$ is called \emph{weakly unramified} if $\chi$ is trivial on the kernel 
$G_1$ of the Kottwitz homomorphism $w_G: G\to 
\textup{Hom}(Z(G^\vee),\mathbb{C}^\times)^\theta=\Omega^\theta$ (cf.~\cite{Ko,HR}). 
We denote by $X_{wur}(G)=(\Omega^\theta)^*=
Z(G^\vee)/(1-\theta)Z(G^\vee)$ the 
diagonalizable group of weakly unramified characters.  This is the group of characters 
$\pi_\alpha$ of $G$ attached to the set of unramified Langlands parameters $H^1_{nr}(W_F,Z(G^\vee))$ 
as constructed in \cite[10.2]{Bo}.
We note that if $\Ibb_1\subset \mbf{G}(F_{ur})_1$ 
is a $\textup{Fr}$-stable Iwahori subgroup (such exist by \cite[1.10.3]{Tits}), then there is a canonical isomorphism 
$\Omega=N_{\mbf{G}(F_{ur})}(\Ibb_1)/\Ibb_1$. Put $\Ibb=\Ibb_1^{\Fr}$, then $\Omega^\theta
=N_{G}(\Ibb)/\Ibb$.
\\

Denote by $X_{wur}^{\textup{temp}}(G)\subset X_{wur}(G)$ the subgroup of tempered 
weakly unramified characters.
Tensoring by (tempered) weakly unramified characters 
defines a natural action of $X_{wur}^{\textup{temp}}(G)$ on the set of (tempered) irreducible characters of 
$G$ of unipotent reduction which is Plancherel density preserving. 
There is also a natural action of $X_{wur}(G)$ on $\Phi_{nr}(G)$ 
as follows. If $\omega\in\Omega^*=Z(G^\vee)$ represents a weakly unramified 
character $[\omega]\in X_{wur}(G)=(\Omega^\theta)^*=\Omega^*/(1-\theta)\Omega^*$ and  
$\varphi:\mathbb{Z}\times\textup{SL}_2(\mathbb{C})\to{}^LG$ represents a class 
$[\varphi]\in\Phi_{nr}(G)$, then we define $[\omega].[\varphi]=[\varphi']$
where $\varphi'(\textup{Fr})=\omega\varphi(\textup{Fr})$ while $\varphi'$ 
coincides with $\varphi$ on $\textup{SL}_2(\mathbb{C})$. One easily verifies that 
$\varphi'$ defines an unramified Langlands parameter of $G$, that $[\varphi']$
is independent of the lift $\omega$ of $[\omega]$, and that this defines by restriction an
action of $X_{wur}^{\textup{temp}}(G)$ on $\Phi^{\textup{temp}}_{nr}(G)$ which preserves the adjoint 
$\gamma$-factors. The following lemma is obvious:
\begin{lemma}\label{lem:gamequiv}
The bijection $\gamma^\mathbb{I}$ of Theorem \ref{thm:cc} is equivariant with respect to the action of 
$ X_{wur}^{\textup{temp}}(G)$ on $\Phi^{\textup{temp}}_{nr}(G)$ and on the set 
$S_{\mathbb{I},\vt}$ of central characters of the tempered irreducible $\Hc_{\mathbb{I},\vt}(G^*)$-modules 
corresponding to the Iwahori-spherical tempered irreducible representations of 
unipotent reduction of $G^*$.
\end{lemma}

The \emph{unramified} characters of $G$ are the complex (quasi-)characters
of $G$  which are 
trivial on the intersection $G^1$ of the kernels of the compositions 
$\textup{Val}_F\circ\chi$ with $\chi\in X^*(G)$. It is clear that $G/G^1$ can 
be identified with a sublattice of the dual of $X^*(G)$, thus $X_{nr}(G)$ is 
the group of complex points of an algebraic torus. 
By the functoriality of 
Kottwitz's map $w_G$ it follows easily that $G_{der}\subset G_1\subset G^1$, 
so we have a natural embedding $X_{nr}(G)\subset X_{wur}(G)$.
If $A\subset Z(G)$ is the group 
of $F$-points of the maximal $F$-split subtorus of the center $Z(G)$ of $G$ then 
the restriction map gives a canonical embedding $X^*(G)\subset X^*(A)$ with 
finite cokernel. 
It follows that there is a canonical surjection with finite kernel 
$X_{nr}(G)\to X_{nr}(A)$ which we will denote by $m$. 
Again by the functoriality of Kottwitz's map we also see easily that 
there exits a canonical surjection with finite kernel 
$X_{wur}(G)\to X_{wur}(A)=X_{nr}(A)$. 
Therefore $X_{nr}(G)\subset X_{wur}(G)$ is the identity component.
An irreducible representation of 
unipotent reduction $\pi$ canonically defines an unramified character $z_\pi\in X_{nr}(A)$ 
since the scalar action of $A$ on $V_\pi$ is clearly unramified. For $\omega\in  X_{wur}(G)$ 
we have $z_{\omega.\pi}=m(\omega)z_\pi$. 

We mention in passing that $X_{wur}(G)$ and $X_{nr}(G)$ are not sensitive to 
inner twists; in particular we have canonical isomorphisms $X_{wur}(G)=X_{wur}(G^*)$
and $X_{nr}(G)=X_{nr}(G^*)$. 
\\

Our main result Theorem \ref{thm:main} deals with existence and uniqueness of \emph{a parameterisation}  
(the precise meaning of the set $\hat{G}^{\textup{temp}}_{uni}$ is explained in Section \ref{sect:uniptypes}):
\[\varphi: \hat{G}^{\textup{temp}}_{uni}\to \Phi_{nr}^{\textup{temp}},\   \pi\to\varphi_\pi\]  
such that the conjectures \ref{conj:1} and \ref{conj:2} hold. By a parameterisation we mean: 
\begin{definition}\label{def:desi}
A map $\hat{G}^{\textup{temp}}_{uni}\to \Phi_{nr}^{\textup{temp}}$, $\pi\to\varphi_\pi$ is called ``a parameterisation'' 
only if it satisfies following properties:  
\begin{enumerate}
\item[(i)] The map $\pi\to\varphi_\pi$ is equivariant 
for the actions of the group of tempered weakly unramified characters $X_{wur}^{\textup{temp}}(G)$ 
as defined above (cf. \cite[\S 10.2; 10.3(2)]{Bo}). 
\item[(ii)] We can express the character by which the center $Z(G)$ acts on $V_\pi$ in terms of $\varphi_\pi$, 
as in \cite[10.1; 10.3(1)]{Bo}. 
\item[(iii)] Compatibility with unitary parabolic induction as in \cite[\S 10.3(3), 11.3, 11.7]{Bo}. 
In particular, an irreducible 
tempered representation is a direct summand of an induced tempered representation $i_P^G(\delta)$ 
for some relevant parabolic $P=MN$ if and only if $\varphi_\pi$ is equivalent to a parameter 
$\varphi^M_{\delta}\in \Phi_{nr}^{\textup{temp}}(M)$ for $M$ 
(considering ${}^LM$ as a subgroup of ${}^LG$), where $\varphi^M$ is a parameterisation for $M$. 
\item[(iv)] The parameterisation is compatible with restriction of scalars \cite[\S 10.1]{Bo}, 
and taking products of reductive groups. 
\item[(v)] Let $\eta:H\to G$ is an $F$-morphism of connected reductive groups with commutative 
kernel and cokernel, and let $\varphi\in\Phi_{nr}^{\textup{temp}}(G)$. 
Given $\pi\in\Pi_\varphi(G)^{\textup{temp}}$, the pull-back of $\pi$ to $H$ 
is a finite direct sum of tempered irreducible representations in $\Pi_{{}^L\eta\circ\varphi}^{\textup{temp}}(H)$, 
where ${}^L\eta:{}^LG\to{}^LH$ denotes the natural map 
(cf. \cite[\S 10.3(5)]{Bo}). It follows in particular that  
a representation $\pi\in\hat{G}^{\textup{temp}}_{uni}$ factorizes through a representation 
of $G/A$ (with $A$ as above) if and only if $\textup{Im}(\varphi_\pi)\subset {}^LG^\natural$ 
(the $L$-group of $G/A$). (We use here that $G$ maps surjectively to $(\mbf{G}/\mbf{A})(F)$.)
\item[(vi)] For Iwahori-spherical representations of unramified connected reductive groups, 
the correspondence equals that of Theorem \ref{thm:Iw}
(compare with \cite[10.4]{Bo}).
\end{enumerate}
\end{definition}

\subsection{Unipotent types and unipotent affine Hecke algebras}\label{sect:uniptypes}
Let $\Lambda_0=\mathbb{C}[v^\pm]$, with $v$ a formal variable. Let $\vt\in\mathbb{R}_+$ 
be such that $\vt^2=\qt=|\mathfrak{O}/\mathfrak{p}|$.
We remark that there are no ``bad primes'' for representations of unipotent reduction
\cite{Lusztig-chars}, and we may and will often replace $\vt$ by the indeterminate $v$ in the theory.
For example, we can view the Hecke algebra over $\mathbb{C}$ with parameter $\vt$ as 
specialization of the corresponding generic Hecke algebra over $\Lambda_0$. 
\\
 
We assume from now on that $\mbf{G}$ is connected reductive 
over $F$ and splits over an unramified extension of $F$. Let $F_{ur}\supset F$ be a maximal unramified 
extension of $F$. Let $\mbf{T}\subset \mbf{G}$ be a maximally split $F$-torus which is $F_{ur}$-split.
In the apartment $\mathcal{A}_1$ of $\mbf{T}(F_{ur})$ we can choose an $\textup{Fr}$-stable alcove, 
by \cite[1.10.3]{Tits}. Let $\mathbb{I}_1\subset G(F_{ur})$ be the corresponding $\textup{Fr}$-stable 
Iwahori subgroup, and let $\mathbb{I}=\mathbb{I}_1^{\textup{Fr}}$ be the corresponding Iwahori 
subgroup of $G$.
\\

Steinberg's vanishing theorem $H^1(F_{ur},\mbf{G})=1$ implies that: 
$$H^1(F,\mbf{G})=H^1(\textup{Gal}(F_{ur}/F),\mbf{G}(F_{ur}))$$. 
Kottwitz's Theorem expresses this 
in terms of the center of the Langlands dual group:  
$$H^1(F,\mbf{G})=H^1(\textup{Fr},\mbf{G}(F_{ur}))
=[\Omega/(1-\theta)\Omega]_{tor}=\textup{Irr}(\pi_0(Z(G^\vee)^\theta))$$ with $\Omega=X_*(T)/Q$, where 
$Q$ denotes the root lattice of the dual group $G^\vee$. \footnote{From now on we will call 
roots of $(G^\vee,T^\vee)$ ``roots'', and roots of $(\mbf{G},\mbf{T})$ ``coroots''. We apologize for the 
incovenvience this may cause.}
The inner forms of $G$ are parametrised by $H^1(F,\mbf{G}_{ad})=(Z(G^\vee_{sc})^\theta)^*=
\Omega_{sc}/(1-\theta)\Omega_{sc}$ with $\Omega_{sc}=P/Q$, where $P$ is the weight lattice of $Q$.  
Given $\omega\in H^1(F,\mbf{G}_{ad})$ we may choose a representative $u\in N_{G_{ad}(F_{ad})}(\mathbb{I}_1)$ 
whose image in $H^1(F,\mbf{G}_{ad})=\Omega_{sc}/(1-\theta)\Omega_{sc}$ is $\omega$. Then 
the inner twist $\textup{Fr}_u:=\textup{Ad}(u)\circ \textup{Fr}^*$ of $\textup{Fr}^*$ defines an inner form 
of $G^*$ corresponding to $\omega$, which we will often denoted by $G^u$.
\\

Let $\Pbb\subset G$ be a parahoric subgroup. There exists an $\Fr$-stable parahoric $\mbf{P}\subset\mbf{G}(F_{ur})$ 
such that $\Pbb=\mbf{P}^\Fr$. 
We put as before 
$$
\textup{Vol}(\Pbb):=\vt^{-\textup{dim}(\overline{\Pbb})}|\overline{\Pbb}|
$$
This is the specialization at $v=\vt$ of a Laurent polynomial $\textup{Vol}(\Pbb)\in\Lambda_0$ in $v$. 
Let $\sigma$ be a cuspidal unipotent representation of $\overline{\Pbb}$, lifted to 
$\Pbb$. Let $\tilde{\mathbb{P}}:=N_{G}(\mathbb{P})$, and choose an 
extension of $\sigma$ to a representation $\tilde{\sigma}$ of $\tilde{\Pbb}\subset G$
(such extensions exist \cite{Lusztig-unirep}). Then $t=(\Pbb,\sigma)$ is a 
type (see Theorem \ref{thm:unitype}) for a finite set of Bernstein components of representations 
of unipotent reduction of $G$. Notice that $t$ and the extension 
$\tilde{t}=(\tilde{\Pbb},\tilde{\sigma})$ are determined  
by data (the local Tits index of $G$ (with trivial action of the inertia group), a facet of the 
apartment of $T$, a cuspidal unipotent representation $\sigma$ of the corresponding parahoric 
subgroup $\Pbb$, and an extension to its normalizer in $G$) which are independent of the base 
field $F$ of $G$. (We use here that the classification of cuspidal unipotent characters of finite 
groups of Lie type is independent of the field of definition \cite{Lusztig-chars}). 
We write $t=(\Pbb,\sigma)$ to refer to this ``abstract'' unipotent type (in which the base 
field $F$ is undetermined, and the cardinality of its residue field considered as indeterminate $v$), 
while we often write $t_\vt=(\Pbb_\vt,\sigma_\vt)$ if we want to refer to the ``concrete'' type of $G$
``specialised at $v=\vt$''. Similarly for $\tilde{t}$.
Such ``families of unipotent types'' $t$ (with varying base field $F$) have explicit meaning on 
the Langlands dual side, as we will see shortly. 
\begin{thm}[\cite{MP1}, \cite{MP2}, \cite{Mo}, \cite{Lusztig-unirep}, \cite{Lusztig-unirep2}]\label{thm:unitype}
Let $t=(\Pbb,\sigma)$ and $\tilde{t}=(\tilde{\mathbb{P}},\tilde{\sigma})$ be as above, and let 
$\Omega^{\mathbb{P},\theta}=\tilde{\Pbb}/\Pbb\subset 
\Omega^\theta$ be the stabilizer of $\Pbb$ (see \cite[1.16]{Lusztig-unirep}).

Then $t_\vt$ is a type for $G$ with Hecke algebra $\Hc_{t,\vt}$ which is of 
the form $\Hc_{t,\vt}=\Hc_{t,\vt}^a\rtimes \Omega^{\mathbb{P},\theta}$, where 
$\Hc^a_{t,\vt}$
is the specialization at $v=\vt$ of a generic (unextended) affine Hecke algebra $\Hc_t^a$ 
defined over $\Lambda_0$ (depending on $t$ only), on which $\Omega^{\mathbb{P},\theta}$ acts 
via diagram automorphisms. If $\Omega^{\mathbb{P},\theta}_1\subset \Omega^{\mathbb{P},\theta}$ 
denotes the subgroup which acts trivially on $\Hc^a_{t}$, and 
$\Omega^{\mathbb{P},\theta}_2=\Omega^{\mathbb{P},\theta}/\Omega^{\mathbb{P},\theta}_1$ then 
\begin{equation}\label{eq:extendtype}
\Hc_{t}=\Hc_{\tilde{t}}^e\otimes\mathbb{C}[\Omega^{\Pbb,\theta}_1];\ \ 
\Hc_{\tilde{t}}^e=\Hc_t^a\rtimes \Omega^{\mathbb{P},\theta}_2
\end{equation}
where $\Hc_{\tilde{t}}^e$ is an extended affine Hecke algebra. 
The set of extensions 
$\tilde{t}$ of $t$ is a torsor for the quotient $(\Omega^{\mathbb{P},\theta})^*$ of the 
group 
$X_{wur}(G)=(\Omega^{\theta})^*$, 
and $\Hc_{\tilde{t}}^e$ corresponds to the orbit of $\tilde{t}$ under the subgroup 
$(\Omega^{\mathbb{P},\theta}/\Omega^{\mathbb{P},\theta}_1)^*$.
Thus the summands $\Hc_{\tilde{t}}^e$ of $\Hc_{t}$ form a 
torsor for the quotient $(\Omega^{\Pbb,\theta}_1)^*$, and 
$X_{wur}(G)=(\Omega^\theta)^*$ acts on the center of $\Hc_t$. 

Let $\textup{Vol}(\Pbb)\in\Lambda_0$ be the Laurent polynomial defined above, 
let $\textup{deg}(\sigma)\in\Lambda_0$ 
be the polynomials such that $\textup{deg}(\sigma)(\vt)=\textup{deg}(\sigma_\vt)$, 
then the trace $\tau$ of $\Hc_{\tilde{t}}^e$ is normalised by 
$\tau(1)=d^t=|\Omega^{\Pbb,\theta}_1|^{-1}\frac{\textup{deg}(\sigma)}{\textup{Vol}(\mathbb{P}^{\Fr_u}}$. 
This turns the summands $\Hc_{\tilde{t}}^e$ into normalised affine Hecke algebras in the sense 
of \cite{Opdl}.
\end{thm}
\begin{proof}
We refer the reader to \cite[Section 2.4, Theorem 2.8]{Opdl}, 
where this is worked out in detail, and to the discussion in \cite{Lusztig-unirep}. 
\end{proof}
\begin{definition}[Lusztig \cite{Lusztig-unirep}]\label{def:unip} 
The category $\mathcal{C}_{uni}(G)$ of smooth representations  
of $G$ of unipotent reduction is the direct product over all conjugacy classes of unipotent types 
$t=(\Pbb,\sigma)$ of the abelian subcategories $\mathcal{C}^t(G)$ of $\mathcal{C}(G)$, 
where $\mathcal{C}^t(G)$ is Morita equivalent to the Hecke algebra $\Hc_{t,\vt}$ of $t$. 
Since $\Hc_t$ is a direct sum  (\ref{eq:extendtype}) of (mutually isomorphic) 
extended affine Hecke algebras $\Hc_{\tilde{t}}^e$ parameterised by the 
set of characters of $\Omega^{\Pbb,\theta}_1$,  
each subcategory $\mathcal{C}^t(G)$ decomposes as a product over a finite set of Bernstein 
components $\mathcal{C}^{\tilde{t},e}(G)$ parameterised by the set of extensions $\tilde{t}$ 
of $t$ to the inverse image of $\Omega^{\Pbb,\theta}_1\subset \Omega^{\Pbb,\theta}$ in 
$\tilde{\Pbb}$. This is a torsor for $(\Omega^{\Pbb,\theta}_1)^*$, 
such that for each extended type $\tilde{t}$, the map 
$(\pi,V_\pi)\to 
\textup{Hom}_{N_{G}(\Pbb^{\Fr_u})}(\tilde{\sigma}_\vt,V_\pi|_{N_{G}(\Pbb^{\Fr_u})})$ 
is a Morita equivalence from $\mathcal{C}^{\tilde{t},e}(G)$ 
to $\Hc_{\tilde{t},\vt}^e$.  
\end{definition}
Using his arithmetic-geometric diagram correspondences 
Lusztig constructed \cite{Lusztig-unirep}, \cite{Lusztig-unirep2} a parameterisation of the 
irreducible objects of $\mathcal{C}_{uni}(G)$ if $G$ is simple of adjoint type.
In particular: 
\begin{thm}[Lusztig]\label{thm:lus}
Let $G=G^u$ be (the group of points of) a simple group of adjoint type defined over $F$.
There exists a partitioning 
\begin{equation}
\textup{Irr}^{\textup{temp}}_{uni}(G)=\bigsqcup_{\varphi\in\Phi^{\textup{temp}}_{nr}}\Pi_\varphi
\end{equation}
such that for all $\varphi\in\Phi^{\textup{temp}}_{nr}$, there is a bijection between 
$\Pi_\varphi$ and $\Pi(\Sc_\varphi,\chi_G)$ (where $\chi_G\in \Omega=(Z_{sc}^\theta)^*$ 
and $\chi_G$ and $\Pi(\Sc_\varphi,\chi_G)$  as in Conjecture \ref{conj:LC}).  
This map $\pi\to\varphi_\pi$ can be taken $X_{wur}(G)$-equivariantly.  
\end{thm} 
For $G=G^*$ and $t_\mathbb{I}=(\mathbb{I},\textup{triv})$, $\mathcal{C}^{t_\mathbb{I}}(G)$ 
is the Bernstein component 
of the minimal unramified principal series. Indeed, by Borel's classical result  
this abelian subcategory of $\mathcal{C}(G)$ is equal to the category of smooth 
representations of $G$ which are generated by their Iwahori-fixed vectors. Then the 
equivalence of Definition \ref{def:unip} is the classical equivalence between  
$\mathcal{C}^{t_\mathbb{I}}(G)$ and the module category of the (extended) 
Iwahori-spherical Hecke algebra $\Hc_{t_\mathbb{I},\vt}^e=\Hc_{\mathbb{I},\vt}$. Via this equivalence, 
the restriction of the correspondence of Theorem \ref{thm:lus} to $\mathcal{C}^{t_\mathbb{I}}(G)$
becomes the classical Kazhdan-Lusztig parameterisation \cite{KLDL}.
Theorem \ref{thm:Iw} (and its proof) shows that in this special case, this parameterisation restricted 
to tempered representations satisfies the conjectures \ref{conj:1} and \ref{conj:2}, and that this 
can in fact be extended to general reductive groups.
\\

The basic problem one is facing when trying to extend this result to all tempered representations 
of $\mathcal{C}_{uni}(G^u)$ (with $G^u$ an arbitrary inner twist of $G$), is how one should  
parameterise the tempered irreducible representations of the affine Hecke algebras of the form 
$\Hc_{\tilde{t}}^e$ (with $t=(\mathbb{P},\sigma)$ a unipotent type of $G^u$, and 
$\tilde{t}$ an extension) in terms of tempered unramified Langlands parameters for $G^u$.
Lusztig \cite{Lusztig-unirep}, \cite{Lusztig-unirep2} does this via his theory of local 
systems on $G^\vee$-orbits of Langlands parameters, and the remarkable isomorphisms between 
arithmetic diagrams (related to the affine diagram of $\Hc_{\tilde{t}}^e$) and geometric diagrams 
related to graded affine Hecke algebras associated with a local system.
We follow a different approach in which the conjectures \ref{conj:1} and \ref{conj:2} play a central role.
In view of Theorem \ref{thm:cc}, it is clear that one would like to map the spectrum of the 
center $Z(\Hc_{\tilde{t},\vt}^e)$ of $\Hc_{\tilde{t},\vt}^e$ to the spectrum of the center of $Z(\Hc_{\mathbb{I},\vt})$, 
so that $S(\Hc_{\tilde{t},\vt}^e)$ maps to $S(\Hc_{\mathbb{I},\vt})$, and such that the Plancherel measures 
(see Theorem \ref{thm:ds} and Theorem \ref{thm:parb}) up to constant factors correspond. 
We call such a map a \emph{spectral transfer map} (STM) $\Hc_{\tilde{t},\vt}^e\leadsto \Hc_{\mathbb{I},\vt}$.  
Such maps  
turns out to exist and turn out to be \emph{essentially unique}. Moreover, these STMs are  
essentially ``the same'' as the maps implicit in Lusztig's arithmetic-geometric 
correspondences. Seeing this through (cf. \cite{Opdl}, \cite{FO}, \cite{FOS}) recovers Lusztig parameterisation 
$\pi\to\varphi_\pi$, proving at the same time the conjectures \ref{conj:1} and \ref{conj:2} for this parameterisation. 
In our Main Theorem \ref{thm:main} we have further extended this to general connected reductive group. 
\subsection{Langlands parameters and residual cosets}
Let us first consider the case of cuspidal representations of unipotent reduction.

Let $t=(\Pbb,\sigma)$ be a unipotent type of $G$, with $\Pbb\subset G$   
a maximal parahoric subgroup. Assume that $G$ has anisotropic center. 
Then $\Hc_t = \Lambda_0[\Omega^{\Pbb,\theta}]$ has rank $0$, and 
its trace is normalised by $\tau(1)=d^t=\frac{\textup{deg}(\sigma)}{\textup{vol}(\Pbb)}$. 
Let $\pi$ be a supercuspidal unipotent character of $G^u$ belonging to the finite set $\mathcal{C}^t(G)$.  
Then, in view of Theorems \ref{thm:ds}, \ref{thm:cc} and \ref{thm:Iw} we need to find a 
$\Lambda_0$-valued residual point 
$r\in T_\mathbb{I}(\Lambda_0)$ such that: 
\begin{equation}\label{eq:cuspmu}
\textup{fdeg}(\pi)
=c\mu^{\{r\}}(r)=\frac{c}{q(w_0)\textup{Vol}(\Ibb)}m_r(v)
\end{equation}
for some constant factor $c\in\mathbb{R}_+$, 
where 
\begin{equation}\label{eq:cuspar}
\textup{fdeg}(\pi)=\frac{\textup{deg}(\sigma)}{|\Omega^{\Pbb,\theta}|\textup{Vol}(\Pbb)}
\end{equation}
Recall that by Theorem \ref{thm:cc}, an orbit of residual points $W_0r\in W_0\backslash T_{\Ibb}$  corresponds to 
a unique orbit of discrete unramified Langlands parameters $\varphi\in\Phi_{nr}^{\textup{temp}}(G^*)$ 
such that $W_0r=[\varphi(\tilde{\Fr}))]\in W_0\backslash T_{\Ibb}$. 
By Theorem \ref{thm:Iw} we then have, in terms of $\varphi$, for some constant factor $c'$: 
\begin{equation}\label{eq:mugamma}
\textup{fdeg}(\pi)=c\mu^{\{r\}}(r)=c'\gamma(0,\textup{Ad}\circ\varphi,\psi)
\end{equation}
viewed as identity of rational functions in $v$ (where $v^2=q$).

\begin{ex} We begin with a very basic example.
Let $G=\textup{PGL}_{m+1}(F)$. The only inner form of $G$ which has a cuspidal 
type in this case is the anisotropic inner form $G^u=\Dbb^\times/F^\times$ with $\Dbb$  
the tame division algebra of degree $m+1$ over $F$. The unique parahoric subgroup 
of $G^u$ is $\mathbb{P}^u=G^u_1=\cap_{\chi\in X_{wur}(G^u)}\textup{Ker}(\chi)$, which 
has a unique unipotent cuspidal $\sigma=\textup{triv}$. This gives a maximal unipotent 
type $t=(\mathbb{P}^u,\sigma)$ whose extensions $\tilde{t}$ to $G^u$ are given 
by $X_{wur}(G^u)=\Omega^*$. Let us call the corresponding cuspidal unipotent 
characters of $G^u$: $\pi_\chi$ (with  $\chi\in X_{wur}(G^u)$). 
Now $G$ has essentially only one 
unramified discrete Langlands parameter, the regular parameter $\varphi_0$, up to the 
action of $X_{wur}(G)=\Omega=C_{m+1}$. We have $\varphi_0|_{W_F}=1$, and 
$\varphi_0(\begin{pmatrix}1&1\\0&1\end{pmatrix})$ is a regular unipotent element.
Hence $W_0r=[\varphi_0({\tilde{\Fr}})]\in W_0\backslash T_\mathbb{I}$, and the 
$\Omega^*$-orbit $\Omega^*r$ is defined by 
the equations $\alpha_i(r)=\qt$ for $i=1\dots m$. We check simply from (\ref{eq:cuspmu})
and (\ref{eq:cuspar}) that 
\begin{equation}
\textup{fdeg}(\pi_\chi)=\mu^r(r)=
\frac{1}{(m+1)[m+1]_q},\ \text{\ with\ }[n]_q:=\frac{v^{m+1}-v^{-(m+1)}}{v-v^{-1}}.
\end{equation}
\end{ex} 
\begin{ex} Another example which is known for a long time \cite{Re0}. 
Let $G=\mbf{G}_2(F)$. Let $\pi$ be the cuspidal unipotent character of $G$ 
which is compactly induced from the cuspidal unipotent representation $\mbf{G}_2(\mathbb{F}_{\qt})$ 
which is denoted by $\mbf{G}_2[1]$, inflated to $\Pbb= \mbf{G_2}(\mathcal{O}_F)$. Then 
\begin{equation}
\textup{fdeg}(\pi)=\frac{1}{6}\mu^r(r)= 
\frac{1}{6(v+v^{-1})^2(v^2+1+v^{-2})}.
\end{equation}
where $W_0r=[\varphi_{sub}({\tilde{\Fr}})]\in W_0\backslash T_\mathbb{I}$ is 
the $\Lambda_0$-point associated with the real discrete unramified Langlands parameter 
$\varphi_{sub}$, associated with the subregular unipotent orbit of $\mbf{G_2}$.
This uniquely determines $W_0r$ and thus $\varphi_{sub}$. Indeed, there is a cuspidal local 
system supported by $G^\vee\dot\varphi_{sub}$, and Lusztig maps $\pi$ to this cuspidal local 
system. 
\end{ex}
It turns out that this always works, at least in the case of $G$ being absolutely simple and of adjoint type:
\begin{thm}[Reeder \cite{Re0}, \cite{Re} (split exceptional groups), \cite{HOH}, \cite{FO}, \cite{Fe2}, \cite{FOS}]\label{thm:cuspidalcase}
Let $G$ be a simple group of adjoint type defined over $F$ and  split over an unramified extension of $F$.
Let $\pi$ be a supercuspidal representation of $G$ of unipotent reduction.
Let $t=(\Pbb,\sigma)$ be a unipotent type for $G$ such that $\pi$ belongs to $\mathcal{C}^t(G)$. 
Then $\Pbb\subset G$ is a maximal parahoric subgroup. Conversely, if $\Pbb\subset G$ is maximal then 
$\mathcal{C}^t(G)$ consists of supercuspidal unipotent representations.  
In this situation there exists a \emph{unique}   
$X_{wur}(G)$-orbit of discrete unramified Langlands parameters $[\varphi]\in\Phi^{\textup{temp}}_{nr}(G)$
such that:
\begin{equation}\label{eq:cuspSTM}
\textup{fdeg}(\pi)=
c\gamma(0,\textup{Ad}\circ\varphi,\psi)
\end{equation}
The collection of classes of discrete unramified Langlands parameters $[\varphi]\in\Phi^{\textup{temp}}_{nr}(G)$ thus 
obtained is exactly equal to the set of classes of discrete unramified Langlands parameters which 
support a cuspidal local system.
\end{thm}
Recall that $\Hc_{\tilde{t}}^e=\Lambda_0$, and that the trace of this Hecke algebra 
is normalised by $\tau^e(1)=\frac{\textup{deg}(\sigma)}{|\Omega^{\Pbb,\theta}|\textup{Vol}(\Pbb)}$. 
In the terminology of STMs we view a $\Lambda_0$-valued point 
$r:\textup{Spec}(\Lambda_0)\to T_\Ibb(\Lambda_0)$ such that (\ref{eq:cuspSTM}) holds 
as realizing a cuspidal STM $\Hc_{\tilde{t}}^e\leadsto\Hc_\Ibb$. 
In the next section we will discuss the notion of STM in the general case.
\\

The proof of Theorem \ref{thm:cuspidalcase} is difficult. In the exceptional cases it reduces to 
explicit case by case computations due to \cite{HOH}, \cite{Re}, \cite{Fe2}. 
In the classical case, see \cite{FO}. This is based on the existence an explicit set of 
generators between the unipotent normalised affine Hecke algebras of 
type $\textup{C}_n$, and the observation that for classical cases other than $\textup{PGL}_{m+1}(F)$, 
the expressions $\frac{\textup{deg}(\sigma)}{\textup{Vol}(\Pbb)}$ do not contain 
odd cyclotomic factors in the numerator or the denominator. This property turns out 
to eliminate most of the discrete unramified Langlands parameters $\varphi$ 
for which an identity of the form (\ref{eq:cuspSTM}) could hold. Analysing the remaining 
cases carefully using the STMs whose \emph{existence} was established directly in 
\cite{Opdl}, \cite{Fe2} completes the proof.

\subsection{Spectral transfer maps}
Let us now introduce the notion of STM formally.
Assume we are given two normalised affine Hecke algebras $\Hc_1,\Hc_2$ defined over $\Lambda_0$. 
For $i=1,2$ we have the torus $T_i$ defined over $\Lambda_0$ associated with the character 
lattice $X_i$ of the root datum of $\Hc_1$. We have the Weyl groups $W_{0,i}$ acting on $T_i$, and 
the $\mu$-function $\mu_i$, a rational function on $T_i$.
\\

Given a residual coset $L\subset T_2$ defined over $\Lambda_0$, we know that $L=r_LT^L$ 
for some residual point $r_L\in T_L$, by Proposition \ref{prop:par}. Let $L_n=L/{K^n_L}$ be the smallest quotient 
of $L$ such that for each $K^n_Lt\in L_n$, the orbit $W_{0,2}K^n_Lt=W_{0,2}t$ is a single $W_0$-orbit in $T_2$.
Then $K^n_L\subset K_L=T_L\cap T^L$ is the finite abelian subgroup 
$K^n_L=N_{W(\Sigma_L)}(L)/Z_{W(\Sigma_L)}(L)\subset K_L$.  Then $\mu_2^{L}$ (see Theorem \ref{thm:parb}) 
is $K^n_L$ invariant, 
and descends to a $\mu$-function $\mu_2^{L_n}$.
\begin{definition}\label{def:STM}
A spectral transfer map $\Hc_1\leadsto\Hc_2$ between the normalised affine Hecke algebras 
over $\Lambda_0$ is an equivalence class of morphisms 
$\Psi:T_1\to L_n$ of torsors of algebraic tori defined over $\Lambda_0$, 
where $L\subset T_2$ is a residual coset for $\Hc_2$, such that 
\begin{enumerate}
\item[(1)] $\Psi$ is surjective with finite fibres. 
\item[(2)] For all $w_1\in W_{0,1}$ there exists a $w_2\in N_{W_{0,2}}(L)$ 
such that $\Psi\circ w_1=w_2\circ \Psi$.
\item[(3)] $\Psi^*(\mu_2^{L_n})=D\mu_1$ for some constant $D\in \mathbb{Q}^\times$. 
\end{enumerate}
We call $\Psi$ and $\Psi'$ equivalent if $\Psi'=w\circ\Psi$ for some $w\in W_{0,2}$.
\end{definition}
\begin{rem}
\begin{enumerate}
\item[(i)] In the special case $\textup{dim}(T_1)=0$ this recovers the notion cuspidal STMs 
as discussed above for unipotent affine Hecke algebras $\Hc_{\tilde{t}}^e$. 
\item[(ii)] The property $(2)$ in Defnition \ref{def:STM} is almost always superfluous, except 
in degenerate cases \cite[cf. Proposition 5.4]{Opds}. 
For unipotent affine Hecke algebras this property is always automatic.
\item[(iii)] Note that $\Psi$ defines a morphism $\Psi_Z: \textup{Spec}(Z_1)\to \textup{Spec}(Z_2)$
between the spectra of the centers of the $\Hc_i$. 
\item[(iv)] Suppose that $L_1\subset T_1$ is a residual coset for $\Hc_1$, and let $\textup{Im}(\Psi)=L_n$. 
Then $\Psi_Z(W_{1,0}L_1)=N_{W_{0,2}}(L)(L_2)$ for some residual coset $L_2\subset L$, 
and all residual cosets $L_2\subset L$ are the image under $\Psi$ of a residual coset of $\Hc_1$ 
in this sense. 
\item[(v)] We can compose STMs. This is useful as we may generate in this way all STMs between 
the unipotent affine Hecke algebras of the various groups $G^u$ from a small number of generators.
\end{enumerate}
\end{rem}
The most important property of STMs with respect to the spectral decomposition of 
affine Hecke algebras is the following result \cite[Theorem 6.1]{Opds}:
\begin{thm}\label{thm:corr}
Let $\Psi:\Hc_1\leadsto \Hc_2$ be an STM. Suppose that $C_{1,\vt}\subset \hat{\Hc}^{\textup{tem}}_{1,\vt}$
is a component of the space of irreducible tempered characters of $\Hc_{1,\vt}$. Let 
$cc_{i,\vt}:\hat{\Hc}^{\textup{temp},\vt}_i\to \textup{Spec}(Z_{i,\vt})$ be the central character map of $\Hc_{i,\vt}$.  
Then $cc_1(C_{1,\vt})=S_{L_{1,\vt}}$ where $L_1\subset T_1$ is a residual coset. 
There exists a unique orbit $W_{0,2}L_2\subset T_2$ of residual cosets of $\Hc_2$ such that 
$\Psi_Z(S_{L_{1,\vt}})=S_{L_{2,\vt}}$. If $C_{2,\vt}\subset \hat{\Hc}^{\textup{tem}}_{2,\vt}$ 
is such that $cc_2(C_{2,\vt})=S_{L_{2,\vt}}$ then consider the fibred product $C_{12}$ defined by the 
diagram
\begin{equation}\label{eq:commreg}
\begin{CD}
C_{12}@>{P_2}>>C_2\\
@V{P_1}VV                  @VV{cc_2}V\\
C_1@>>{\Psi_Z\circ cc_1}> S_{L_2}
\end{CD}
\end{equation}
Then there exist constants $r_i\in\mathbb{Q}_+$ and a family of positive measures  
$\nu_\vt$ on $C_{12,\vt}$ (for all $\vt>0$) such that for $i=1,2$ we have 
$r_i.P_{i,*}(\nu_\vt)=\nu_{i\vt}|_{C_{1,\vt}}$. 
\end{thm}
This result states that an STM $\Psi: \Hc_1\leadsto\Hc_2$ defines a correspondence between the 
tempered spectrum of $\Hc_1$ and of $\Hc_2$ which respects the connected components in these 
tempered irreducible spectra and, 
up to constant rational factors only depending on the components, the Plancherel measures 
$\nu_{i,\vt}$ of $\Hc_i$.  
\\

Let $\mbf{G}$ be a connected reductive group defined over $F$ which splits over an unramified 
extension of $F$. Let $t=(\Pbb,\sigma)$ be a unipotent type of $G$, i.e. $\Pbb\subset G$ is  
a parahoric subgroup, and $\sigma$ is a cuspidal unipotent representation of $\Pbb$. 
Suppose in Theorem \ref{thm:corr} that $\Hc_1=\Hc_t(G)$ (a finite direct sum of 
extended affine Hecke algebras of the form $\Hc_{\tilde{t}}^e$), and that $\Hc_2=\Hc_{\Ibb}(G^*)$, 
the Iwahori Hecke algebra of a quasi-split inner form $G^*$ of $G$
(with the obvious extension of the notion of an STM on a finite direct sum of extended affine Hecke 
algebras such as $\Hc_{t}(G)$, by allowing $T_1$ to be disjoint union of algebraic tori over $\Lambda_0$).

Let $M\subset G$ be a Levi subgroup such that the set $\Omega^t$ of $G$-conjugacy classes  
of cuspidal pairs which belong in the inertial classes covered by $t$ are of the form $[(M,\delta)]$.  
By Corollary \ref{cor:beta} the diagonalizable group $T_{1,\vt}$ over $\mathbb{C}$ 
can be identified with the space $\Omega^t(M)$ of $M$-conjugacy classes of 
such cuspidal pairs. The natural action (by taking tensor products) of $X_{wur}(M)$ on  
$\Omega^t(M)$ turns $T_{1,\vt}$ into a $X_{wur}(M)$-space, 
and in fact each component of $T_{1,\vt}$ is a quotient of $X_{wur}(M)$ with finite kernel 
(because $X_{nr}(M)\subset X_{wur}(M)$ is the identity component, and the components of 
$\Omega^t(M)$ are by definition already homogeneous for $X_{nr}(M)$ with finite kernel).  

\begin{rem}
Recall that (cf. Theorem \ref{thm:cc}) the semisimple conjugacy classes of the set 
$G^\vee\theta\subset {}^LG$  are in natural bijection with the set of $W^\theta$-orbits in the complex torus 
$T_\Ibb=\textup{Hom}(X_*(T)^\theta,\mathbb{C}^\times)$. Observe that 
\[X_{wur}(M)=\textup{Hom}((X_*(T)/Q_M)^\theta,\mathbb{C}^\times)\to 
\textup{Hom}(X_*(T)^\theta,\mathbb{C}^\times)=T_2=T_\Ibb.\]
Hence $X_{wur}(M)$ also acts naturally on $T_\Ibb$ (and faithfully, in fact).
\end{rem}
\begin{cor}\label{cor:STMpar}  Let $G$, $\Hc_1$ and $\Hc_2$ be as above. 
Suppose that $\Psi:\Hc_{t}(G)\leadsto \Hc_{\Ibb}(G^*)$ is a $X_{wur}(G)$-equivariant STM.
Let $\psi: \hat{G}^{t,temp}\to \Phi^{\textup{temp}}_{nr}(G)$ be the map 
$(\gamma^\Ibb)^{-1}\circ \Psi_Z\circ\beta^{t}\circ cc$ (cf. (\ref{eq:bercc}), 
Corollary \ref{cor:beta},  Theorem \ref{thm:cc}
and Definition \ref{def:STM}).  
\begin{enumerate}
\item[(1)] Via the map $\pi\to \psi_\pi$ on $\hat{G}^{t,\textup{temp}}$, the restriction of the 
Plancherel measure of $G$ to $\hat{G}^{t,\textup{temp}}$ is expressed as in 
conjectures \ref{conj:1} and \ref{conj:2}, up to rational constant factors.
\item[(2)] 
$\psi$ satisfies the conditions (i), (ii), and (vi) of a  parameterisation if and only if 
the following additional compatibility conditions hold:  
\begin{enumerate}
\item[(a)] $\Psi(\eta.t)=\eta.\Psi(t)$ for all $\eta\in X_{wur}(M)$, $t\in T_{1,\vt}$. 
\item[(b)] If $t=(\Ibb,\textup{triv})$, with $\Ibb\subset G^*$ the Iwahori subgroup 
of a quasi-split group $G^*$ which is an inner twist of $G$.  Then we require that 
$\Psi:\Hc_\Ibb\leadsto \Hc_\Ibb$ is the identity.
\end{enumerate}
\end{enumerate}
\end{cor}
\begin{proof}
By Theorem \ref{thm:BHK} $\hat{m}_t$ defines a Plancherel measure preserving homeomorphism
from $\hat{G}^{t,\textup{temp}}$ to $\hat{\Hc}_t(G)$. 
By Theorem \ref{thm:ds}, Theorem \ref{thm:parb} a component $C_1$ in the tempered irreducible spectrum 
of $\Hc_t(G)$ is defined by unitary parabolic induction of discrete series characters modulo center of 
Levi-subalgebras. The image $cc^t(C_1)$ under the central character map $cc^t$ is the image 
$S_{L_1}\subset \textup{mSpec}(Z(\Hc_t(G)))$ of the tempered form of a residual coset 
$L_1^{\textup{temp}}$. 
Moreover the Plancherel measure of $\Hc_{t}(G)$ on $C_1$ is given, up to a constant factor, by the pull back 
$cc^{t,*}(\mu_t^{L_1}|_{S_{L_1}})$ on $L_1$ to $C_1$ (using that $cc^t$ is a smooth finite covering map on 
a dense open subset of $C_1$). Hence modulo constant factors, the Plancherel density on $C_1$ is a function 
of the central character only. 

By Theorem \ref{thm:corr} $\Psi_Z(S_{L_1})=S_{L_2}\subset W_{\Ibb,0}\backslash T_{\Ibb}$ for a unique 
residual coset $L_2\subset L=\textup{Im}(\Psi)$, and up to rational constant factors 
$\Psi_Z^*(\mu_{\Ibb,L_2})=c\mu_t^{L_1}$. 
Hence Theorem \ref{thm:cc} and Theorem \ref{thm:Iw} imply that in this way 
we can express the Plancherel density at $\pi\in \hat{G}^{t,\textup{remp}}$
up to locally constant rational factors by the appropriate adjoint $\gamma$-factors in $\psi_\pi$, 
proving (1).

For (2): It is easy to see that $\psi$ satisfies (i), (ii) iff $\Psi$ satisfies the stated 
compatibility condition (a). Clearly (vi) makes sense only in the case of the Borel 
component of $G^*$. In this case, (2)(b) forces (vi) by insisting that $\Psi$ is identical. 
\end{proof}
\begin{rem}
Without this condition (2)(b) it would be allowed in the case $t=(\Ibb,\textup{triv})$ 
that our map $\pi\to \psi_\pi$ is that of 
Theorem \ref{thm:Iw} (as required in (vi)) but twisted by an STM $\Psi:\Hc_\Ibb\leadsto\Hc_\Ibb$ 
satisfying (2)(a). Such STMs are given by the action 
of $X_{wur}(G)=(\Omega^\theta)^*$. 
For general components of $\hat{G}^t$ 
in general I do not know of a preferred choice for $\Psi$ within 
its $X_{wur}(G)$-orbit.
\end{rem}
\subsection{Lusztig's geometric-arithmetic correspondences and STMs}
In this subsection we will assume that $\mbf{G}$ is absolutely simple and adjoint
(we will reduce the general case to this case), 
and that $\mbf{G}$ is split over an 
unramified extension of $F$.

The following result was essentially 
proven in \cite[Theorem 3.4]{Opds}:
\begin{thm}\label{thm:STMexuni}
Let $t=(\Pbb,\sigma)$ be a unipotent type for $G=\mbf{G}(F)$, and let $\Hc_\Ibb(G^*)$ 
denote the Iwahori Hecke algebra of the quasi-split inner twist $G^*$ of $G$.
Let $\tilde{t}=(\tilde{\Pbb},\tilde{\sigma})$ be an extension of $t$ to $\tilde{\Pbb}=N_G(\Pbb)$.
\begin{enumerate}
\item[(a)] There exist STMs $\Phi_{\tilde{t}}:\Hc_{\tilde{t}}^e\leadsto\Hc_\Ibb(G^*)$.  
\item[(b)] If $\Phi'_{\tilde{t}}:\Hc_{\tilde{t}}^e\leadsto\Hc_\Ibb(G^*)$ is also an STM, 
then there exists an 
$\omega\in X_{wur}(G)= X_{wur}(G^*)$ and a spectral automorphism 
$\alpha$ of  $\Hc_t$ such that $\Phi'_{\tilde{t}}=m_\omega\circ\Phi_{\tilde{t}}\circ \alpha$. Here
$m_\omega: \Hc_\Ibb(G^*)\leadsto \Hc_\Ibb(G^*)$ 
denotes the STM given by multiplication with $\omega$.
\item[(c)] There exists a $\Phi_t:\Hc_t^e\leadsto\Hc_\Ibb(G^*)$ which is 
$X_{wur}(G)$-equivariant and which satisfies the conditions (2)(a) and (b)
of Corollary \ref{cor:STMpar}. Such $\Phi_t$ is unique up to the action of 
$X_{wur}(G)$.
\item[(d)] The parameterisation $\varphi_t:\hat{G}^t\to \Phi^{\textup{temp}}_{nr}(G)$
associated to $\Phi_t$ as in Corollary \ref{cor:STMpar} is, up to a twist by $X_{wur}(G)$, 
the same as Lusztig's parameterisation 
$\varphi_{Lu}:\hat{G}^t\to \Phi^{\textup{temp}}_{nr}(G)$ of \cite{Lusztig-unirep}, 
\cite{Lusztig-unirep2}.
\end{enumerate}
\end{thm}
 In this subsection I will sketch the proof of the Theorem, which is quite involved. 
 For details the reader is 
 referred to \cite{Opdl}, \cite{Opds}, \cite{FO}, \cite{Fe2} and \cite{FOS}.
 
 First the existence (a) is proved in \cite{Opdl} and \cite{Fe2}. 
  
  In the exceptional cases the existence is shown by constructing  
STMs associated to Lusztig's arithmetic-geometric correspondences, in a way we will explain below.
Given a arithmetic-geometric correspondence of diagrams as in Lusztig, it is not difficult to find the 
candidate map $\Phi_{\tilde{t},T}:T_{\tilde{t}}\to L_n$ underlying the alleged STM
 $\Phi_{\tilde{t}}:\Hc_{\tilde{t}}^e\leadsto\Hc_\Ibb(G^*)$. To verify the main property (3) of 
 Definition \ref{def:STM} one needs to do a rather cumbersome computation.
  
 In the classical cases this is not practical. Fortunately the required STMs 
 can be obtained from a small set of generators of STMs between 
certain unipotent affine Hecke algebras of the form $\textup{C}_d[m_-,m_+](q^\beta)$
(see \cite[3.2.6; 3.2.7]{Opdl}). There are three kind of generating STMs: The 
translation STMs which decrease one of $m_\pm$ by $1$ (if this parameter is in $\mathbb{Z}_++1/2$)
or by $2$ (if this parameter is in $\mathbb{Z}_+$), increasing the rank accordingly, and do not 
change $\beta$; the spectral isomorphisms, which interchange $m_-$ and $m_+$, or 
give one of these a minus sign, and the extraspecial STMs. In this case 
$m_\pm\in \mathbb{Z}\pm 1/4$ and $\beta=2$, while the target Hecke algebra is 
of the form $C_n[\delta_-,\delta_+](q)$ with $\delta_\pm\in\{0,1\}$. The latter STMs 
correspond to Hecke algebras of types of inner forms $G^u$  of $G^*$ of 
even orthogonal or symplectic groups where $u\in\Omega/(1-\theta)\Omega$ 
has maximal order. 

Let us now explain the construction of STMs associated with 
a correspondence of diagrams as in Lusztig in the exceptional cases. 
For more details, see \cite{Opdl}.

First one needs to establish the cuspidal case of Theorem \ref{thm:STMexuni}
in the exceptional case. This is an explicit case by case verification based 
on the classification of the cusipdal unipotent characters of finite groups of Lie type 
\cite{Lusztig-chars}, the classification of the residual points for $\Hc_\Ibb$, and 
the computation of the formal degree (up to constant factors) for discrete series 
representations of $\Hc_\Ibb$ supported by the corresponding central characters  
\cite{HOH}, \cite{OpdSol}. This was carried out in \cite[3.2.2]{Opdl}.

Next we discuss the diagram of an STM 
$\Phi:\Hc_{\tilde{t}}^e\leadsto\Hc_\Ibb(G^*)$.  Consider the underlying map 
$\Phi:T_{\tilde{t}}\to L_n$ with $L\subset T=T_\Ibb=T^\vee/(1-\theta)T^\vee$
a residual coset. Such a map is determined by assigning 
weights to the nodes of the Kac diagram $D(\frak{g}^\vee,\theta)$ \cite{Retorsion} of 
${}^LG=G^\vee\rtimes\langle\theta\rangle$ according to the following steps (i)-(vii).
\\

\noindent (i) Lift $\tilde{\Phi}$ to an affine linear map $\frak{t}_{\tilde{t}}\to\frak{t}=\frak{t}_\Ibb$ 
and use the action of the dual affine Weyl group 
$W_\Ibb^{a,\vee}=Q^\vee\rtimes W_{\mathbb{I},0}$ to replace $\tilde{\Phi}$ by the lift $\tilde{\Phi'}$ 
of a map equivalent to $\Phi$,  such that $P_L:=\textup{Im}(\tilde{\Phi'}_{\vt=1})$
meets the closure $\overline{C^\vee}$ of the dual alcove $C^\vee$ in a facet of dimension 
equal to $\textup{dim}(T_{\tilde{t}})$. We assume from now on that $\Phi=\Phi'$ has this property.
\\

\noindent (ii) We have $\textup{Im}(\Phi)=L_n$ for a unique residual coset $L_\Phi=r_LT^L\subset T_\Ibb$, 
where $r_L=s_Lc_L\in T_L$ is a residual point. Then $P_L$ is a lift of $L_{\Phi,\vt=1}$.
\\

\noindent (iii) Let us denote the spectral diagram \cite[Section 8]{OpdSol}, \cite[2.3]{Opds} of 
$\Hc_{\tilde{t}}^e$ by $\Sigma_t^{spec}$. Let  $(W_t^{a,\vee},S_t^{a,\vee})$ be the associated 
affine Coxeter group.
It turns out that property (3) of Definition \ref{def:STM} implies that there exists an isomorphism of 
affine Coxeter groups $(W_t^{a,\vee},S_t^{a,\vee})\to W_L^*:=N_{W_\Ibb^{a,\vee}}(P_L)/C_{W_\Ibb^{a,\vee}}(P_L)$. 
Here $W_\Ibb^{a,\vee}$ is the affine Weyl group associated to the Kac diagram $D(\frak{g}^\vee,\theta)$ 
acting on $\frak{t}_\Ibb=\frak{t}^\theta$ (see \cite{Retorsion}). 

Indeed, using the theory of intertwiners Harish-Chandra proved that the reflection in a 
hyperplane on which the Plancherel density 
of a generalised principal series vanishes $i_{P=LN}^{G^*}(\delta)$ 
is an element of $N_{W_{\Ibb}}(L,\delta)$. Apply this to the generalised principal series of $G^*$ 
supported by $L_\Phi$, then we see that such hyperplane reflection is in  $N_{W_{\Ibb}}(L_\Phi)$.
On the other hand, by (3) of Definition \ref{def:STM} all reflection hyperplanes of $W_t^{a,\vee}$ 
correspond to such zeroes of the Plancherel density $\mu^{L_\Phi}$ via $\Phi_T$. 
This implies that the simple affine reflections in $P_L$ associated with the faces of $P_L$ 
are restrctions to $P_L$ of elements of the normalizer of $P_L$ in $W_\Ibb^{a,\vee}$. 
\\

\noindent (iv) By (iii), the material in \cite[Section 2]{Lusztig-unirep} applies. We see that 
if $P_L$ corresponds to the set $J\subset I$ (with $I$ an index set for the set of nodes of the 
Kac diagram $D(\frak{g}^\vee,\theta)$), then $J\subset I$ is an \emph{excellent}  
subset, and the set of affine simple reflections of $W^*\simeq  (W_t^{a,\vee},S_t^{a,\vee})$ is in natural 
bijection with $K=I\backslash J$. Hence the set $I$ decomposes in a subset $J$ correspond 
to the affine simple roots constant which are constant on $L$, and $K=I\backslash J$, the 
which is in canonical bijection with the set $\tilde{K}$ of simple reflections of $\Sigma_t^{spec}$
via the isomorphism $(W_t^{a,\vee},S_t^{a,\vee})\to W_L^*$ induced by $\Phi$.
\\

\noindent (v) Assign weights to the nodes of $I$: If $i\in I$ let $a_i^\vee$ be the corresponding 
affine simple Kac root, and let $Da^\vee_i$ be its gradient, viewed as character of 
$T_\Ibb$ (defined over $\Lambda_0$). 

We define $w_i(s):=D(a^\vee_i)(\Phi(s))$, viewed as a function of $s\in T_{\tilde{t},v}$ 
(the real vector group $T_{\tilde{t},v}\subset T_{\tilde{t}}$).  
We note that $w_i$ is independent of $s\in T_{\tilde{t},v}$ if and only if $i\in J$, and that 
$\prod_{i\in I} w_i^{n_i}=1$, where $n_i$ are such that $\sum_{i\in I}n_i a^\vee_i=1$. 

Observe that for all $j\in J$, $w_j=v^{\dt_j}$ for certain $\dt_j\in\mathbb{Z}_{\geq 0}$, 
and these coordinates determine the residue point $r_L\in T_L=T_J$ modulo the 
action of $(\Omega_{J,ad})^*$. 

We identify $\tilde{K}$ and $K$ via the isomorphism of (iv). Suppose that 
$\sum_{k\in K}\tilde{n}_kb^\vee_k=1$ is the affine relation for the affine simple roots of 
$\Sigma_t^{Spec}$. It turns out that $\zt_k:=n_k/\tilde{n}_k\in\mathbb{N}$ for all $k\in\tilde{K}=K$. 
We then have for all $k\in K$ that  
$w_k(s)=\zeta_{k}v^{\ct_k}(Db^\vee_k(s))^{1/\zt_k}$, where $b^\vee_k$ is the affine root of 
$(W_t^{a,\vee},S_t^{a,\vee})$ corresponding to $k$ via $\Phi$. Moreover, if $k_0$ denotes the 
unique affine extension root of $\Sigma_t^{Spec}$ then $\zeta_{k_0}$ is a primitive $n_{k_0}$-th
root of unity, $\zeta_k=1$ if $k\not=k_0$, and $\ct_k\in\mathbb{Z}$ can be computed 
by $Da_k^\vee(c_L)=v^{c_k}$. 
\\

\noindent (vi) We call the diagram $D(\frak{g}^\vee,\theta)$ so obtained, with the weights $w_i$ attached, 
the diagram $D(\Phi)$ of the STM $\Phi$. Note that $\Phi$ is determined by $D(\Phi)$, since 
$\tilde{\Phi}$ is.   
\\

\noindent (vii) Given such a diagram, to check that the map it defines is an STM comes down to checking 
first of all that the constant weights $w_j$ on $J$ define a residue point of $T_J$, and secondly
that (3) of Definition \ref{def:STM} holds. The latter verification is straightforward but cumbersome
(a lot of cancellations are taking place).   
\\

We claim that the diagram $D(\Phi)$ without the weights $w_i$, but with the 
subsets $J$ and $K$ remembered, is the ``geometric diagram'' which Lustig \cite{Lusztig-unirep}, 
\cite{Lusztig-unirep2} attaches to the arithmetic diagram of $\Hc_{\tilde{t}}^e$. 

Conversely, given the ``geometric diagram'' Lusztig attaches to the arithmetic diagram of 
$\Hc_{\tilde{t}}^e$. Now one needs to assign the appropriate weights $w_i$ to the nodes of the 
diagram so that it becomes the diagram of an STM. This is done as follows.
From $(\Pbb,\sigma)$ we obtain a corresponding 
cuspidal pair $(M,\delta)$ with a Levi subgroup $M\subset G$ and a supercuspidal representations 
of unipotent reduction $\delta$ of $M/A_M$ (with $A_M$ the maximal split torus in $Z(M)$). 
Note that $M/A_M$ is again of adjoint type.
Thus by the cuspidal case of Theorem \ref{thm:STMexuni} (which we have established already 
for exceptional groups) we obtain a \emph{unique}  
$X_{wur}(M)$-orbit of central characters $W_Mr_M$ with $r_M\in T_M$ a residue points for 
$M/A_M$ associated to $\delta$ such that (\ref{eq:cuspmu}) holds (for $(M,\delta)$ instead of $(G,\pi)$).  
Let $L=r_MT^M\subset T_\Ibb$ be the corresponding residual coset.
Using $W_\Ibb$ we may assume that $L$ is in the position as in (a),(b).
We observe that $X_{wur}(M)/(T_M\cap T^M)$ is a quotient of $X_{wur}(G)$. 
Thus $(M,\delta)$ determines up to the action of $W_\Ibb$ a 
unique $X_{wur}(G)$-orbit of residual subspaces, and we have a 
representative $L=r_MT^M\subset T_\Ibb$ of this orbit in the position described in (b) above.
We now assign weights $w_i$ to the nodes of $D(\frak{g}^\vee, \theta)$,  
exactly as has been described above in (a)-(g), see \cite[3.2.4; 3.2.5]{Opdl}.
This defines a map $\Phi:T_{\tilde{t}}\to L_n$.  
Finally one needs to verify the property (3) of Definition \ref{def:STM} for this map, 
which amounts to case by case computations \cite{Opdl}.
\\

At this stage we have established the existence of enough STMs, Theorem \ref{thm:STMexuni}(a).

Next we now look at the essential uniqueness statement Theorem \ref{thm:STMexuni}(b)
for $\Phi:\Hc_{\tilde{t}}^e\leadsto\Hc_\Ibb(G^*)$
This reduces to the cuspidal case using STM diagrams as follows:  
By Theorem \cite[Proposition 7.13]{Opds} if suffices to show that 
one can find a $\omega \in X_{wur}(G)$ such that $\Phi'_{\tilde{t}}$ and 
$m_\omega\circ\Phi_{\tilde{t}}$ have the same image (in the sense of 
\cite[Definition 5.10]{Opdl}). Assume that the existence and uniqueness 
property (b) for the cuspidal case has been solved already, and let $\Hc_{\tilde{t}}^e$
be given. Let $(M,\delta)$ be the cuspidal pair associated to $t$ by \cite{Mo}, 
then $\Pbb_M=M\cap\Pbb$ is a maximal parahoric subgroup of $M$, and 
$\sigma_M=\mathbb{P}|_{\Pbb_M}$ is a cuspidal unipotent character for $\Pbb_M$.
By the already established cuspidal case of Theorem \ref{thm:STMexuni}(a),(b), 
equation (\ref{eq:cuspmu})  
determines a unique $X_{wur}(M)$-orbit of $W_M$-orbits $W_Mr_M$ 
of residual points $r_M=s_Mc_M$ for $\Hc_{\Ibb_M}(M^*)$. Let $\tilde{J}$ be the 
maximal proper subset of the 
spectral diagram $D(\frak{m},\theta_M)$ of $\Hc_{\Ibb_M}(M^*)$  
which corresponds to $r_M$, i.e. the set of affine simple roots of 
$D(\frak{m},\theta_M)$ which are trivial on $s_M$.
Then it turns out that in all cases, $\tilde{J}$ fits as excellent subdiagram of 
$D(\frak{g}^\vee,\theta)$, with image $J\subset D(\frak{g}^\vee,\theta)$ 
which is \emph{unique} up to the action of $X_{wur}(G)$. 
This determines $T_L$ and the image $W_0r_L\in T_L\subset T_\Ibb$ of 
$W_Mr_M\in T_{\Ibb_M}$, up to the action of $X_{wur}(G)$.
Hence the image $L=r_LL$ of $\Phi$ is uniquely determined by $t=(\Pbb,\sigma)$
up to the action of $X_{wur}(G)$. By \cite[Proposition 7.13]{Opds} 
this proves the desired uniqueness.

The above reasoning reduces the general uniqueness statement 
Theorem \ref{thm:STMexuni}(b) to the existence and uniqueness 
\ref{thm:STMexuni}(a), (b) for the cuspidal case. This existence and uniqueness
for the cuspidal case was shown in general in \cite{FO} (using the 
general existence statement Theorem \ref{thm:STMexuni}(a)), 
thereby finishing the proof of Theorem \ref{thm:STMexuni}(b). 

The existence and essential uniqueness of a $X_{wur}(G)$-equivariant STM, 
and its compatibility to Lusztig's parameterisation (Theorem \ref{thm:STMexuni}(c), (d)) 
also reduces to the cuspidal case, by using the same construction of ``inducing 
higher rank STMs from cuspidal ones'' as discussed above. 
In the cuspidal case, these statements follow from a case by case verification 
\cite{FO}, \cite{FOS}.
This finishes the discussion of the proof of Theorem \ref{thm:STMexuni}.
\subsection{Main Theorem}
The following Theorem is the main result. It is a generalisation of results 
of \cite{Re}, \cite{Opdl}, \cite{FO}, \cite{FOS}. 
\begin{thm}\label{thm:main}
Let $G=\mbf{G}(F)$ be the group of points of a connected reductive group 
$\mbf{G}$ defined over a non-archimedean local field $F$ which splits over an 
unramified extension of $F$. Let $\hat{G}^{\textup{temp}}_{uni}=
\sqcup_{s\in \mathcal{B}_{uni}(G)}\hat{G}^{\textup{temp}}_s$ denote the set of equivalence 
classes of irreducible tempered representations of $G$ of unipotent reduction. 
Here $\mathcal{B}_{uni}(G)\subset \mathcal{B}(G)$ is the finite subset of unipotent inertial 
equivalence classes, and for each $s\in \mathcal{B}_{uni}(G)$, $\hat{G}^{\textup{temp}}_s\subset \Pi(G)_s$
is the subsets of equivalence classes of tempered irreducible representations in $\mathcal{C}(G)_s$.
For a unipotent type $t=(\Pbb,\sigma)$, $\mathcal{B}^t(G)$ is an $X_{wur}(G)$-orbit 
in $\mathcal{B}_{uni}(G)$. For $s\in\mathcal{B}^t(G)$ choose a corresponding extension 
$\tilde{t}(s)$ of $t$ to $N_G(\Pbb)$, and extend this choice $X_{wur}(G)$-equivariantly. 
Let $\Hc_t$ be the Hecke algebra of the type $t$, then $\Hc_t=\oplus_{s\in \mathcal{B}^t(G)}\Hc_{\tilde{t}(s)}^e$. 
Put $\mathcal{H}_{uni}(G):=\oplus_{s\in\mathcal{B}_{uni}(G)}\Hc_{\tilde{t}(s)}^e$. Let $\Hc_\Ibb=\Hc_\Ibb(G^*)$
be the Iwahori Hecke algebra of the quasi-split group $G^*$ in the inner class of $G$. 
\begin{enumerate}
\item[(a)] There exists a parameterisation 
$\varphi:\hat{G}^{\textup{temp}}_{uni}\to \Phi_{nr}^{\textup{temp}}(G)$,  $\pi\to\varphi_\pi$  
such that:
\begin{enumerate}
\item[(1)] Conjectures \ref{conj:1} and \ref{conj:2} hold, up to constant factors 
independent of $q$.
\item[(2)] For every unipotent type $t$ of $G$ there exists a morphism 
$\Phi_Z^t:\textup{Spec}(Z(\Hc_t))\to \textup{Spec}(Z(\Hc_{\mathbb{I}}))$ such that the map 
$\gamma^\mathbb{I}\circ\varphi: \hat{G}^{t,\textup{temp}}\to S_\mathbb{I}\subset \textup{Spec}(Z(\Hc_{\mathbb{I}}))$
factorises as $\gamma^\mathbb{I}\circ\varphi=\Phi_Z^t\circ \beta^t\circ\textup{cc}$.
\end{enumerate} 
The morphism $\Phi_Z^t$ is associated to an STM $\Phi^t:\Hc_t\leadsto \Hc_{\mathbb{I}}$ satisfying 
the conditions (2)(a) and (b) of Corollary \ref{cor:STMpar}. 
\item[(b)] Such a parameterisation is unique up to automorphisms $\alpha\in\textup{Aut}(\mathcal{H}_{uni})$
such that $\varphi':=\varphi\circ\alpha^*$ is again a parameterisation. 
\item[(c)] If $G$ is of adjoint type then Lusztig's enhanced parameterisation 
\cite{Lusztig-unirep,Lusztig-unirep2} $\tilde{\varphi}_{Lu}:\hat{G}^{\textup{temp}}_{uni}\to\tilde{\Phi}_{nr}^{\textup{temp}}(G)$ 
satisfies the conjectures \ref{conj:1} and \ref{conj:2} of Hiraga, Ichino and Ikeda \cite{HII}, 
and moreover satisfies property (a)(2) above. (For general unramified $G$: see Theorem \ref{thm:Iw}.)
\end{enumerate}
\end{thm}
\begin{proof}
(c) For $G$ simple and of adjoint type, Conjecture \ref{conj:1} for discrete series representations is 
\cite[Theorem 4.11]{Opdl}. The proof is based on Theorem \ref{thm:STMexuni} (to show the validity modulo 
rational constant factors), and the precise computation of the rational constant $d_{\Hc,\delta}$ for the 
discrete series of affine Hecke algebras based on Theorem \ref{thm:rat}.
Obviously this implies the result for all semisimple groups $G$ of adjoint type.
If $G$ is of adjoint type and $M\subset G$ is a Levi subgroup with maximal split central torus 
$A_M\subset Z_M$, then $M/{A_M}$ is also of adjoint type. 
Hence Conjecture \ref{conj:1} also holds for the discrete series of $M$ 
modulo the center and of unipotent reduction. Thus Conjecture \ref{conj:2} follows from the Plancherel formula of 
Harish-Chandra \cite[Th\'eor\`eme VIII.1.1]{Wal}. 
By Theorem \ref{thm:STMexuni} (see also \cite[Theorem 3.4]{Opdl}), Lusztig's parameterisation 
$\varphi_{Lu}: \hat{G}_{uni}^{\textup{temp}}(G)\to \Phi^{\textup{temp}}_{nr}(G)$ 
of the irreducible tempered representations of unipotent reduction gives rise to a 
spectral transfer map $\Phi_{uni}:\Hc_{uni}(G)\leadsto  \Hc_{\mathbb{I}}(G^*)$. 
By \cite[Theorem 6.1]{Opds} and \cite{BHK} this implies that, up to constants independent of $q$, 
the map 
$\gamma^\mathbb{I}\circ \varphi_{Lu}:  
\hat{G}_{uni}^{\textup{temp}}\to S_\mathbb{I}\subset 
\textup{mSpec}(Z(\Hc_{\mathbb{I}}(G^*)_\vt))$ also satisfies (a)(2). 
\\

(a) and (b): For an absolutely semisimple group $G$ of adjoint type, (c) implies the existence of 
a parameterisation $\pi\to\varphi_\pi$ of tempered representations of unipotent 
reduction satisfying (a)(1) and (2). Let us also prove the uniqueness property (b) for 
such $G$. Given a parameterisation satisfying (a)(1) and (2) it follows 
from Definition \ref{def:desi}(iv) of the parameterisation and \cite[Proposition 8.4]{Bo}, that we may 
assume that $G$ is absolutely simple. 
For all unipotent types $t$ of $G$, the morphism $\Phi^t_Z$ defines a spectral transfer map 
$\Phi^t:\Hc_t\leadsto\Hc_\mathbb{I}$ by definition. Indeed, by Definition \ref{def:desi} (i), (iii) and (v) of a 
parameterisation, the morphism $\Phi^t_Z$ comes from an affine morphism of algebraic 
tori $\Phi^t:T^t=X_{nr}(M)\to L'$ where $L'$ is an intermediate quotient $L\to L'\to L_n$ 
of a residual coset $L\subset T_{\mathbb{I},\vt}$ (indeed, the points of $L'$ are in bijection 
with the orbit of $X_{nr}(M)$ twists of $\delta$, which is a finite quotient of $X_{nr}(M)$, 
and $L_n$ corresponds by definition to the central character map for representations of 
$M$ of this type, hence is a quotient of $L'$). By (a)(1), \cite{BHK}, Theorems \ref{thm:ds} and 
\ref{thm:parb}, and \cite[Definition 5.1]{Opds} it follows that $\Phi^t$ defines a spectral transfer
map $\Hc_t\leadsto \Hc_\mathbb{I}$. By Definition \ref{def:desi}(i) and by 
Theorem \ref{thm:STMexuni} or \cite[Theorem 3.4]{Opdl} 
the required 
uniqueness property follows. Hence we have established the existence and uniqueness of
parameterisations satisfying (a)(1) and (2) for $G$ absolutely simple of adjoint type. 
\\

Let us now show (a) and (b) when $G$ is connected reductive with anisotropic center. 
We write 
$\mbf{G}=\mbf{Z}^0\Dc(\mbf{G})$ with $\Dc(\mbf{G})$ the derived subgroup (which is a connected 
semisimple group) and $\mbf{Z}^0$ an anisotropic torus. Then  
the quotient $\mbf{Z}=\mbf{G}/\Dc(\mbf{G})$ is an isogenous quotient of $\mbf{Z}^0$, in  
particular an anisotropic torus. Let $\mbf{G}_{ad}$ be the adjoint quotient of $\mbf{G}$, 
a connected semisimple group of adjoint type. 
Now consider the isogeny $\psi:\mbf{G}\to \mbf{Z}\times \mbf{G}_{ad}$, and the corresponding 
dual isogeny $\psi^\vee:Z^\vee\times G^\vee_{sc}\to G^\vee$. Let $T^\vee$ be a maximal torus of $G^\vee$, 
and $T^\vee_{sc}$ its inverse image in $G^\vee_{sc}$.
The kernel $\psi$ can be expressed as 
$\textup{Hom}(K_G^*,F^\times_{ur})^{\textup{Fr}}$, where $K_G=X^*(T^\vee_{sc})/X^*(T^\vee)$. 
\\

We have a surjection with finite 
kernel ${}^L(Z^\vee\times G^\vee_{sc})\to {}^LG$. Now $\Phi^{\textup{temp}}_{nr}(Z\times G_{ad})
=\Phi^{\textup{temp}}_{nr}(Z)\times \Phi^{\textup{temp}}_{nr}(G_{ad})\simeq \Phi^{\textup{temp}}_{nr}(G_{ad})$, 
since for the anisotropic torus $Z$, $\Phi^{\textup{temp}}_{nr}(Z)=\{[\varphi_{triv}]\}$. 
The corresponding map $\Phi^{\textup{temp}}_{nr}(Z\times G_{ad})\to \Phi^{\textup{temp}}_{nr}(G)$
can be identified, via these isomorphisms, with the natural map 
$\Phi^{\textup{temp}}_{nr}(G_{ad})\to \Phi^{\textup{temp}}_{nr}(G)$. 
Via the isomorphism of Theorem \ref{thm:cc} this map comes from the covering map 
$T_{sc}^\vee/(1-\theta)T^\vee_{sc}\to T^\vee/(1-\theta)T^\vee$ applied to $S_{sc,\mathbb{I},\vt}$. 
Therefore this map is surjective, with fibres that are orbits for the canonical 
image $I_G\subset (\Omega_{sc}^\theta)^*$ of the finite abelian group 
$(K^\theta_G)^*$ in $X_{wur}(G_{ad})=(\Omega_{sc}^\theta)^*$. Then $I_G$ is 
the kernel of the canonical surjection $(\Omega_{sc}^\theta)^*\to(\Omega_G^\theta)^*$, 
and we have an identification 
$\Phi^{\textup{temp}}_{nr}(G)= I_G\backslash \Phi^{\textup{temp}}_{nr}(G_{ad})$.
\\

We propose the following parameterisation $\hat{G}^{\textup{temp}}_{uni}\to \Phi^{\textup{temp}}_{nr}(G)$.
Given an irreducible tempered representation $(\pi,V_\pi)$ of unipotent reduction of $G$. The kernel 
$\textup{Hom}(K_G^*,F^\times_{nr})^{\textup{Fr}}$ of the isogeny $G\to Z\times G_{ad}$ acts trivially 
in $V_\pi$, hence $(\pi,V_\pi)$ descends to a representation of the image of $G$ in $Z\times G_{ad}$. 
Since the image in $Z$ acts trivially on $V_\pi$, we can extend to a representation $\overline{\pi}$ 
of $Z\times \textup{Im}(G)$ (with trivial action of the $Z$-factor) where $\textup{Im}(G)\subset G_{ad}$. 
We \emph{claim} that $\overline{\pi}$ is a summand of the restriction of 
a representation $\tilde{\pi}$ of unipotent reduction of $G_{ad}$, and that this uniquely determines 
the Lusztig-Langlands parameter $\varphi_{\tilde{\pi}}$ of $\tilde{\pi}$ up to the action of 
$I_G$.
We define, in view of the above identification, $\varphi_\pi:=I_G\varphi_{\pi_{ad}}$.
In order that this definition makes sense we need to verify the above claim.
\\

Let $t=(\mathbb{P},\delta)$ be a ``unipotent type'' for the unipotent Bernstein  
component of $G$, where $\mathbb{P}$ is a parahoric subgroup, $\delta$ 
a cuspidal unipotent of $\mathbb{P}$, and $\tilde{\delta}$ an extension of $\delta$ to $N_G\mathbb{P}$.
The Hecke algebra $\Hc_t$ of $t$ is a direct sum of extended affine Hecke algebras $\Hc_{\tilde{t}}^e$.
We have (see \cite{HR}) $\mathbb{P}=G_1\cap N_G\mathbb{P}$, and this group  
is self-normalizing within $G_1$, the kernel of the Kottwitz homomorphism $w_G$.  
The set of Bernstein 
components described by $t$ corresponds bijectively to the set of extensions 
$\tilde{t}=(N_G\mathbb{P},\tilde{\delta}')$ of $\delta$ to $N_G\Pbb$.
This is a torsor for the character group $(\Omega_G^{\mathbb{P},\theta})^*$ of the abelian group 
$\Omega_G^{\mathbb{P},\theta}$, the subgroup of $\Omega_G^{\theta}=G/G_1$ fixing $\mathbb{P}$
under conjugation. The subalgebra $\Hc_t^a$ of the extended affine Hecke algebra $\Hc_{\tilde{t}}^e$ 
of functions with support in $G_1$ 
is the (unextended) underlying affine Hecke algebra, by the arguments of Lusztig  
\cite[1.10--1.20]{Lusztig-unirep}. The Hecke algebra $\Hc_t$ of the type $t$ is isomorphic to 
$\Hc_t^a\#(\Omega_G^{\mathbb{P},\theta})^*$. This algebra can be written as a direct sum of 
the various extended affine Hecke algebras $\Hc_{\tilde{t}}^e$ associated with the extensions 
$\tilde{t}$. 
This is a direct sum of mutually isomorphic extended affine Hecke algebras $\Hc_{\tilde{t}}^e$ 
whose set of summands 
form a torsor for the group of characters of the subgroup $(\Omega_G^{\mathbb{P},\theta})_1$ 
of $\Omega_G^{\mathbb{P},\theta}$ of elements acting trivially on $\Hc_t^a$.  
In any case, one obtains a canonical action of $(\Omega_G^{\mathbb{P},\theta})^*$, 
and thus of $(\Omega_G^{\theta})^*$, on 
the disjoint union of the centers of the summands $\Hc_{\tilde{t}}^e$ of $\Hc_t$. 
\\

There is a 
bijection between the set of parahoric 
subgroups of $G$ and of $Z\times G_{ad}$ defined as follows. 
First choose a maximally $F$-split maximal $F$-torus $T$ of $G$ which splits over $F_{ur}$. 
Let $\mathcal{A}$ be the apartment of $G$ associated with $T$, which can be embedded 
in the apartment of $\mbf{G}(F_{ur})$ as the set of $\textup{Fr}$-invariant elements.
By \cite[Definition 1; Remark 16]{HR}, \cite[Corollary 2.2]{Opdl}, the stabilizer 
of a facet $f$ in $\mathcal{A}$ in $G_1=\mbf{G}(F_{ur})^{\textup{Fr}}$
is also the pointwise fixator of $f$ in $G_1$, since $\mbf{G}(F_{ur})$ is generated 
by $\mbf{G}_{der}(F_{ur})$ and elements acting trivially on $\mathcal{A}$. Hence 
by \cite{HR} it follows that the parahoric subgroup $\mathbb{P}_f$ of $G$ is also equal 
to the normalizer $\mathbb{P}_f=N_{G_1}(\mathbb{P}_f\cap G_{der})$. Indeed, $\subset$ 
follows since $\mathbb{P}_f$ is self-normalizing in $G_1$, while $\supset$ follows since 
the right hand side is stabilizing, hence pointwise fixing $f$.
 By the functoriality of Kottwitz's map $w_G$ we have a 
homomorphism $G_1\to (Z\times G_{ad})_1$. 
The bijection between the parahoric subgroups is defined as follows: Given 
$\mathbb{P}$ of $G$, we have $\mathbb{P}':=N_{(Z\times G_{ad})_1}(\psi(\mathbb{P}))$ is a parahoric 
group of $Z\times G_{ad}$, as follows from the discussion above and the remark that 
$(Z\times G_{ad})_{der}\subset \psi(G)$; conversely 
if $\mathbb{P}'\subset Z\times G_{ad}$ is parahoric then $\mathbb{P}:=\psi^{-1}(\mathbb{P}')\cap G_1$ 
is parahoric in $G$. These maps are inverse to each other. Moreover, the cuspidal unipotent 
representations of $\mathbb{P}$ and $\mathbb{P}'$ correspond bijectively to each other, since 
it is known that the set of cuspidal unipotent characters of a connected reductive group $\mathcal{G}$ 
over a finite field is independent of the type of $\mathcal{G}$ within its isogeny class (cf. \cite{Lusztig-chars}).  
It follows that the type $t_1:=(\mathbb{P},\delta)$ of $G$ corresponds to a unique type 
$\tilde{t}_1=(\mathbb{P}',\delta')$ of $Z\times G_{ad}$, and the affine Hecke algebras  
$\Hc_{t_1}'$ and $\Hc_{\tilde{t}_1}'$ are isomorphic unextended affine Hecke algebras 
via $\psi$, with the same normalization of their respective traces.  
\\

Now let us return to the verification of the above claim. 
Consider an irreducible tempered representation $(\pi,V_\pi)$ 
of $G$ of unipotent reduction, and let $\tilde{t}=(N_G\mathbb{P},\tilde{\delta})$ be an extension of 
a unipotent type $t$ 
associated to the Bernstein component to which $\pi$ belongs. The induction of $\overline{\pi}$ to 
$Z\times G_{ad}$ is a unitary tempered representation of $Z\times G_{ad}$ with finitely many irreducible 
summands $\tilde{\pi}$, and this gives an obvious construction of tempered irreducible 
representations $\tilde{\pi}$ which contain $\overline{\pi}$ when restricted to the image of 
$G$ in $Z\times G_{ad}$. 
By restriction to $\psi:G_1\to (Z\times G_{ad})_1$ it follows that an extension  
$\tilde{t}$ describing the Bernstein component of 
$\tilde{\pi}$ must restrict to $\tilde{t}_1=(\mathbb{P}',\delta')$ as described above. 

Hence $\tilde{t}$ is of the form $\tilde{t}=(N_{Z\times G_{ad}}\mathbb{P}',\tilde{\delta'})$.
As before, the algebra $\Hc_{\tilde{t}_1}$ is a direct sum of isomorphic extended affine Hecke
algebras. Via $\psi$ we can view $\Hc_{t_1}$ as 
a subalgebra of $\Hc_{\tilde{t}_1}$, and by the Bernstein description of the center of 
an affine Hecke algebra it follows that $Z(\Hc_{t_1})\subset Z(\Hc_{\tilde{t}_1})$ is a 
subalgebra of finite index. The kernel of the corresponding surjective homomorphism 
$\textup{Spec}(Z(\Hc_{\tilde{t}_1}))\to \textup{Spec}(Z(\Hc_{t_1}))$  is the image of 
$I_G\subset (\Omega_{sc}^\theta)^*$ under the surjection 
$(\Omega_{sc}^\theta)^*\to(\Omega_{sc}^{\Pbb,\theta})^*$, and  
the natural action by the group $(\Omega_{sc}^{\theta})^*$ on the 
spectrum of $\textup{Spec}Z(\Hc_{\tilde{t}_1})$ corresponds to the 
natural action of $(\Omega_G^{\theta})^*$ on $\textup{Spec}(Z(\Hc_{\tilde{t}_1}))$. 
Clearly $cc(\tilde{\pi})$ belongs to the $I_G$-orbit of central characters 
in the fibre above $cc(\pi)$ under this map.  

Now recall that for the representation $\tilde{\pi}$ of $G_{ad}$ (or $Z\times G_{ad}$) the 
Langlands parameter $\varphi_{\tilde{\pi}}$ is defined by 
$\varphi_{{\tilde{\pi}}}:=(\gamma^\mathbb{I})^{-1}(\Phi_Z(cc^t(m_t(\tilde{\pi}))))$, where 
$\Phi:\Hc_{uni}(G_{ad})\leadsto \Hc_\mathbb{I}(G_{ad}^*)$ is the 
spectral transfer map of \cite[Theorem 3.4]{Opdl}. We just explained above that 
$cc^t(m_t(\tilde{\pi}))$ belongs to a single $I_G$-orbit, and then the  
$X_{wur}(G_{ad}):=(\Omega_{sc}^{\theta})^*$-equivariance of $\Phi$ and of $\gamma^\mathbb{I}$ 
(by Lemma \ref{lem:gamequiv}) 
finally establishes the claim.
\\

Hence for $G$ connected reductive with anisotropic kernel, we have now shown existence of a 
Langlands correspondence $\varphi:\hat{G}^{\textup{temp}}_{uni}\to \Phi_{nr}^{\textup{temp}}(G)$
which obviously satisfies (a)(1) (we reduced this to the case $G_{ad}$ where we know (a)(1) 
already, via spectral transfer maps which preserve Plancherel densities up to constants by \cite{Opds})  
and (a)(2) (since spectral transfer maps yield such morphisms by definition). The uniqueness property 
(b) follows from the case of $G_{ad}$ if we can show that there always exists a lift 
of a $X_{wur}(G)=I_G\backslash X_{wur}(G_{ad})$-equivariant 
parameterisation $\varphi:\hat{G}^{\textup{temp}}_{uni}\to \Phi_{nr}^{\textup{temp}}(G)$ to a 
$X_{wur}(G_{ad})$-equivariant parameterisation  
$\varphi_{ad}:(\hat{G}_{ad})^{\textup{temp}}_{uni}\to \Phi_{nr}^{\textup{temp}}(G_{ad})$. 
For this we need to invoke a stronger form of the uniqueness property of STMs 
$\Phi^t:\Hc_t\to \Hc_{\mathbb{I},\vt}(G_{ad})$ (where $t$ is an (extended) unipotent type 
which represents a Bernstein component of representations of unipotent reduction for an 
inner form of $G_{ad}$) as explained in \cite[3.2.8]{Opdl} and \cite[Proposition 7.13]{Opds}:  
If $\Phi^t_1:\Hc_t\to \Hc_{\mathbb{I},\vt}(G_{ad})$ is another such STM, then there exists 
an isomorphism $\alpha: \Hc_t\to \Hc_t$, $w\in W_{\mathbb{I},0}$, and 
$\omega\in X_{wur}(G_{ad})=(\Omega_{sc}^\theta)^*$ such that 
$\Phi^t_1=\omega\circ w \circ \Phi^t\circ \alpha$. It follows that any 
two matching $I_G$-orbits of connected components under $\varphi$ must also 
correspond under any equivariant correspondence $\tilde\varphi_{ad}$ (which we know exists)
up to the action of $X_{wur}(G_{ad})$. In this way we can compose $\tilde\varphi_{ad}$ with  
an element of $X_{wur}(G_{ad})$ to obtain an $I_G$-equivariant lifting of $\varphi$ defined 
on this $I_G$-orbit of components. Clearly we can do this for any $I_G$-orbit, thus defining 
the desired equivariant lifting $\varphi_{ad}$ of $\varphi$.  
\\

Finally let us consider the general case. Let $G$ be connected reductive and 
let $A$ be the maximal $F$-split torus in the center of $G$. Then $H=G/A$
is the group of $F$-points of the quotient $\mbf{G}/\mbf{A}$, whose center 
is anisotropic. In particular the 
functoriality of the Kottwitz homomorphism implies that $G_1$ maps surjectively 
to  $H_1$. 

Given a tempered irreducible representation $\pi$ of unipotent reduction of $G$, 
choose a $x\in X_{wur}(G)$ such that $x^{-1}\otimes \pi$ is trivial on $A$. Hence 
$x^{-1}\otimes \pi$ descends to a tempered irreducible representation of unipotent 
reduction $\overline{x^{-1}\otimes\pi}$ of $H$. Let 
$\varphi_{\overline{x^{-1}\otimes\pi}}\in \Phi_{nr}^{\textup{temp}}(H)$. Let 
$p_G(\varphi_{\overline{x^{-1}\otimes\pi}})$ in 
$\Phi_{nr}^{\textup{temp}}(G)$ be the image of 
$\varphi_{\overline{x^{-1}\otimes\pi}}$ under the canonical map ${}^LH\to{}^LG$.
Now define $\varphi_\pi:= x.p_G(\varphi_{\overline{x^{-1}\otimes\pi}})\in 
\Phi_{nr}^{\textup{temp}}(G)$. In order that this makes sense, we need to show the 
independence of the choice of $x\in X_{wur}(G)$ such that $x^{-1}\otimes \pi$ is 
trivial on $A$. Suppose that $z\in X_{wur}(G)$ such that $z|_A=1$. 
Then $z$ restricts to $\overline{z}$, a character of $H$. Since $G_1$ surjects onto 
$H_1$, we see that $\overline{z}\in X_{wur}(H)$. Hence if we would replace $x$ 
by $zx$ then we would get $\varphi'_\pi:= (zx).p_G((\varphi_{\overline{(z^{-1}x^{-1})\otimes\pi}}))$.
The equivariance for $X_{wur}(H)$ of the parameterisation of $H$ and of $p_G$ 
(which is obvious) implies that $\varphi'_\pi=\varphi_\pi$. 
Hence our definition is sound, $X_{wur}(G)$-equivariant by construction, and is the 
only possible extension of a given equivariant parameterisation for $H$. The 
properties (a)(1), (a)(2) and (b) easily follow.
\\

This finishes the proof of the existence and essential uniqueness of STM's 
$\Phi_Z^t:\textup{Spec}(Z(\Hc_t))\to \textup{Spec}(Z(\Hc_{\mathbb{I}}))$ 
in the general case for all unipotent types $t$, satisfying the conditions (2)(a) and (b)
of Corollary \ref{cor:STMpar}. We finally need to show that the corresponding map 
$\hat{G}^{\textup{temp}}_{uni}\to \Phi_{nr}^{\textup{temp}}$, $\pi\to\varphi_\pi$  
defines a parameterisation in the sense of Definition \ref{def:desi}. Of course, 
in conditions (iii) and (v) where two such maps play a role we need to 
allow for the fact that the maps are only uniquely determined up to 
certain spectral automorphisms of $\Hc_{uni}$, which means that we may need to 
choose the relevant morphisms $\Phi_Z^t$ judiciously.  

The conditions (i), (ii) and (vi) hold because $\Phi_Z^t$ satisfies conditions (2)(a) and (b)
of Corollary \ref{cor:STMpar}. Furthermore condition (iv) follows easily from 
\cite[Proposition 8.4]{Bo}. 

Let us show that condition (iii) (the compatibility with unitary parabolic induction) 
holds. Let $\pi\in \hat{G}^{\textup{temp}}_{uni}$ be a tempered 
representation of unipotent reduction of $G$, and suppose that $\pi$ is a summand of 
$i_P^G(\pi')$ (unitary induction), for an $F$-parabolic subgroup $P=HU$ with $H$ a standard 
$F$-Levi subgroup of 
$G$, and $\pi'\in \hat{H}^{\textup{temp}}_{uni}$. 
Let $\varphi^H_{\pi'}\in\Phi_{nr}^{\textup{temp}}(H)$ be the parameter for $\pi'$, defined by 
a morphism $\Phi^{t_H}_Z$ as in (a) (thus originating from an 
STM $\Phi^{t_H}:\Hc_{t_H}\leadsto \Hc_{\Ibb}(H^*)$ which satisfies the conditions (2)(a) and (b)
of Corollary \ref{cor:STMpar}) for $H$). Then we want to show that we can find such 
a morphism $\Phi^{t_G}_Z$
as in $(a)$ for $G$, such that it attaches the parameter 
$\varphi^H_{\pi'}\in\Phi_{nr}^{temp}(H)\subset \Phi_{nr}^{temp}(G)$ to $\pi$. 
In order to do so, suppose that $\pi\in \hat{G}^{\textup{temp}}_{t_G}$ 
for some unipotent type $t_G$ for $G$ (such type exists by \cite{Mo},\cite{MP1}). 
Then $t_G$ is a $G$-cover of a cuspidal unipotent type 
$t_M$ for a cuspidal pair $(M,\tau)$ in the cuspidal support of $\pi$ 
(\cite{BK}, \cite{Mo}, \cite{MP2}), 
where we choose $M\subset G$ to be a standard parabolic subgroup.
Since the cuspidal support of $\pi'$ equals that of $\pi$, we see that we may assume 
$M\subset H$, and that we have an intermediate $H$-cover $t_H$ of $t_M$ such that 
$\pi'\in \Pi(H)^{t_H}$. In this situation 
we have (\cite{BK}, \cite{Mo})) an injective homomorphism 
$j:\Hc_{t_H}\to \Hc_{t_G}$ (a ``parabolic subalgebra''  
in the sense of \cite{Opd3})  
inducing a canonical embedding of commutative algebras 
$Z(\Hc_{t_G})\subset Z(\Hc_{t_H})$.  
Similarly we have $\Hc_{\Ibb}(H^*)\subset \Hc_{\Ibb}(G^*)$ (likewise 
a parabolic subalgebra), inducing a canonical embedding 
$Z(\Hc_{\Ibb}(G^*))\subset Z(\Hc_{\Ibb}(H^*))$. The morphisms defined by induction
on the spaces of central characters of these Hecke algebras are given by 
these inclusion maps.

It follows from 
\cite[3.1.1]{Opdl} that $\Phi^{t_H}: \Hc_{t_H}\to\Hc_\Ibb(H^*)$ is induced form a cuspidal STM 
$\Phi'_M:\Hc_{t_M}\to\Hc_\Ibb(M^*)_M$. Similarly $\Phi_{t_G}$ (whose existence 
we have established in the first part of the proof) is induced from a cuspidal STM 
$\Phi_M:\Hc_{t_M}\to\Hc_\Ibb(M^*)_M$. But it is clear that $\Phi^{t_G}$ also 
defines an STM $\Hc_{t_H}\leadsto \Hc_{\Ibb}(H^*)$ satisfying the conditions (2)(a) and (b)
of Corollary \ref{cor:STMpar} for $H$ (defined by the same map as $\Phi^{t_G}$ on $T_{t_H}=T_{t_G}$; 
indeed this clearly defines a spectral preserving map for the parabolic subalgebra 
$\Hc_{t_H}$ if it does so for $\Hc_{t_G}$).
In particular, we may just as well choose $\Phi'_M=\Phi_M$. The resulting commuting diagram 
\[\begin{CD}
\textup{Spec}(Z(\Hc_{t_H}))@>>>\textup{Spec}(Z(\Hc_{t_G}))\\
@V{\Phi^{t_H}_Z}VV                            @VV{\Phi^{t_G}_Z}V\\
\textup{Spec}(Z(\Hc_{\Ibb}(H^*)))@>>>\textup{Spec}(Z(\Hc_{\Ibb}(G^*)))
\end{CD}\]
now easily shows the desired compatibility with unitary induction for 
the associated maps $\varphi^H$ and $\varphi^G$.
\\

The proof of (v) is comparable to that of (iv). Suppose that $\eta:H\to G$ 
is an $F$-morphism with abelian kernel and co-kernel. 
We want to show that if $\varphi\in\Phi_{nr}^{\textup{temp}}(G)$ is given 
and $\pi\in\Pi_\varphi(G)^{\textup{temp}}$, the pull-back of $\pi$ to $H$ 
is a finite direct sum of tempered irreducible representations in 
$\Pi_{{}^L\eta\circ\varphi}^{\textup{temp}}(H)$, 
where ${}^L\eta:{}^LG\to{}^LH$ denotes the natural map.
\\

We choose maximally $F$-split tori $T_H$ and $T_G$ such that 
$\eta(T_H)\subset T_G$.  
The induced map  by $\eta$ on the (absolute) root data of $(H,T_H)$ and $(G,T_G)$ 
is an isomorphism on the root systems, but the lattices may differ. This defines a 
bijective correspondence between the sets of standard $F$-Levi subgroups 
of $H$ and $G$. Let $M_G$ and $M_H$  be matching standard Levi subgroups 
of $G$ and $H$ respectively. Then 
$\eta^M:=\eta|_{M_H}:M_H\to M_G$ also has abelian kernel and co-kernel. 
The pull-back of a cuspidal character of $M_G$ of unipotent reduction defines a 
cuspidal character of unipotent reduction of $M_H$, and conversely, every cuspidal 
character of unipotent reduction of $M_H$ is obtained like this. Suppose that 
we have such matching unipotent inertial equivalence classes $\mathfrak{s}_H$
and $\mathfrak{s}_G$, and let $t_H$ and $t_G$ be unipotent types for $H$ 
and $G$ respectively covering the corresponding Bernstein components.  
Let $\Phi^{t_H}: \Hc_{t_H}\leadsto\Hc_\Ibb(H^*)$ and $\Phi^{t_G}: \Hc_{t_G}\leadsto\Hc_\Ibb(G^*)$
be corresponding STMs which satisfy the conditions (2)(a) and (b)
of Corollary \ref{cor:STMpar} (the existence and essential uniqueness of such STMs 
was proved in the first part of the proof).
By \cite[3.1.1]{Opdl} these STMs are induced from cuspidal unipotent STMs 
$\Phi_{M_H}:\Hc_{t_{M_H}}\to\Hc_\Ibb(M_H^*)_{M_H}$ and 
$\Phi_{M_G}:\Hc_{t_{M_G}}\to\Hc_\Ibb(M_G^*)_{M_H}$ of the 
reductive groups $(M_H)_{ssa}:=M_H/A^M_H$ and 
$(M_G)_{ssa}=M_G/A^M_G$ respectively, 
which are almost direct products of a semisimple and 
an anisotropic groupn (here $A^M_G$ denotes the connected split 
center of $M_G$, and likewise $A^M_H$
is the connected split center of $M_H$). 
Hence the latter Hecke algebras are 
all semisimple. By the equivariance property (2)(a) of 
Corollary \ref{cor:STMpar}, the STMs $\Phi^{t_H}$ and 
$\Phi^{t_G}$ are completely determined by $\Phi_{M_H}$ 
and $\Phi_{M_G}$ respectively. 
Note that $\eta$ induces a spectral isomorphism 
$\Phi^t_\eta:\Hc_{t_{M_H}}\leadsto \Hc_{t_{M_G}}$ 
(since these Hecke algebras have rank zero, if an 
STM between them exists it is an isomorphism)  
and an embedding (in particular an 
STM--see \cite[7.1.3]{Opds}) 
$\Phi^\Ibb_\eta:\Hc(M_H^*)_{M_H}\leadsto \Hc_{t_{M_G}}(M_G^*)_{M_G}$.  

Because the root systems of $G$ and $H$ are isomorphic via $\eta$, 
it is clear that $\Phi'_{M_G}=\Phi^\Ibb_\eta\circ \Phi_{M_H}\circ (\Phi^t_\eta)^{-1}$
also defines an STM $\Hc_{t_{M_G}}\to\Hc_\Ibb(M_G^*)_{M_H}$ that can be induced 
to an STM $(\Phi^{t_G})':\Hc_{t_G}\leadsto\Hc_\Ibb(G^*)$. (Indeed, the $\mu$-functions 
of the $\Hc_\Ibb(G^*)$ and $\Hc_\Ibb(H^*)$ are given by ``the same'' formula.) 
Thus we may and will from now on 
assume that $\Psi^{t_G}=(\Phi^{t_G})'$ is constructed like this.
It follows easily that we have a commutative diagram of morphisms 
\[\begin{CD}
\textup{Spec}(Z(\Hc_{t_G}))@>{(\eta^t_Z)^*}>>\textup{Spec}(Z(\Hc_{t_H}))\\
@V{\Phi^{t_G}_Z}VV                            @VV{\Phi^{t_H}_Z}V\\
\textup{Spec}(Z(\Hc_{\Ibb}(G^*)))@>{(\eta^\Ibb_Z)^*}>>\textup{Spec}(Z(\Hc_{\Ibb}(H^*)))
\end{CD}\]
where $\eta^t_Z$ and $\eta^\Ibb_Z$ denote the natural morphisms 
which $\eta$ induces on the centers of the unipotent affine Hecke algebras 
involved. By comparing the central characters 
of a tempered representation of $\Hc_{t_G}$ and a summand of its pull back 
via $\eta$ the desired property (v) now follows in a straightforward fashion.
\end{proof}
\begin{rem} In \cite[Theorem 1.3]{FOS} the conjecture \ref{conj:1} 
(including the rational constants) 
was proved for 
supercuspidal representations of unipotent reduction of semisimple groups. 
We hope that this can be extended to prove the 
conjectures \ref{conj:1} and \ref{conj:2} of \cite{HII} for all tempered representations of 
unipotent reduction of a general connected reductive group over $F$ (split over 
an unramified extension), using \cite{CO} (in particular Theorem 
\ref{thm:rat}) and drawing on ideas from \cite{Ree4}.
\end{rem}


\begin{thebibliography}{99}

\bibitem[Ar1]{Ar1} Arthur, J.~, 
``Unipotent automorphic representations: conjectures'', 
Asterisque \textbf{171-172} (1989), pp.~13--71. 

\bibitem[Ar2]{Ar2} Arthur, J.~, 
``A note on L-packets'', 
Pure Appl. Math. Q. \textbf{2} (2006), no. 1, 
Special Issue: In honor of John H. Coates. Part 1, pp. 199--217.

\bibitem[ABV]{ABV} Adams,J.~, Barbasch, D.~, and Vogan, D.~, 
``The Langlands Classification and Irreducible Characters for Real Reductive Groups'', 
Progr. Math. \textbf{104}, Birkhauser, Boston, 1992.

\bibitem[AMS]{AMS} Aubert, A-M., Moussaoui, A., Solleveld, M., 
``Affine Hecke algebras for Langlands parameters'',
math.RT:1701.03593 (2017).

\bibitem[Be]{Be} Bernstein, J.N., 
``All reductive p-adic groups are tame'',  
Functional Anal. Appl. \textbf{8} (1974), 91--93.

\bibitem[BeDe]{BeDe} Bernstein, J.N., Deligne, P., 
``Le ÓcentreÓ de Bernstein'',
In ÓRepresentations des groups redutifs sur un corps local,
Traveaux en coursÓ (P. Deligne ed.), Hermann, Paris, 1-32 (1984).

\bibitem[Be1]{Be1} Bernstein, J.N., 
``On the support of Plancherel measure'', 
Journal of Geometry and Physics \textbf{5}, Issue 4 (1988), p. 663-710.

\bibitem[Bo]{Bo} Borel, A.~,
``Automorphic $L$-functions", Proceedings of Symposia in Pure Mathematics,
Vol.~\textbf{33} (1979), part 2, pp. 27--61.

\bibitem[Bo1]{Bo1} Borel, A.~,
``Admissible representations of a semisimple group over a local field with vectors
fixed under an Iwahori subgroup'', Inventiones Math., \textup{35} (1976), pp. 233--259.

\bibitem[BT]{BT} Bruhat, F.~, Tits, J.~,
``Groupes r\'{e}ductifs sur un corps local. II. Sch\'{e}mas en groupes. Existence d'une donn\'{e}e radicielle valu\'{e}e", Inst. Hautes Etudes Sci. Publ. Math., No.~\textbf{60} (1984), pp.~197-376.

\bibitem[BK]{BK} C.J.~Bushnell, P.C.~Kutzko, 
``Smooth representations of $p$-adic groups: structure theory via types'', Proceedings of the London Mathematical Society, \textbf{vol} 77, issue 3, Nov.~1998, pp.~582--634.

\bibitem[BKH]{BHK} Bushnell, C.J.~, Henniart, G.~, Kutzko, P.C.~,
``Types and Explicit Plancherel Formulae for Reductive $p$-adic groups'', in \textit{On Certain L-Functions}, 55--80, Clay Math.~Proc.~, \textbf{13}, Amer.~Math.~Soc.~, Providence, RI, 2011.

\bibitem[Car]{Carter} Carter, Roger W.~,
``\textsc{Finite groups of Lie type}: Conjugacy classes and complex characters"
John Wiley and Sons, 1985.

\bibitem[CK]{CK} Ciubotaru, D.~, Kato, S.~,
``Tempered modules in exotic Deligne-Langlands correspondence'', 
Advances in Mathematics, \textbf{226} (2011) pp.~1538--1590.

\bibitem[COT]{COT}
Ciubotaru, D., Opdam, E.M., Trapa, P.,
``Algebraic and analytic Dirac induction for graded affine Hecke algebras'',
{\it J. Inst. Math. Jussieu} {\bf 13} (2014), no. 3, 447--486.

\bibitem[CO]{CO} Ciubotaru, D., Opdam, E.M., 
``A uniform classification of the discrete series representations 
of affine Hecke algebras",  
Algebra and Number Theory, \textbf{11}(5)  (2017), 
pp.~1089--1134. 



\bibitem[DeRe]{DeRe} DeBacker, S.~, Reeder, M.~,
``Depth-zero supercuspidal $L$-packets and theri stability''
Annals of Mathematics, \textbf{169} (2009), pp.~795--901.

\bibitem[DeLu]{DL} Deligne, P.~, Lusztig, G.~, 
``Representations of reductive groups over finite fields.'', 
Ann. of Math., Vol. \textbf{103}, No. 1, 103 (1) (1976), 103--161. 

\bibitem[Fe1]{Fe1} Feng, Y.~, ``On Cuspidal Unipotent Representations'', Thesis 2015, University of Amsterdam.

\bibitem[Fe2]{Fe2} Feng, Y.~, ``Notes on Spectral Transfer Maps for Affine Hecke Algebras'', in preparation.

\bibitem[FO]{FO} Feng, Y., Opdam, E.M., 
``On a uniqueness property of supercuspidal unipotent representations'', 
arXiv:1504.03458[math.RT].

\bibitem[FOS]{FOS} Feng, Y., Opdam, E.M., Solleveld, M., 
``Supercuspidal unipotent representations and L-packets'', 
arXiv:1805.01888.

\bibitem[GK]{GK} Gelbart, S.~, Knapp, A.~, 
``L-indistinguishability and R groups for the special linear group",
Adv. in Math. \textbf{43} (1982), 101--121

\bibitem[G]{G} Gross, B.~,
``On the motive of a reductive group'', 
Invent. math. 130, (1997) pp.~ 287--313. 

\bibitem[GG]{GG} Gan, W. T.~, Gross, B.~, 
``Haar measure and the Artin conductor'', 
Trans. Amer. Math. Soc. \textbf{351} (1999), pp.~1691--1704.

\bibitem[GR]{GR} Gross, B.~, Reeder, M.~, 
``Arithmetic invariants of discrete Langlands parameters",
 Duke Math.~J.~ \textbf{154} (2010), no. 3, pp.~431--508.

\bibitem[PR]{HR} Pappas, G.~, Rapoport, M.~,
``Twisted loop groups and their affine flag varieties'', with an appendix by Haines, T.~, Rapoport, M.~, Adv.~Math.~\textbf{219} (2008), no.~1, pp.~118--198.

\bibitem[HC1]{HC} Harish-Chandra, 
``Harmonic analysis on real reductive groups I the theory of the constant term''
Journal of Functional Analysis \textbf{19} (1975) pp.~104--204. 

\bibitem[HC2]{HC1} Harish-Chandra, 
``Representations of a semisimple Lie group on a Banach space I'', 
Trans. Amer. Math. Soc. \textbf{75} (1953), 185--243.

\bibitem[HC3]{HC2} Harish-Chandra, 
Harmonic analysis on real reductive groups. III. 
``The Maass-Selberg relations and the Plancherel formula'', 
Ann. of Math. \textup{104} (1976), pp.~117--201.

\bibitem[HT]{HT} Harris, M., Taylor, R. , 
``The geometry and cohomology of some simple Shimura varieties'', 
Annals of Mathematics Studies \textbf{151}, Princeton University Press (2001).

\bibitem[HO1]{HOH} Heckman, G.J.,  Opdam, E.M.,
{\it{Harmonic analysis for affine Hecke algebras}},
Current Developments in Mathematics (S.-T. Yau, editor),
1996, Intern. Press, Boston.

\bibitem[HO]{HO} Heckman, G.J.~, Opdam, E.M.~,
``Yang's system of particles and Hecke algebras'',
Annals of Math.~\textbf{145}, 1997, pp.~139--173.

\bibitem[H1]{H} Henniart, G.,  ``Une preuve simple des conjectures de Langlands pour GL(n) sur un corps p-adique'', 
Inventiones Mathematicae, \textbf{139} (2000), no. 2, pp.~ 439--455.

\bibitem[H2]{H1} Henniart, G., 
``Caract\'erisation de la correspondance de Langlands locale par les
facteurs $\epsilon$ de paires",
Inv. Math. \textbf{113} (1993), 339--350

\bibitem[H3]{H2} Henniart, G., 
``Une caract\'erisation de la correspondance de Langlands locale pour
GL(n)",
Bull. Soc. math. France \textbf{130} (2002), no. 4, 587--602

\bibitem[HII]{HII} 
Hiraga, K.~, Ichino, A.~, Ikeda, T.~, 
``Formal degrees and adjoint gamma factors", 
J. Amer. Math. Soc.~\textbf{21} (2008), no. 1, pp.~283--304.

\bibitem[HIIcor]{HIIcor} 
Hiraga, K.~, Ichino, A.~, Ikeda, T.~, 
``Correction to: Formal degrees and adjoint gamma factors", 
J. Amer. Math. Soc.~\textbf{21} (2008), no. 4, pp.~1211--1213.

\bibitem[HS]{HS}
Hiraga, K.~, Saito, H.~, 
``On L-Packets for Inner Forms of $\textup{SL}_n$'', 
Memoirs of the AMS \textbf{215}, no 1013, AMS, Providence RI, USA (2012). 

\bibitem[Kac]{Kacbook}
Kac, V.G.~,
\textit{Infinite dimensional Lie algebras}, Third Edition, Cambridge University Press, (1994). 

\bibitem[Kal]{Kal} Kaletha, T.~, 
``The local Langlands conjectures for non-quasi-split groups'',
In: Families of Automorphic Forms and the Trace Formula, 
Simons Symposia, Springer (2016), pp. 217--257.

\bibitem[L]{L}  Langlands, R. P., 
``On the functional equations satisfied by Eisenstein series'', 
Lecture Notes in Mathematics \textup{544}, Springer-Verlag, 1976.

\bibitem[KL]{KLDL}
Kazhdan, D.~, Lusztig, G.~,
``Proof of the Deligne-Langlands conjecture for Hecke algebras",
Inventiones mathematicae Volume \textbf{87}, issue 1, (1987), pp.~153--215.

\bibitem[Kot]{Ko0} Kottwitz, R.E.~, 
``Stable trace formula: elliptic singular terms'', 
Math. Ann.~, \textbf{275} (1986), pp.~365--399.

\bibitem[Kot1]{Ko} Kottwitz, R.E.~,
``Isocrystals with additional structure. II'',
Compositio Mathematica \textbf{109} (1997), pp.~255--339.

\bibitem[Lus1]{Lusztig-chars} Lusztig, G.~, 
``Characters of reductive groups over a finite field", 
Annals of Mathematics Studies, vol.~\textbf{107}, Princeton University Press, Princeton, NJ, 1984.

\bibitem[Lus2]{Lusztig-IC} Lusztig, G.~,
``Intersection cohomology complexes on a reductive group'',
Invent.~math.~ \textbf{75} (1984), pp.~205--272.

\bibitem[Lus3]{Lus3} Lusztig, G.,
``Affine Hecke algebras and their graded version'',
\emph{J. Amer. Math. Soc} \textbf{2} (1989), 599--635.

\bibitem[Lus4]{Lusztig-unirep} Lusztig, G.~
``Classification of Unipotent Representations of Simple $p$-adic groups",
Internat. Math. Res. Notices \textbf{11} (1995), pp.~517--589.

\bibitem[Lus5]{Lusztig-unirep2} Lusztig, G.~, 
``Classification of unipotent representations of simple $p$-adic groups. II'', 
Represent. Theory, \textbf{6} (2002), pp.~243--289.

\bibitem[Mo]{Mo} Morris, L.~,
``Level zero $G$-types'',
Compositio Math. \textbf{118}(2) (1999), pp.~135--157.

\bibitem[MP1]{MP1} Moy, A.~, Prasad, G.~, 
``Unrefined minimal $K$-types for $p$-adic groups'', 
Invent. Math. \textbf{116} (1994), no. 1-3, pp.~393--408.

\bibitem[MP2]{MP2} Moy, A.~, Prasad, G.~,
``Jacquet functors and unrefined minimal K-types",
Comment.~Math.~Helvetici \textbf{71} (1996), pp.~98--121.

\bibitem[Opd1]{Opd0} Opdam, E.M.,
{\it{A generating function for the trace of the
Iwahori-Hecke algebra}},
in {\it{Studies in memory of Issai Schur}},
Joseph, Melnikov and Rentschler (eds.), Progress in Math.,
Birkhauser (2002), 301--324.

\bibitem[Opd2]{Opd3} Opdam, E.M.~,
``On the spectral decompostion of affine Hecke algebras'',
Journal of the Inst.~of Math.~Jussieu (2004) \textbf{3}(4), pp.~531--648.

\bibitem[Opd3]{Opd4} Opdam, E.M.,
``The central support of the Plancherel measure of an 
affine Hecke algebra'',
\emph{Moscow Mathematical Journal} {\bf 7}(4) (2007), 723--741.

\bibitem[Opd4]{Opds} Opdam, E.M.~,
``Spectral correspondences for affine Hecke algebras",
Advances in Mathematics \textbf{286} (2016), pp.~912--957.

\bibitem[Opd5]{Opdl} Opdam, E.M.~,
``Spectral transfer morphisms for unipotent affine Hecke algebras",
To appear in Selecta Mathematica (2016), and  
math.RT.1310.7790 (2013).

\bibitem[OS1]{OpdSol1} Opdam E.M., and Solleveld, M.S.,
``Homological algebra for affine Hecke algebras'',
 {\it Adv. Math.} {\bf 220} (2009), 1549--1601.

\bibitem[OS2]{OpdSol} Opdam, E.M.~, Solleveld M.~,
``Discrete series characters for affine Hecke algebras and their formal degrees",
Acta Mathematica, September 2010, Volume \textbf{205}, Issue 1, pp.~105--187.


\bibitem[PR]{HR} Pappas, G., Rapoport, M.,
``Twisted loop groups and their affine flag varieties''. 
With an appendix by Haines, T.~, Rapoport, M.~, \emph{Adv.~Math.} \textbf{219} (2008), no.~1, 118--198

\bibitem[R1]{Re0} M. Reeder,
``On the Iwahori spherical discrete series of $p$-adic
Chevalley groups; formal degrees and $L$-packets", 
Ann. Sci. Ec. Norm. Sup. \textbf{27} (1994), pp.~463--491.

\bibitem[R2]{Re1} M. Reeder, 
``Whittaker models and unipotent representations of p-adic groups'', 
Math. Ann. \textbf{308} (1997), pp.~587--592.

\bibitem[R3]{Re} Reeder, M.,
``Formal degrees and L-packets of unipotent
discrete series of exceptional $p$-adic groups'',
with an appendix by Frank L\"ubeck,
J. reine angew. Math. \textbf{520} (2000), pp.~37--93.

\bibitem[R4]{Retorsion} Reeder, M.~,
``Torsion automorphisms of simple Lie algebras", 
L'Enseignement Mathematique (2), 56, (2010), pp.~3--47.

\bibitem[R5]{Ree4} Reeder, M.,
``Isogenies of Hecke algebras and a Langlands correspondence for 
ramified principal series representations", {\it Represent. Theory} {\bf 6} (2002), 101--126.


\bibitem[Sha]{Sha} Shahidi, F.,
``A Proof of Langlands' Conjecture on Plancherel Measures; 
Complementary Series of p-adic groups'', 
Annals of Mathematics
Second Series, Vol. \textbf{132}, No. 2 (Sep., 1990), pp. 273--330. 

\bibitem[S]{S} Shelstad, D.~,
``Characters and inner forms of a quasi-split group over $\mathbb{R}$'', 
Comp. Math., \textbf{39}, no 1 (1979), pp 11--45. 

\bibitem[SZ]{SZ} Silberger, A. J. and Zink, E.-W., 
``The formal degree of discrete series representations of central 
simple algebras over p-adic fields'', 
Max-Planck-Institut f\"ur Mathematik, 1996.

\bibitem[Slo]{Slooten} Slooten, K.~,
``Generalized Springer correspondence and Green functions for type B/C
graded Hecke algebras'',
Advances in Mathematics \textbf{203} (2006), pp.~34--108.

\bibitem[Spr]{Spr} Springer, T.A.~,
\textit{Linear Algebraic Groups}, 
Progress in Mathematics \textbf{9}, Birkh\"{a}user (1998). 

\bibitem[Tate]{Tate} Tate, J.~
``Number theoretic background",
Proceedings of Symposia in Pure Mathematics, Vol. \textbf{33} (1979), part 2, pp.~3--26.

\bibitem[Tits]{Tits} Tits, J.~
``Reductive groups over local fields",
Proceedings of Symposia in Pure Mathematics, Vol. \textbf{33} (1979), part 1, pp.~29--69.

\bibitem[Vog]{Vogan} Vogan, D.~
``The local Langlands conjecture'',
Representation theory of groups and algebras, Contemp. Math. \textbf{145}, AMS, Providence, RI, 1993, pp.~305--379.

\bibitem[Wal]{Wal} Waldspurger, J.-L.~
``Repr\'{e}sentations de r\'{e}duction unipotente pour $SO(2N+1)$: quelques cons\'equnces d'un article de Lusztig'', 
Contributions to Automorphic Forms, Geometry and Number Theory, Johns Hopkins Univ.~Press, Baltimore MD, 2004, pp.~803--909.

\bibitem[Z]{Z} Zink, E.-W., 
``Comparison of $GL_N$ and division algebra representations II'', 
Max-Planck- Institut f\"ur Mathematik, 1993.

\end{thebibliography}
\end{document}